\documentclass[a4paper, 11pt]{article}

\usepackage{a4wide}
\usepackage[english]{babel}
\usepackage{amsmath,caption}
\usepackage{psfrag}
\usepackage[latin1]{inputenc}
\usepackage[T1]{fontenc}
\usepackage{amsthm} 
\usepackage{amssymb,amsxtra}
\usepackage[usenames]{color}
\usepackage{stmaryrd}
\usepackage{mathrsfs}

\newcommand{\QQ}{\mathbb{Q}}
\newcommand{\NN}{\mathbb{N}}
\newcommand{\ZZ}{\mathbb{Z}}
\newcommand{\FF}{\mathbb{F}}

\DeclareMathOperator{\id}{id}
\DeclareMathOperator{\rk}{rk}
\DeclareMathOperator{\sgn}{sgn}

\DeclareMathOperator{\ind}{Ind}

\DeclareMathOperator{\Rad}{Rad}
\DeclareMathOperator{\Hd}{Hd}

\DeclareMathOperator{\Hom}{Hom}
\DeclareMathOperator{\End}{End}
\DeclareMathOperator{\Mat}{Mat}

\renewcommand{\leq}{\leqslant}
\renewcommand{\geq}{\geqslant}

\renewcommand{\b}{\mathbf}

\begin{document}

\swapnumbers
\theoremstyle{definition}
\newtheorem{defi}{Definition}[section]
\newtheorem{rem}[defi]{Remark}
\newtheorem{ques}[defi]{Question}
\newtheorem{expl}[defi]{Example}
\newtheorem{conj}[defi]{Conjecture}
\newtheorem{claim}[defi]{Claim}
\newtheorem{nota}[defi]{Notation}
\newtheorem{noth}[defi]{}

\theoremstyle{plain}
\newtheorem{prop}[defi]{Proposition}
\newtheorem{lemma}[defi]{Lemma}
\newtheorem{cor}[defi]{Corollary}
\newtheorem{thm}[defi]{Theorem}

\renewcommand{\proofname}{\textsl{\textbf{Proof}}}

\begin{center}
{\bf\Large On Integral Forms of Specht Modules \\ Labelled by Hook Partitions}

Susanne Danz and Tommy Hofmann
\medskip

\today

\end{center}

\begin{abstract}
\noindent
We investigate integral forms of simple modules of symmetric groups over fields of characteristic $0$ labelled by hook partitions. 
Building on work of Plesken and Craig, for every odd prime $p$, we give a set of representatives of the isomorphism classes of $\ZZ_p$-forms
of the simple $\QQ_p \mathfrak{S}_n$-module labelled by the partition $(n-k,1^k)$, where $n\in\NN$ and $0\leq k\leq n-1$. 
We also settle the analogous question for $p=2$, assuming that $n\not\equiv 0\pmod{4}$ and $k\in\{2,n-3\}$.
As a consequence this  leads to a set of representatives of the isomorphism classes of $\ZZ$-forms 
of the simple $\QQ\mathfrak{S}_n$-modules labelled by $(n-2,1^2)$ and $(3,1^{n-3})$, again assuming $n\not\equiv 0\pmod{4}$.

\smallskip

\noindent
{\bf Keywords:} integral representation, integral form, Jordan--Zassenhaus, symmetric group, Specht module, hook partition

\smallskip
\noindent
{\bf MR Subject Classification:} 20C10, 20C11, 20C30, 20C20
\end{abstract}

\section{Introduction}\label{sec intro}

Suppose that $R$ is a principal ideal domain and $K$ its field of fractions, and let $G$ be  a
finite group. As is well known, every finitely generated $KG$-module $V$ admits an $R$-form, that is,
a finitely generated $RG$-module $M$ that is $R$-free of finite rank and satisfies $V\cong K\otimes_R M$. In general,
$R$-forms of $V$ are far from being unique. However,  under suitable conditions on $R$ and $K$, the Jordan--Zassenhaus Theorem ensures that
there are only finitely many $RG$-isomorphism classes of $R$-forms of $V$. 
By \cite[Theorem (24.1), Theorem (24.7)]{Curtis1981} this holds, in particular, if $K$ is a global field, or if $R$ is a complete discrete valuation
ring and $K$ is a local field.
%but we shall  
%content ourselves with $R$ being a principal ideal domain and $K$ being a global field for the time being. 
A finitely generated $R$-free $RG$-module of finite $R$-rank will be called an $RG$-lattice throughout.

In light of the Jordan--Zassenhaus Theorem one is immediately led to asking for the precise number
of isomorphism classes of $R$-forms of $V$, and possibly concrete representatives of these isomorphism classes. 
In such generality this is of course a completely hopeless task. So one might impose restrictions on the $KG$-modules under consideration, starting
with simple $KG$-modules. It turns out that even then not too much is known when it comes to
determining all isomorphism classes of $R$-forms of $V$ or their number. 

A list of some known results in this direction can be found in~\cite[\S 34]{Curtis1981} and~\cite{Reiner1970}.
Moreover, if $\mathfrak S_n$ is the symmetric group of degree $n \in \NN$ and $V$ is
the natural simple $\QQ \mathfrak S_n$-module, then the isomorphism
classes of $\ZZ$-forms of $V$ have been determined independently by
Craig~\cite{Craig1976} and Plesken~\cite{Plesken1974, Plesken1977}.
This was generalized by Feit first to the reflection representation of Weyl groups of indecomposable root systems, see~\cite{Feit1998}, and later to the natural representation of some complex reflection groups, see~\cite{Feit2003}.

In this article we are concerned with the case where $G$ is the symmetric group $\mathfrak{S}_n$ of degree $n\in \NN$,
and $R$ is the ring of integers $\ZZ$ or its $p$-adic completion $\ZZ_p$, for some prime number $p$.
The simple $\QQ\mathfrak{S}_n$-modules have been well studied for more than a century, and are known as {\sl Specht modules}.
Their isomorphism classes are in bijection with the partitions of $n$, and the Specht $\QQ\mathfrak{S}_n$-module
labelled by a partition $\lambda$ will be denoted by $S^\lambda_\QQ$. Specht modules have a number of remarkable properties: 
for instance, they are absolutely simple, self-dual, and each $\QQ\mathfrak{S}_n$-module
 $S^\lambda_\QQ$ already comes with a distinguished $\ZZ$-form, which we shall denote by $S^\lambda_\ZZ$ and
 whose definition will be recalled in \ref{noth Young Specht}. 
 
 The aim of this article now is to investigate the $\ZZ$-forms of the Specht $\QQ\mathfrak{S}_n$-modules
 labelled by {\sl hook partitions}, that is, partitions of shape $(n-k,1^k)$, for $k\in\{0,\ldots,n-1\}$.
 Since $S^{(n)}_\QQ$ is just the trivial $\QQ\mathfrak{S}_n$-module and $S^{(1^n)}_\QQ$ is the one-dimensional
 module affording the sign representation, each of these clearly has only one $\ZZ$-form up to isomorphism.
 The modules $S^{(n-1,1)}_\QQ$ and $S^{(2,1^{n-2})}_\QQ$ are precisely those dealt with by Plesken and Craig mentioned above;
 the number of isomorphism classes of $\ZZ$-forms of each of them is the number of positive divisors of $n$, and
 both Plesken and Craig give explicit representatives. 
 
 Starting from Plesken's and Craig's results, we shall thus focus on the case where $k\geq 2$.
 Our overall strategy is as follows: first, in order to determine the isomorphism classes of $\ZZ$-forms
 of $S^{(n-k,1^k)}_\QQ$, it suffices to determine the isomorphim classes of $\ZZ_p$-forms of
 the $p$-adic completion $S^{(n-k,1^k)}_{\QQ_p}:=\QQ_p\otimes_\QQ S^{(n-k,1^k)}_\QQ$, for
 every prime number $p$; we shall explain this in more detail in Section~\ref{sec forms}.
 Given $p$, every $\ZZ_p$-form of $S^{(n-k,1^k)}_{\QQ_p}$ is isomorphic
 to a full-rank sublattice of any fixed $\ZZ_p$-form of $S^{(n-k,1^k)}_{\QQ_p}$; see Remark~\ref{rem forms all in one}.
 
 Thus, for every prime number $p$, it then suffices to determine the full-rank $\ZZ_p\mathfrak{S}_n$-sublattices of the 
 given $\ZZ_p$-form $S^{(n-k,1^k)}_{\ZZ_p}:=\ZZ_p\otimes_\ZZ S^{(n-k,1^k)}_\ZZ$
 of $S^{(n-k,1^k)}_{\QQ_p}$, up to isomorphism. 
 It turns out that the cases $p=2$ and $p\geq 3$ behave completely differently, the main issue being
 the poorly understood behaviour of the Specht lattices $S^{(n-k,1^k)}_{\ZZ_2}$ after $2$-modular reduction. While
 for odd $p$ we are able to give representatives of the isomorphism classes of $\ZZ_p$-forms
 of $S^{(n-k,1^k)}_{\QQ_p}$ for all $k\in\{2,\ldots,n-2\}$, for $p=2$ we only get partial information
 and only in the case where $k\in\{2,n-3\}$.  As the first main result of this paper we obtain the following theorem, which will
 be proved in Section~\ref{sec proofs}. Here, for a prime number $p$ and a natural number $n \in \NN$, we denote by $\nu_p(n)$ the $p$-adic valuation of $n$, that is, $\nu_p(n) = \max \{ k \in \NN_0 : p^k \mid n \}$.
 
\begin{thm}\label{thm intro}
Let $n \in \NN$ be such that $n\geq 3$, and let $k \in \{1,\dotsc,n-2\}$. Let $h_p(k)$ be the number of isomorphism classes of $\ZZ_p$-forms of $S_{\QQ_p}^{(n-k, 1^k)}$.

\smallskip

{\rm (a)}\,  If $p$ is odd, then $h_p(k) = \nu_p(n) + 1$.

\smallskip

{\rm (b)}\,  Assume that $p=2$, $n\geq 5$, $n \not\equiv 0 \pmod 4$ and $k\in\{2,n-3\}$.

\smallskip

\quad {\rm (i)} If $n \geq 5$ is odd, then $h_2(k) = 3$.

\quad {\rm (ii)} If $n \equiv 2 \pmod 4$, then $h_2(k) = 4$.
\end{thm}

We would like to stress that the statement of
part (a) of Theorem~\ref{thm intro} can already be found implicitly in work of Plesken \cite[Theorem~(VI.2)]{Plesken1983}.
The arguments given there are, however, not too obvious.
We shall give an elementary and detailed proof of part (a) in Section~\ref{sec p odd}.
In fact, it will follow from the results of Plesken and Craig concerning the case $k=1$, 
and our investigation of the $\ZZ_p\mathfrak{S}_n$-sublattices of $S^{(n-k,1^k)}_{\ZZ_p}$ in Section~\ref{sec p odd}.
In Theorem~\ref{thm forms bijection p} we shall prove that, for $p\geq 3$ and $k\in\{1,\ldots,n-2\}$, there is a bijection between the isomorphism
classes of $\ZZ_p$-froms of $S^{(n-1,1)}_{\QQ_p}$ and those of $S^{(n-k,1^k)}_{\QQ_p}$. The key step here is Theorem~\ref{thm exterior forms iso} on exterior powers, which should also be of independent interest.

Part (b) of Theorem~\ref{thm intro} will be a consequence of a very careful analysis
of the structure of the $\ZZ_2\mathfrak{S}_n$-lattice $S^{(n-2,1^2)}_{\QQ_2}$ in Section~\ref{sec p 2}, the main results there being
Theorem~\ref{thm n odd} and Theorem~\ref{thm n 2 mod 4}. 
Unfortunately, at present, we are not able to settle the case where $p=2$ and $n\equiv 0\pmod{4}$. However, based on computational data we shall
state a conjecture at the end of Section~\ref{sec p 2}. 

As well, the case $k\in\{3,\ldots,n-4\}$ and $p=2$ remains open so far,  due to a lack of knowledge of the structure of 
the $\ZZ_2\mathfrak{S}_n$-lattices $S^{(n-k,1^k)}_{\ZZ_2}$ and their $2$-modular reductions. 
Conjecture~\ref{conj}(b) concerns the number of isomorphism classes of $\ZZ_2$-forms of $S^{(n-3,1^3)}_{\QQ_2}$
and $S^{(4,1^{n-4})}_{\QQ_2}$, and
is also based on computer calculations.

\smallskip

As an immediate consequence of Theorem~\ref{thm intro} and Corollary~\ref{cor number local}, we 
get

\begin{cor}\label{cor intro}
Let $n \in \NN$ be such that $n\geq 4$ and $n \not\equiv 0 \pmod 4$. Let $k\in\{2,n-3\}$ and denote by $j(k)$ the number of isomorphism classes of $\ZZ$-forms of $S_{\QQ}^{(n-k, 1^k)}$, and by $d(n)$ the number of divisors of $n$ in $\NN$.

\smallskip

{\rm (a)}\, If $n \geq 5$ is odd, then $j(k) = 3 d(n)$.

\smallskip

{\rm (b)}\,  If $n \equiv 2 \pmod 4$, then $j(k) = 2 d(n)$.
\end{cor}

In fact, our proof of Theorem~\ref{thm intro} will not only reveal the number
of isomorphism classes of $\ZZ_p$-forms of the $\QQ_p\mathfrak{S}_n$-modules in question, but
will provide explicit representatives of their isomorphism classes. Hence we have the following, a
more precise statement being the content of Theorem~\ref{thm expl}.

\begin{thm}\label{thm intro expl}
Let $n \in \NN$ be such that $n\geq 5$ and $n \not\equiv 0 \pmod 4$. For $k\in\{2,n-3\}$, we can explicitly construct representatives of the isomorphism classes of $\ZZ$-forms of $S_{\QQ}^{(n - k, 1^k)}$.
\end{thm}

\medskip

The present paper is organized as follows: in Section~\ref{sec general}, we summarize the
necessary background on $RG$-lattices and $KG$-modules that will be used throughout. In Section~\ref{sec forms}
we explain how the $\ZZ$-forms of an absolutely simple $\QQ G$-module $V$ are related to
the $\ZZ_p$-forms of its $p$-adic completion; specifically Proposition~\ref{prop rep iso} will
be of great importance for our proof of Theorem~\ref{thm intro expl}.

Since the Specht module $S^{(n-k,1^k)}_{\QQ}$, for $k\in\{1,\ldots,n-2\}$, arises as the $k$th exterior power of $S^{(n-1,1)}_\QQ$, 
in Section~\ref{sec exterior} we give a brief overview of exterior powers of $KG$-modules and $RG$-lattices.
In Section~\ref{sec specht} we introduce Specht modules and Specht lattices, and recall the connection of Specht modules
labelled by hook partitions with exterior powers just mentioned. Sections~\ref{sec p odd} and \ref{sec p 2}
are then devoted to analyzing the $\ZZ_p\mathfrak{S}_n$-lattices $S^{(n-k,1^k)}_{\ZZ_p}$ in the case where
$p\geq 3$ and $p=2$, respectively. The main results of these two sections will pave the way towards 
proving Theorem~\ref{thm intro} and Theorem~\ref{thm intro expl} in Section~\ref{sec proofs}.

We conclude this article with an appendix concerning dual Specht $\ZZ\mathfrak{S}_n$-lattices. 
More precisely, in unpublished work \cite{Wildon2002} Wildon has presented, for every partition $\lambda$ of $n$, a concrete 
$\ZZ\mathfrak{S}_n$-monomorphism that embeds the $\ZZ$-linear dual $(S^\lambda_\ZZ)^*$ into $S^\lambda_\ZZ$.
Wildon's proof refers to the proof of \cite[\S 7.4, Lemma 5]{Fulton1997}, which, however, seems to contain some subtleties. Wildon's  
embedding $(S^\lambda_{\ZZ})^*\to S^\lambda_\ZZ$ has been one of the key steps in our proof of Theorem~\ref{thm n 2 mod 4}, and
should also be of independent interest. Therefore, we consider it to be worthwhile giving some more details on the construction of the
embedding, following the lines of \cite{Wildon2002} and \cite[Theorem 6.7]{James1978}. We also mention that similar
constructions can be found in work of Fayers~\cite[Section 4]{Fayers2003}.

\bigskip
\noindent
{\bf Acknowledgements:} The authors would like to thank John Murray for drawing
their attention to \cite{Wildon2002}. 
As well, both authors would like to thank the Universities of Kaiserslautern and Eichst\"att-Ingolstadt for their
kind hospitality during mutual visits.

%%%%%%%%%%%%%%%%%%%%%%%%%%%%%%%%%%%%%%%%%%%%%%%%%%%%%%%%%%%%%%

\section{Preliminaries on modules and lattices}\label{sec general}

Throughout this paper let $R$ be a principal ideal domain and $K$ its field of fractions. 
We start by summarizing some well-known facts concerning $RG$-lattices and $KG$-modules that we shall need 
throughout. We assume the reader to be familiar with the basic notions on modules over group algebras of finite groups, and
refer to \cite{Curtis1981,Nagao1989} for background.

By $\ZZ$, $\NN$ and $\NN_0$ we denote the set of integers, the set of positive integers and the set of non-negative integers, respectively. 
For every prime number $p$, we further denote by $\ZZ_p$ the ring of $p$-adic integers and by $\QQ_p$ the field of $p$-adic numbers, the $p$-adic valuation will be denoted by $\nu_p$.
The localization of $\ZZ$ at the prime ideal $(p)$ will be denoted by $\ZZ_{(p)}$.

\begin{nota}\label{nota lattice} Let $G$ be a finite group.

\smallskip

(a) An {\sl $RG$-lattice} $L$ is always understood to be a left $RG$-module that is finitely generated and free over $R$; the $R$-rank of
$L$ will be denoted by $\rk_R(L)$.

\smallskip

(b)\, Suppose that $F$ is any field. By an {\sl $FG$-module} we shall always mean a finitely generated left $FG$-module.
If $V$ is an $FG$-module, then we denote by $\Rad(V)$ the {\sl (Jacobson) radical} of $V$ and by $\Hd(V):=V/\Rad(V)$
the {\sl head} of $V$. 

For $i\geq 0$, we denote the $i$th radical of $V$ by $\Rad^i(V)$, where $\Rad^0(V):=V$. Suppose that $V$ has Loewy length $l\geq 1$ with
Loewy layers $\Rad^{i-1}(V)/\Rad^i(V)\cong D_{i1}\oplus\cdots \oplus D_{i r_i}$, for $i\in\{1,\ldots,l\}$, $r_1,\ldots,r_l\in \NN$
and simple $FG$-modules
$D_{i1},\ldots,D_{ir_i}$. Then we shall write
\begin{equation}\label{eqn Loewy}
V\sim\begin{bmatrix}   D_{11}\oplus\cdots \oplus D_{1r_1}\\\vdots\\D_{l1}\oplus\cdots \oplus D_{lr_l} \end{bmatrix}\,,
\end{equation}
and say that $V$ has {\sl Loewy series} (\ref{eqn Loewy}).
\end{nota}

\begin{noth}\label{noth change rings}{\bf Change of coefficient rings.}\,
Let $S$ be a principal ideal domain domain, and let $\rho:R\to S$ be a unitary ring homomorphism, so that $S$ becomes an $(R,R)$-bimodule via $\rho$.

\smallskip

(a)\, With the above notation, we shall identify the $S$-algebra $S\otimes_R RG$ with the group algebra $SG$ in the usual way. 

Suppose that $L$ is an $RG$-lattice with $R$-basis $\{b_1,\ldots,b_k\}$. Then the $S$-lattice
$L_S:=S\otimes_R L$ becomes naturally an $SG$-lattice with $S$-basis $\{1\otimes b_1,\ldots,1\otimes b_k\}$. 

If $L_1$ and $L_2$ are $RG$-lattices then there is an isomorphism of $SG$-lattices $S\otimes_R(L_1\otimes_R L_2)\cong (S\otimes_R L_1)\otimes_S (S\otimes_R L_2)$.

\smallskip

(b)\, Suppose that $\rho$ is injective. Then we may view $L$ as an $RG$-submodule of $S\otimes_ R L$, by identifying
$x\in L$ with $1\otimes x$. 
 With these conventions,
we then have $L_S=S\otimes_R L=SL={}_S\langle b_1,\ldots,b_k\rangle$.

\smallskip

(c)\, Consider the special case where $S=K$ and $\rho$ is the inclusion map. Every $KG$-module $V$ admits an $RG$-lattice $L$ such that $V\cong KL$ and $\rk_R(L)=\dim_K(V)$; see, for instance,
\cite[Theorem (73.6)]{Curtis1962}. One calls $L$ an {\sl $R$-form} of $V$.

\smallskip

(d)\, Suppose that $L$ is an $RG$-lattice, and denote by $L^*:=\Hom_R(L,R)$ the dual $RG$-lattice. 
If $\{b_1,\ldots,b_k\}$ is an
$R$-basis of the $RG$-lattice $L$ then we denote by $\{b_1^*,\ldots,b_k^*\}$ the $R$-basis of $L^*$ that is dual to $\{b_1,\ldots,b_k\}$.
If $L\cong L^*$ as $RG$-lattices, then $L$ is called {\sl self-dual}.

Furthermore, we always obtain an $SG$-isomorphism 
$$S\otimes_R L^*\to (S\otimes_R L)^*=\Hom_S(S\otimes_R L,S)\,,$$
sending $1\otimes b_i^*$ to $(1\otimes b_i)^*$, for $i\in\{1,\ldots,k\}$.

\smallskip

(e)\, If $S$ is flat as a right $R$-module via $\rho$ and if $L_1$ and $L_2$ are $RG$-lattices, then there is an isomorphism
of $S$-modules
$$\Psi:S\otimes_R\Hom_{RG}(L_1,L_2)\cong \Hom_{SG}(S\otimes_RL_1,S\otimes_RL_2)$$
such that $(\Psi(\alpha \otimes \varphi))(\beta\otimes x)= \alpha\beta\otimes \varphi(x)$, for
$\alpha,\beta\in S$, $\varphi\in \Hom_{RG}(L_1,L_2)$, and $x\in L_1$;
see \cite[Theorem 1.11.7]{Nagao1989}.

\smallskip

(f)\, Let $\mathfrak{a}$ be an ideal in $R$, let $S:=R/\mathfrak{a}$, and let $\rho:R\to S$ be the
canonical projection. Suppose that $L$ is an $RG$-lattice. Then the factor $RG$-module $L/\mathfrak{a}L$ naturally
becomes an $SG$-lattice, and one has an $SG$-isomorphism
$$L/\mathfrak{a}L\cong S\otimes_RL\,;$$
see \cite[Theorem 1.9.17]{Nagao1989}. As well, note that, by the Third Isomorphism Theorem, every $SG$-sublattice of $L/\mathfrak{a}L$
is of the form $M/\mathfrak{a}L$, where $M$ is an $RG$-sublattice of $L$ containing $\mathfrak{a}L$.
\end{noth}

\begin{lemma}\label{lemma max sub}
Let $M$ be an $RG$-lattice.

\smallskip

{\rm (a)}\, 
  If $N$ is a maximal $RG$-sublattice of $M$, then there exists a maximal ideal $\mathfrak m$ of $R$ such that $\mathfrak m M \subseteq N \subseteq M$.
  
  \smallskip

{\rm (b)}\,  Let $N$ be an $RG$-sublattice of $M$ with $\mathfrak m M \subseteq N \subseteq M$, for some maximal ideal $\mathfrak m$ of $R$. Then the following are equivalent:

\smallskip

\quad {\rm (i)}\, $N$ is a maximal $RG$-sublattice of $M$;

\quad {\rm (ii)}\,  $N/\mathfrak m M$ is a maximal $(R/\mathfrak m)G$-submodule of $M/\mathfrak m M$;

\quad {\rm (iii)}\, $M/N$ is a simple $(R/\mathfrak m)G$-module.
\end{lemma}

\begin{proof}
Part (a) is consequence of Nakayama's Lemma; a proof can be found in \cite[Lemma~8.3]{Hofmann2016}.
Part (b) is obvious.
\end{proof}

\begin{noth}\label{noth order ideal}{\bf Order ideal and index.}\, 
Suppose that $M$ and $N$ are $R$-lattices of the same rank $n\in\NN$ with $N\subseteq M$. Then the factor module $M/N$ is  a torsion
$R$-module and, hence, admits a decomposition 
$$M/N\cong \prod_{i=1}^m R/(r_i)\,,$$
for suitable $m\in\NN$ and $r_1,\ldots,r_m\in R$. The ideal $(r_1\cdots r_m)$ of $R$ is independent of the chosen
decomposition; one sets $(M:N):=(M:N)_R:=(r_1\cdots r_m)$, and calls $(M:N)$ the {\sl order ideal} of $N$ in $M$.
For details concerning order ideals see \cite[\S 4D]{Curtis1981}.

\smallskip

Note that one always has an $R$-endomorphism $\phi:M\to M$ with $\phi(M)=N$.
Using this, a connection between the determinant of any such $R$-endomorphism of $M$, the order ideal of $N$ in $M$ and the ordinary index $[M:N]$ from group theory is given by the next proposition.
In the proof we shall use that $[M:N]=\prod_{i=1}^m|R/(r_i)|=|R/(r_1\cdots r_m)|=|R/(M:N)|$, provided that every proper factor ring of $R$ is
finite. This is due to the fact that $R$ is a unique factorization domain, and can be deduced immediately from the following
observation and the Chinese Remainder Theorem: Suppose that every proper factor ring of $R$ is finite. Suppose further that $p$ is a prime element in $R$ and $k\in \NN$.
Then one has a surjective $R$-module homomorphism $R/(p^k)\to (p^{k-1})/(p^k)$, given
by multiplication with $p^{k-1}$, with kernel $(p)/(p^k)$. This gives $R/(p)\cong (p^{k-1})/(p^k)$ as
$R$-modules, and implies $|R/(p^k)|=|R/(p)|\cdot |(p)/(p^2)|\cdots |(p^{k-1}/(p^k)|=|R/(p)|^k$.
\end{noth}

\begin{prop}\label{prop det}
Let $M$ and $N$ be $RG$-lattices such that $N\subseteq M$
and $\rk_R(M)=\rk_R(N)=n\in\NN$. Let further $\phi: M\to M$ be an $R$-endomorphism of $M$ with $\phi(M) = N$.
Then one has $(M:N)=(\det(\phi))$. If, moreover, every proper factor ring of $R$ is finite, then
$$|R/(M:N)|=[M:N]=|R/(\det(\phi))|\,;$$
in particular, if $R=\ZZ$, then $[M:N]=|\det(\phi)|$. If $R=\ZZ_p$ for some prime number $p$, then 
$[M:N]=p^{\nu_p(\det(\phi))}$. 
\end{prop}

\begin{proof}
By \cite[Proposition~(4.20a)]{Curtis1981}, we have $(M:N)=(\det(\phi))$.
Now suppose that every proper factor ring of $R$ is finite.
Let $M/N\cong \prod_{i=1}^m R/(r_i)$ be a decomposition into cyclic torsion $R$-modules as in \ref{noth order ideal}.
Then $r_i\neq 0$, for all $i\in \{1,\ldots,m\}$. Since all proper quotients of $R$ are finite and since $R$ is
a principal ideal domain, we have
$[M:N]=\prod_{i=1}^m|R/(r_i)|=|R/(r_1\cdots r_m)|=|R/(M:N)|$, as noted in \ref{noth order ideal}.
\end{proof}

\begin{lemma}\label{lemma aligned bases}
Let $\mathfrak{m}=(p)$ be a  maximal ideal in $R$, and let $M$ and $N$ be $R$-lattices of rank $n\in\NN$ such that
$\mathfrak{m}M\subseteq N\subseteq M$. Let further $k$ be the residue field $R/\mathfrak{m}$, and
let $s:=\dim_k(M/N)$. 

\smallskip

{\rm (a)}\, There exists an $R$-basis $\{v_1,\ldots,v_n\}$ of $M$ and an $R$-basis
$\{w_1,\ldots,w_n\}$ of $N$ such that $v_i=w_i$, for $i\in\{1,\ldots,n-s\}$, and $w_i=p v_i$, for
$i\in\{n-s+1,\ldots,n\}$.

\smallskip

{\rm (b)}\, Suppose that $s=1$, and let $\{v_1,\ldots,v_n\}$ be any $R$-basis of $M$. Then there
exists an $R$-basis $\{w_1,\ldots,w_n\}$ of $N$ and $r_1,\ldots,r_n\in R$ such that $w_j=p v_j$, for some
$j\in\{1,\ldots,n\}$, and $w_i=v_i+r_i v_j$, for all $i\neq j$.
\end{lemma}

\begin{proof}
In the proof we shall make use of the Hermite and Smith normal form, see \cite[Section 2.4]{Cohen1993}
and \cite[Sections 5.2, 5.3]{Adkins1992}.

\smallskip

(a)\, Let $\{x_1,\dotsc,x_n\}$ and $\{y_1,\dotsc,y_n\}$ be arbitrary $R$-bases
of $M$ and $N$ respectively, and let $T=(t_{ij}) \in \Mat_{n\times n}(R)$ be the 
(uniquely determined)
matrix with 
$\sum_{j=1}^nt_{ji}x_j=y_i$, for $i\in\{1,\ldots,n\}$. That is, $T$ is the matrix of the $R$-endomorphism
$\phi:M\to M,\, x_i\mapsto y_i$ with respect to the basis $\{x_1,\dotsc,x_n\}$ of $M$.
By the
theory of the Smith normal form, there exists a diagonal matrix $S = \operatorname{diag}(s_1,\dotsc,s_n) \in
\Mat_{n\times n}(R)$ and invertible matrices $U_1, U_2 \in \operatorname{GL}_n(R)$
with $U_2\cdot T\cdot U_1 = S$. Moreover, $s_i \mid s_{i+1}$, for $i\in\{1,\ldots,n - 1\}$ and $M/N \cong \prod_{i=1}^n R/(s_i)$ as $R$-modules.
But as $k$-vector spaces and $R$-modules we also have $M/N \cong k^s=(R/(p))^s$.
So, by \cite[Corollary 7.6]{Adkins1992}, we may assume that $s_i = 1$ for $i \in\{1,\dotsc,n - s\}$, and $s_i = p$ for $i \in\{n-s+1,\dotsc,n\}$.
Now if $U_1=(u_{ij})$ and $U_2^{-1}=(u_{ij}')$, then we set $v_i:=\sum_{j=1}^nu'_{ji}x_j$ and 
$z_i:=\sum_{j=1}^n u_{ji} x_j$. Then $\{z_1,\ldots,z_n\}$ and $\{v_1,\ldots,v_n\}$ are $R$-bases of $M$ and
$S$ is the matrix of $\phi$ with respect to these bases. Thus $\phi(z_i)=s_i v_i\in N$, for $i\in\{1,\ldots,n\}$.
Setting $w_i:=s_iv_i$, for $i\in\{1,\ldots,n\}$,
the claim of (a) follows.

\smallskip

(b)\, We proceed as in (a) but use the Hermite normal form instead of the Smith normal form.
Let $\{v_1,\dotsc,v_n\}$ be the fixed $R$-basis of $M$, let $\{y_1,\dotsc,y_n\}$ be an arbitrary $R$-basis of $N$, and
let $T \in \Mat_{n\times n}(R)$ be such that 
$\sum_{j=1}^nt_{ji}v_j=y_i$, for $i\in\{1,\ldots,n\}$.

Now let $P\subseteq R$ be a set of representatives of the equivalence classes modulo $R$-associates. We may choose 
$1\in P$ (as a representative of the equivalence class of the units in $R$) as well as $p\in P$. For each $r\in P$, let further $P(r)$ be
a set of representatives of the residue classes of $R/(r)$; in particular, we choose $P(1)=\{0\}$.
By appealing to the Hermite normal form, there exists a transformation $U \in \operatorname{GL}_n(R)$ such that $H = TU$
is a lower triangular matrix with the following properties: the diagonal entries are non-zero elements of $P$, and
if $r_i\in P$ is the entry at position $(i,i)$, then all entries at the positions $(i,1),\ldots,(i,i-1)$ are elements of $P(r_i)$.
By our hypothesis, we have
$M/N\cong R/(p)$ as $R$-modules and $k$-vector spaces. Hence 
$(p)=(M:N)=(\det(T))=(\det(H))$, by Proposition~\ref{prop det}.
With our choice of $P$ this forces that there is a unique diagonal entry of $H$ not equal to 1, say at position $(j,j)$, and this
entry is equal to $p$.
If $H=(h_{il})$, the claim then follows with $w_i:=\sum_{l=1}^nh_{li}v_l$, for $i\in\{1,\ldots,n\}$, since then
$w_j=pv_j$ and $w_i=v_i+h_{ji}v_j$ for $i\neq j$.
 \end{proof}

The next results will provide the key method for determining $R$-forms of simple
$K G$-modules in subsequent sections. 
These statements go back to \cite{Plesken1974} in case $R = \ZZ$ and can easily be generalized to the setting of the present paper;
see also \cite[Lemma 8.4]{Hofmann2016b}.

\begin{prop}\label{prop p max sublattices}
Let $\mathfrak{m}$ be a maximal ideal in $R$, and let $k$ be the residue field $R/\mathfrak{m}$. Moreover, let
$M$ be an $RG$-lattice, and let $\pi:M\to M/\mathfrak{m}M$ be the canonical projection.

\smallskip

{\rm (a)}\, Let $D$ be a simple $kG$-module, and suppose that there is a non-zero
$kG$-homomorphism $\phi\in \Hom_{kG}(M/\mathfrak{m}M,D)$. Let $N:=\pi^{-1}(\ker(\phi))$. Then $N$ is a 
maximal $RG$-sublattice of $M$ with $\mathfrak{m}M\subseteq N\subseteq M$ and $M/N\cong D$ as
$kG$-modules; in particular, $\mathfrak{m}^{\dim_k(D)}=(M:N)$.

\smallskip

{\rm (b)}\, Conversely, let $N$ be a maximal $RG$-sublattice of $M$ such that $\mathfrak{m}M\subseteq N\subseteq M$,
and let $s\in\NN$ be such that $(M:N)=\mathfrak{m}^s$. Then
there is a simple $kG$-module $D$ and some $\phi\in\Hom_{kG}(M/\mathfrak{m}M,D)$
such that $\dim_k(D)=s$ and $N=\pi^{-1}(\ker(\phi))$.
\end{prop}

\begin{cor}\label{cor p max sublattices}
Let $\mathfrak m$ be a maximal ideal in $R$, let $k$ be the residue field $R/\mathfrak m$, and let $M$ be an $RG$-lattice.

\smallskip

{\rm (a)}\, Suppose that $D$ is an absolutely simple $kG$-module occurring as a composition factor of $\Hd(M/\mathfrak{m}M)$
with multiplicity one. Then there is a unique maximal $RG$-sublattice $N$ of $M$ such that $\mathfrak{m}M\subseteq N\subseteq M$ and
$M/N\cong D$. 

\smallskip

{\rm (b)}\,  Suppose that $\Hd(M/\mathfrak m M)$ is multiplicity-free with absolutely simple composition factors 
$D_1,\dotsc,D_l$ of $\Hd(M/\mathfrak m M)$. Then $M$ has exactly $l$ maximal $RG$-sublattices $M_1,\dotsc,M_l$ with 
$\mathfrak m M \subseteq M_i \subseteq M$, for $i\in\{1,\ldots,l\}$. Moreover, after reordering, one has $M/M_i \cong D_i$ for 
$i\in\{1,\ldots,l\}$.
\end{cor}

\begin{proof}
To simplify the notation, let us write $\bar M = M / \mathfrak m M$.
  First note that if $D$ is a simple $kG$-module, then $\Hom_{kG}(\bar M, D) \neq \{ 0 \}$ if and only if $D$ is a composition factor of $\Hd(\bar M)$.

\smallskip

(a)\, By Proposition~\ref{prop p max sublattices},  it suffices to show that $\dim_k(\Hom_{kG}(\bar M, D)) = 1$. Then every non-zero element of $\Hom_{kG}(\bar M, D)$ will have the same kernel.
Since $D$ is absolutely simple and occurs with compostion multiplicity one in $\Hd(\bar{M})$,
we have $\dim_k(\Hom_{kG}(\Hd(\bar M), D)) = 1$. The kernel of every non-zero $kG$-homomorphism
$\bar{M}\to D$ is a maximal submodule of $\bar{M}$, thus factors through $\Hd(\bar{M})=\bar{M}/\Rad(\bar{M})$.
Hence the canonical map $\Hom_{kG}(\bar M, D) \to \Hom_{kG}(\Hd(\bar M), D)$ is an isomorphism of $k$-vector spaces.

Assertion (b) now follows from (a).
\end{proof}

Let $M$ and $N$ be $RG$-lattices. Via the conventions in \ref{noth change rings}, we may and shall from now on
identify the $K$-vector spaces $\Hom_{KG}(KM,KN)$ and $K\otimes_R\Hom_{RG}(M,N)$. 
Note that every homomorphism of $RG$-lattices $\phi \colon M \to N$ has a unique extension to a homomorphism of $KG$-modules $KM \to KN$, which, by abuse of notation, will also be denoted by $\phi$.
\begin{lemma}\label{lem:hom1}
Let $M$ and $N$ be $RG$-lattices.

\smallskip

{\rm (a)}\, 
One has 
$$\Hom_{RG}(M, N) = \{ \phi_{|_M} : \phi \in \Hom_{KG}(KM,KN), \, \phi(M) \subseteq N \}$$ 
and $\rk_R(\Hom_{RG}(M,N))=\dim_K(\Hom_{KG}(KM,KN))$.

\smallskip

{\rm (b)}\, Suppose that $\Hom_{RG}(M,N)$ is an $R$-lattice of rank one, generated by $\phi$. If $0\neq r\in R\smallsetminus R^\times$,
then the $KG$-homomorphism $r^{-1}\phi\in \Hom_{KG}(KM,KN)$ is 
not the extension of an element in $\Hom_{RG}(M,N)$.

\smallskip

{\rm (c)}\, Suppose that $\Hom_{RG}(M,N)$ is an $R$-lattice of rank one. Suppose further that 
there is an $RG$-isomorphism $\phi:M\to N$. Then $\Hom_{RG}(M,N)={}_R\langle \phi\rangle$.
\end{lemma}

\begin{proof}
Part (a) is clear. 

As for (b), note that if $\Hom_{RG}(M,N)$ has $R$-rank 1, then 
$\Hom_{KG}(KM,KN)$ has $K$-dimension 1, and is spanned by $r^{-1}\phi$, for every $0\neq r\in R$. 
Thus, if $(r^{-1}\phi)_{\mid M}\in \Hom_{RG}(M,N)$, then 
$\phi(x)=rs\phi(x)$, for some $s\in R$ and all $x\in M$. But this forces $\phi=rs\phi$, $rs=1$ and $r\in R^\times$.

To prove (c), let $\{\psi\}$ be an $R$-basis 
of $\Hom_{RG}(M,N)$, and let $r\in R$ be such that $\phi=r\psi$.
Then $rN \subseteq N = \phi(M) = (r\psi)(M) = r \psi(M) \subseteq r N$, which implies $rN = N$. Hence $r\in R^\times$ and $\{\phi\}$ is also an $R$-basis of $\Hom_{RG}(M, N)$.
\end{proof}

\begin{prop}\label{prop extend}
Let $V_1$ and $V_2$ be simple $KG$-modules with $R$-forms $L_1$ and $L_2$, respectively.
If $\phi:L_1\to L_2$ is a non-zero $RG$-homomorphism, then $\phi$ is injective.
\end{prop}

\begin{proof}
Assume that $x\in L_1$ is such that $x\neq 0$ and $\phi(x)=0$, and view $\phi$ as a  $KG$-homomorphism
$KL_1\to KL_2$ as before. Since $KL_1\cong V_1$ and $KL_2\cong V_2$, Schur's Lemma implies that 
$\phi$ is either a $KG$-isomorphism or the zero map.
Since $0\neq x\in KL_1$ and $\phi(x)=0$, we must have $\phi=0$, both as $KG$-homomorphism and as $RG$-homomorphism,
a contradiction. 
Therefore, $\phi$ is injective.
\end{proof}

In the next section we shall explain how to obtain $\ZZ$-forms of absolutely simple $\QQ G$-modules from
$p$-local data. There and in the following we shall make use of the following observation:

\begin{rem}\label{rem forms all in one}
Let $M$ be an $RG$-lattice of rank $n\in\NN$.
Assume that also $N$ is an $RG$-lattice and there is an $KG$-isomorphism $\phi\colon KN\to KM$. Since both $M$ and $\phi(N)$ are full-rank $RG$-sublattices of $KM$, there exists $0 \neq r \in R$ such that $\phi(rN) = r \phi(N) \subseteq M$. 
Moreover, $\phi(rN)\cong rN\cong N$ as $RG$-lattices.
This shows that every $RG$-lattice $N$ with $KN \cong KM$ as $KG$-modules is isomorphic to an $RG$-sublattice of the fixed $RG$-lattice $M$. 

Conversely, if $N\subseteq M$ is an $RG$-sublattice of rank $n$, then $KN$ is a $KG$-submodule of $KM$ and both modules
have $K$-dimension $n$. Thus $KM=KN$.
Consequently, determining all isomorphism classes of $RG$-lattices $N$ with $KN \cong KM$ as $KG$-modules is equivalent to determining all isomorphism classes of
$RG$-sublattices of $M$ of full rank.
\end{rem}
 
 We conclude this section with the following consequence of a theorem of Brauer--Nesbitt.

\begin{prop}\label{prop notdivides}
Suppose that $R$ is a discrete valuation ring with maximal ideal $\mathfrak{m}$ and residue field
$k:=R/\mathfrak{m}$. Let $V$ be any $KG$-module.

\smallskip

{\rm (a)}\, If $M$ and $N$ are $RG$-forms of $V$, then the $kG$-modules $M/\mathfrak{m}M$ and
$N/\mathfrak{m}N$ have the same composition factors.

\smallskip

{\rm (b)}\, Suppose that $V$ is absolutely simple. Suppose further that, for an $R$-form $M$ of $V$,
the $kG$-module $M/\mathfrak{m}M$ is also absolutely simple. Then there is only one isomorphism
class of $R$-forms of $V$. If, moreover, $V$ is self-dual, then so is every $R$-form of $V$.
\end{prop}

\begin{proof}
The assertion  in part (a) is due to Brauer--Nesbitt; a proof can be found in \cite[Proposition (16.16)]{Curtis1981}.

The first assertion of (b) is then a direct consequence of (a) and Corollary~\ref{cor p max sublattices}: if $M$ is an $R$-form
such that $M/\mathfrak{m}M$ is absolutely simple, then $\mathfrak{m}M$ is the unique maximal sublattice of $M$. 
Moreover, $\mathfrak{m}M\cong M$ as $\mathfrak m$ is principal.

Lastly suppose that $V$ is self-dual. If $M$ is an $R$-form of $V$, then
so is its dual lattice $M^*$, by \ref{noth change rings}(d). Since all $R$-forms of $V$ are mutually isomorphic, the claim follows.
\end{proof}

%%%%%%%%%%%%%%%%%%%%%%%%%%%%%%%%%%%%%%%%%%%%%%%%%%%%%%

\section{Determining integral forms of simple $\QQ G$-modules}\label{sec forms}

Let $G$ be a finite group, and suppose that $V$ is a simple $\QQ G$-module.  As mentioned in \ref{noth change rings},
$V$ always admits a $\ZZ$-form $M$, which  is in general far from being unique, not even up to
isomorphism.  However the Theorem of Jordan--Zassenhaus ensures that there are only finitely many 
$\ZZ$-forms of $V$ up to $\ZZ G$-isomorphism.
A similar result holds in the local setting: If $p$ is a prime number then every simple $\QQ_p G$-module
admits only finitely many $\ZZ_p$-forms up to $\ZZ_p G$-isomorphism. For both statements we refer to \cite[Theorem (24.1), Theorem (24.7)]{Curtis1981},
which holds in much greater generality. We content ourselves with the rational version, since this is the one we shall need 
in Section~\ref{sec proofs}.

In the following we shall be concerned with the case that $V$ is absolutely simple, that is, $K\otimes_\QQ V$ is a simple
$KG$-module, for every extension field $K$ of $\QQ$. We shall explain in more detail how $\ZZ$-forms of $V$ and $\ZZ_p$-forms
of the $p$-adic completion $V_{\QQ_p}=\QQ_p\otimes_{\ZZ_p} V$ are related.
These methods will provide our strategy towards proving Corollary~\ref{cor intro} and Theorem~\ref{thm intro expl} in Section~\ref{sec proofs}. Our main reference here is \cite[\S 30, \S 31]{Curtis1981}.

\begin{noth}\label{noth local}{\bf Localization and $p$-completion.}\,
(a)\, Let $M$ be a $\ZZ$-form of $V$, and let $p$ be a prime number. Then $M_{\ZZ_{(p)}}$ is a $\ZZ_{(p)}$-form
of $V$ as well, and $M_{\ZZ_p}$ is a $\ZZ_p$-form of $V_{\QQ_p}$. 

\smallskip

(b)\, Suppose further that $N\subseteq M$ is also a $\ZZ$-form of $V$, and let $\phi \colon M \to M$ be a
$\ZZ$-endomorphism with $\phi(M)=N$.  Since $\ZZ_{(p)}$ and $\ZZ_p$
are both flat over $\ZZ$, we can regard $N_{\ZZ_{(p)}}$ as a $\ZZ_{(p)} G$-sublattice of $M_{\ZZ_{(p)}}$, and
$N_{\ZZ_{p}}$ as a $\ZZ_{p} G$-sublattice of $M_{\ZZ_{p}}$. 
The $\ZZ$-endomorphism $\phi$ of $M$ yields a $\ZZ_{(p)}$-endomorphism of $M_{\ZZ_{(p)}}$ with image $N_{\ZZ_{(p)}}$
as well as a $\ZZ_p$-endomorphism of $M_{\ZZ_{p}}$ with image $N_{\ZZ_{p}}$.
By 
\cite[Corollary~(4.18), Proposition~(4.20a)]{Curtis1981}, one has
$(M:N)=\ZZ\cdot \det(\phi)$, $(M_{\ZZ_{(p)}}:N_{\ZZ_{(p)}})=\ZZ_{(p)}\cdot \det(\phi)$ and
$(M_{\ZZ_{p}}:N_{\ZZ_{p}})=\ZZ_{p}\cdot \det(\phi)$. 

In particular, if $\det(\phi)\in \ZZ_p^\times$, then one has $M_{\ZZ_p}=N_{\ZZ_p}$, by Proposition~\ref{prop det}. 

If $[M:N]$ is a $p$-power, then Proposition~\ref{prop det} implies
$[M:N]=[M_{\ZZ_{(p)}}:N_{\ZZ_{(p)}}]=[M_{\ZZ_{p}}:N_{\ZZ_{p}}]$.
 \end{noth}

As a converse of (a), one has

\begin{prop}\label{prop rep iso p}
Let $V$ be an absolutely simple $\QQ G$-module, let $p$ be a prime number, and let $N$ be a $\ZZ_p$-form of
$V_{\QQ_p}$. Then there is a $\ZZ$-form $M$ of $V$  and a $\ZZ_{(p)}$-form $L$ of $V$ such that 
$M_{\ZZ_{(p)}}\cong L$ and $M_{\ZZ_p}\cong L_{\ZZ_p}\cong N$.
\end{prop}

\begin{proof}
The assertion follows from \cite[Proposition~(23.13), Corollary~(30.10)]{Curtis1981}.
\end{proof}

\begin{prop}\label{prop local}
Let $V$ be an absolutely simple $\QQ G$-module with $\ZZ$-forms $M$ and $N$. Let further $\{p_1,\ldots,p_g\}$ be
the set of prime numbers dividing $|G|$. Then the following are equivalent:

\smallskip

{\rm (i)}\, $M\cong N$ as $\ZZ G$-lattices;

\smallskip

{\rm (ii)}\, for all prime numbers $p$, one has $M_{\ZZ_{(p)}}\cong N_{\ZZ_{(p)}}$ as $\ZZ_{(p)} G$-lattices;

\smallskip

{\rm (iii)}\, for all $i\in\{1,\ldots,g\}$, one has  $M_{\ZZ_{(p_i)}}\cong N_{\ZZ_{(p_i)}}$ as $\ZZ_{(p_i)} G$-lattices;

\smallskip

{\rm (iv)}\, for all prime numbers $p$, one has $M_{\ZZ_{p}}\cong N_{\ZZ_{p}}$ as $\ZZ_{p} G$-lattices;

\smallskip

{\rm (v)}\, for all $i\in\{1,\ldots,g\}$, one has  $M_{\ZZ_{p_i}}\cong N_{\ZZ_{p_i}}$ as $\ZZ_{p_i} G$-lattices.
\end{prop}

\begin{proof}
The equivalences (ii)$\Leftrightarrow$(iv) and (iii)$\Leftrightarrow$(v) follow from \cite[Proposition~(30.17)]{Curtis1981},
while the equivalences (ii)$\Leftrightarrow$(iii) and (iv)$\Leftrightarrow$(v) follow from \cite[Proposition~(31.2)]{Curtis1981}.
Lastly, By Maranda's Theorem \cite[Theorem~(81.5)]{Curtis1962}, also (i) and (ii) are equivalent.
\end{proof}

As a consequence of Proposition~\ref{prop local} and the considerations in \cite[\S 31A., page 651]{Curtis1981},
one obtains

\begin{cor}\label{cor number local}
Let $V$ be an absolutely simple $\QQ G$-module. For each prime number, let $h_p(V)$ be the
number of isomorphism classes of $\ZZ_p$-forms of $V_{\QQ_p}$, and let $j(V)$ be the number
of isomorphism classes of $\ZZ$-forms of $V$. If 
$\{p_1,\ldots,p_g\}$ is the set of prime divisors
of $|G|$, then one has
$$j(V)=\prod_{i=1}^g h_{p_i}(V)\,.$$
\end{cor}

The following result shows how Corollary~\ref{cor number local} can be strengthened to give explicit representatives for the isomorphism classes of $\ZZ$-forms of $V$.
It is implicitely contained in~\cite{Plesken1977, Plesken1974} (see also \cite{Hofmann2016b}).
Since it will be essential in Section~\ref{sec proofs}, for the readers' convenience we give a proof.

\begin{prop}\label{prop rep iso}
Let $V$ be an absolutely simple $\QQ G$-module, and let $M$ be  a $\ZZ$-form of $V$.
Let $\{p_1,\ldots,p_g\}$ be the set of prime divisors of $|G|$. 
Suppose that, for each $p \in \{p_1,\dotsc,p_g\}$,
one is given a set $X_p$ of $\ZZ$-forms of $V$ contained in $M$ such that, for all $N \in X_p$,
the index $[M : N]$ is a $p$-power and $\{ N_{\ZZ_p} : N \in X_p \}$ is a set of
representatives of the isomorphism classes of $\ZZ_p$-forms of $V_{\QQ_p}$.
Then
\[ \{ N_1 \cap \dotsb \cap N_g : N_i \in X_{p_i}, \, 1 \leq i \leq g \} \]
is a set of representatives of the isomorphism classes of $\ZZ$-forms of $V$.
\end{prop}

\begin{proof}
By Proposition~\ref{prop local},  the $\ZZ$-forms $L$ and $N$ of $V$ are isomorphic if and only if the $\ZZ_p$-forms $L_{\ZZ_p}$ and$N_{\ZZ_p}$
of $V_{\QQ_p}$ are isomorphic, for all $p\in \{p_1,\dotsc,p_g\}$. Moreover, by Corollary~\ref{cor number local}, the number of isomorphism classes of $\ZZ$-forms of $V$ is equal to $\prod_{i=1}^g |X_{p_i}|$.
Let us write $X := \{ N_1 \cap \dotsb \cap N_g : N_i \in X_{p_i}, \, 1 \leq i \leq g \}$.
It suffices to show that $|X| = \prod_{i=1}^g |X_{p_i}|$ and that no two elements of $X$ are isomorphic.
To this end, let $L = N_1 \cap \dotsb \cap N_g \in X$ with $N_i \in X_{p_i}$ for $i\in\{1,\ldots,g\}$.
Let $i\in\{1,\ldots,g\}$. By assumption, for $j \neq i$, we have that
$(M_{\ZZ_{p_i}}:(N_j)_{\ZZ_{p_i}})=\ZZ_{p_i}[M:N_j]$, where $[M:N_j]$ is a power of $p_j\in \ZZ_{p_i}^\times$, hence
$(M_{\ZZ_{p_i}}:(N_j)_{\ZZ_{p_i}})= \ZZ_{p_i}$ and $(N_j)_{\ZZ_{p_i}} = M_{\ZZ_{p_i}}$. This yields
\[ L_{\ZZ_{p_i}} = (N_1 \cap \dotsb \cap N_g)_{\ZZ_{p_i}} = (N_i)_{\ZZ_{p_i}} \cap M_{\ZZ_{p_i}} = (N_i)_{\ZZ_{p_i}}. \]
This shows that $X$ has the correct cardinality and the elements of $X$ are pairwise non-isomorphic.
\end{proof}

We also mention the following useful result, which will be applied in Section~\ref{sec p odd}.

\begin{lemma}\label{lemma possible Loewy}
Let $p$ be a prime, and let $W$ be an absolutely simple self-dual $\QQ_p G$-module. Let $M$ be
a $\ZZ_p$-form of $V$, and suppose that the $\FF_p G$-module $M/pM$ is uniserial with
Loewy series
$$M/pM\sim \begin{bmatrix}  D_1\\ D_2\end{bmatrix}\,,$$
for absolutely simple self-dual $\FF_p G$-modules $D_1$ and $D_2$ with $D_1\not\cong D_2$. 
%Moreover
%suppose that the endomorphism ring $\End_{\FF_p G}(M/pM)$ is one-dimensional. If 
Let $L$ be any
$\ZZ_p$-form of $V$. Then one has the following

\smallskip

{\rm(a)}\, $\End_{\FF_p G}(M/pM)\cong \FF_p$; in particular, the $\FF_pG$-module $M/pM$ is absolutely
indecomposable;

{\rm (b)}\, $L/pL\sim\begin{bmatrix} D_1\\D_2 \end{bmatrix}$ if and only if $L\cong M$;

\smallskip

{\rm (c)}\, $L/pL\sim\begin{bmatrix} D_2\\D_1 \end{bmatrix}$ if and only if $L\cong M^*$.
\end{lemma}

\begin{proof}
To prove (a), let $\phi$ be a non-zero endomorphism of the $\FF_pG$-module $M/pM$. Then $\phi$ is an isomorphism, since $D_1\not\cong D_2$. Thus the endomorphism algebra $\End_{\FF_pG}(M/pM)$ is a division
algebra of finite $\FF_p$-dimension. Since $\FF_p$ is a finite field, $\End_{\FF_pG}(M/pM)$ then has to be
an extension field of $\FF_p$, say $\End_{\FF_pG}(M/pM)\cong \FF_{p^r}$, for some $r\in\NN$. Assume that $r>1$.
Then, by \cite[Theorem 1.11.12]{Nagao1989}, we have 
$$\End_{\FF_{p^r}G}(\FF_{p^r}\otimes_{\FF_p} (M/pM))\cong \FF_{p^r}\otimes_{\FF_p} \End_{\FF_pG}(M/pM)\cong \underbrace{\FF_{p^r}\oplus\cdots\oplus \FF_{p^r}}_r\,$$
as $\FF_{p^r}$-algebras. So, by \cite[Theorem 1.5.4]{Nagao1989}, the $\FF_{p^r}G$-module
$\FF_{p^r}\otimes_{\FF_p} (M/pM)$ splits into the direct sum of $r$ indecomposable modules. Since $D_1$ and $D_2$
are absolutely simple, this forces $r=2$ and $\FF_{p^2}\otimes_{\FF_p} (M/pM)\cong (\FF_{p^2}\otimes_{\FF_p} D_1)\oplus (\FF_{p^2}\otimes_{\FF_p} D_2)$. So $\FF_{p^2}\otimes_{\FF_p} (M/pM)$ is semisimple, and
$\Rad(\FF_{p^2}\otimes_{\FF_p} (M/pM))=\{0\}$. 
%Since $\FF_{p^2}|\FF_p$ is a separable field extension, we also have
By \cite[Lemma 2.51, Lemma 3.1.28]{Nagao1989}, we also have
$\Rad(\FF_{p^2}G)=\FF_{p^2}\otimes_{\FF_p} \Rad(\FF_p G)$. Thus
\begin{align*}
\Rad(\FF_{p^2}\otimes_{\FF_p} (M/pM))&=\Rad(\FF_{p^2}G)\cdot (\FF_{p^2}\otimes_{\FF_p}(M/pM))\\
&=(\FF_{p^2}\otimes_{\FF_p}\Rad(\FF_pG))\cdot (\FF_{p^2}\otimes_{\FF_p}(M/pM))\\
&=\FF_{p^2}\otimes_{\FF_p} \Rad(M/pM)\,.
\end{align*}
But $\Rad(M/pM)\neq\{0\}$, thus $\FF_{p^2}\otimes_{\FF_p} \Rad(M/pM)\neq\{0\}$, a contradiction. This shows that, in fact, $r=1$, thus $\End_{\FF_p G}(M/pM)\cong \FF_p$; in particular, $M/pM$ is an absolutely indecomposable $\FF_pG$-module, by
\cite[Lemma 4.7.1]{Nagao1989}. 
 
\smallskip

I remains to prove (b) and (c). Since $V\cong V^*$, also $M^*$ is a $\ZZ_p$-form of $V$, by \ref{noth change rings}(c).
Since $D_1$ and $D_2$ are both self-dual, $M^*/pM^*$ is uniserial with Loewy series $M^*/pM^*\sim\begin{bmatrix}  D_2\\ D_1\end{bmatrix}$.

\smallskip

Conversely, let $L$ be a $\ZZ_p$-form of $S(k)_{\QQ_p}$ with the same Loewy structure as $M$. 
%Denote by ${}^-:\ZZ_p\to \FF_p$ the canonical projection.
By Lemma~\ref{lem:hom1}, $\Hom_{\ZZ_p \mathfrak{S}_n}(L,M)\cong \ZZ_p$ as $\ZZ_p$-lattice. Thus
$\FF_p\cong \FF_p\otimes_{\ZZ_p} \Hom_{\ZZ_p \mathfrak{S}_n}(L,M)\cong \Hom_{\ZZ_p \mathfrak{S}_n}(L,M)/p \Hom_{\ZZ_p \mathfrak{S}_n}(L,M)$ as $\FF_p$-vector spaces.
Consider the canonical projection ${}^-:\Hom_{\ZZ_p \mathfrak{S}_n}(L,M)\to \Hom_{\ZZ_p \mathfrak{S}_n}(L,M)/p \Hom_{\ZZ_p \mathfrak{S}_n}(L,M)$.
By
\cite[Lemma 4.8.8]{Nagao1989}, we have an injective $\FF_p$-linear map
\begin{equation}\label{eqn mono}
\Phi:\overline{\Hom_{\ZZ_p \mathfrak{S}_n}(L,M)}\to \Hom_{\FF_p\mathfrak{S}_n}(L/pL,M/pM)\,;
\end{equation}
in particular, there is some non-zero homomorphism $\phi\in \Hom_{\FF_p\mathfrak{S}_n}(L/pL,M/pM)$.
If $\phi$ was not an isomorphism, we would have $\mathrm{Im}(\phi)\cong D_2\cong (L/pL)/\ker(\phi)\cong D_1$, a contradiction.
Hence $\phi$ is an isomorphism, implying $\Hom_{\FF_p\mathfrak{S}_n}(L/pL,M/pM)\cong \End_{\FF_p\mathfrak{S}_n}(M)\cong \FF_p$
as $\FF_p$-vector spaces.  Hence the $\FF_p$-monomorphism in (\ref{eqn mono})
actually has to be an isomorphism, so that there is some $\psi\in \Hom_{\ZZ_p\mathfrak{S}_n}(L,M)$ such that
$\Phi(\bar{\psi})=\phi$. 
%Since $0\neq \det(\phi)=\overline{\det(\psi)}$, we conclude that $\det(\psi)\in\ZZ_p^\times$, so that
%$\psi$ is an isomorphism of $\ZZ_p\mathfrak{S}_n$-lattices, proving (a).
Let $\{b_1,\ldots,b_s\}$ be a $\ZZ_p$-basis of $L$ and let $\{c_1,\ldots,c_s\}$ be a $\ZZ_p$-basis of $M$.
Let further $A\in\mathrm{Mat}_{s\times s}(\ZZ_p)$ be the matrix of $\psi$ with respect to these bases.
%We identify $L/pL$ with $\FF_p\otimes_{\ZZ_p} L$ and $M/pM$ with $\FF_p\otimes_{\ZZ_p} M$ as in \ref{noth change rings}(f).
Then $\{b_1+pL,\ldots,b_s+pL\}$ and $\{c_1+pM,\ldots, c_s+pM\}$ are $\FF_p$-bases of $L/pL$
and $M/pM$, respectively. The matrix $\bar{A}$ of $\Phi(\bar{\psi})=\phi$ with respect of these bases is obtained
by reducing the entries of $A$ modulo $p$. In particular, $0\neq \det(\bar{A})=\det(A)+p\ZZ_p$. Hence $\det(A)\in\ZZ_p^\times$,
implying that $\psi$ is a $\ZZ_p\mathfrak{S}_n$-isomorphism.

\smallskip

To prove assertion (b), suppose that $L$ is a $\ZZ_p$-form of $V$ 
such that $L/pL\sim\begin{bmatrix} D_2\\ D_1 \end{bmatrix}$.
Then the dual lattice $L^*$ is also a $\ZZ_p$-form of $V$, and the $\FF_p\mathfrak{S}_n$-module
$L^*/pL^*\cong \FF_p\otimes_{\ZZ_p} L^*\cong
(\FF_p\otimes_{\ZZ_p} L)^*\cong (L/pL)^*$ is uniserial with Loewy series $\begin{bmatrix} D_1\\D_2\end{bmatrix}$.
Thus, by (a), we have $L^*\cong M$, hence $L\cong M^*$, as claimed.
\end{proof}

In general, given a principal ideal domain, its field of fractions $K$ and
an absolutely simple $KG$-module $V$, it is quite difficult to determine all isomorphism classes of $R$-forms of $V$.
We record here two special cases, where $R$ is local and where the reduction modulo the maximal ideal of $R$ reveals enough information to determine all isomorphism classes.
Both will be applied in the setting of Specht modules in Sections~\ref{sec p odd} and~\ref{sec p 2},
where $G$ is a symmetric group and $R=\ZZ_p$, for some prime number $p$.

\begin{prop}\label{prop max forms 2}
Suppose that $R$ is  local with maximal ideal $\mathfrak m$ and residue field $k:=R/\mathfrak{m}$. Let $V$ be a $KG$-module, and let $M$ an $R$-form of $V$. Let $n:=\rk_R(M)\in\NN$.

\smallskip

{\rm (a)}\, Suppose that, for every  $RG$-lattice $N$ with $\mathfrak m M \subseteq N \subseteq M$, the $k G$-module $N/\mathfrak{m}N$ is uniserial. Then the set of $RG$-sublattices of $M$ of rank $n$ is totally ordered.
Moreover, every $RG$-sublattice of $M$ of rank $n$ is of the form 
$\mathfrak m^i N$, for some $i \in \NN_0$ and some $RG$-sublattice $N$ of $M$ with $\mathfrak m M \subsetneq N \subseteq M$.

\smallskip

{\rm (b)}\,  Suppose that the set of $RG$-sublattices of $M$ of rank $n$
is totally ordered. Then the set of $RG$-lattices $\{ N : \mathfrak m M \subsetneq N \subseteq M \}$ contains a set of representatives of the isomorphism classes of $R$-forms of $V$.
\end{prop}

\begin{proof}
First recall from Remark~\ref{rem forms all in one} that, up to isomorphism, the $R$-forms of $V$ are precisely the $RG$-sublattices of $M$ of rank $n$.
Lifting any composition series of the $kG$-module $M/\mathfrak{m}M$ to $RG$, we obtain an $r\in\NN$ and $RG$-lattices $M_1,\ldots,M_r$ such that $\mathfrak{m}M=M_r\subseteq M_{r-1}\subseteq \cdots \subseteq M_2\subseteq M_1=M$ and such that
$M_i$ is a maximal $RG$-sublattice of $M_{i-1}$, for $i\in\{2,\ldots,r\}$.

\smallskip

(a)\, Since, in particular, $M/\mathfrak m M$ is uniserial, we deduce that $M_1,\ldots,M_r$ are precisely the $RG$-sublattices 
of $M$ containing $\mathfrak{m}M$. We claim that
$$\{ \mathfrak m^j M_i: j \in \NN_0, 1 \leq i \leq r \}$$
is the set of all $RG$-sublattices of $M$ of full rank $n$, which is clearly totally ordered. To this end, for every $RG$-sublattice $L$ of $M$ with $\rk_R(L)=n$, we define
$\ell(L)$ to be the minimal integer $s\in\NN_0$ such that there exists a chain of $RG$-lattices $L=L_s\subsetneq \cdots \subsetneq L_0=M$,
where $L_i$ is a maximal $RG$-sublattice of $L_{i-1}$, for all $i\in\{1,\ldots,s\}$.
We argue by induction on $\ell(L)$ to show that $L\in\{ \mathfrak m^j M_i : j \in \NN_0, 1 \leq i \leq r \}$. The case $\ell(L)=0$ is trivial, and if $\ell(L)=1$ then $L=M_2$, since $M_2$ is the unique maximal $RG$-sublattice of $M$. So suppose now that $\ell(L)\geq 2$.
Then there is some $RG$-sublattice $L'$ of $M$ of full rank such that $\ell(L')\leq \ell(L)-1$ and such that $L$ is maximal in $L'$.
By induction, $L' = \mathfrak m^j M_i$ for some $j \in \NN_0$ and some $i \in\{1,\ldots, r\}$.
We shall show that $L'$ has a unique maximal sublattice, namely $\mathfrak m^j M_{i+1}$ if $i <r$, and $\mathfrak m^{j+1} M_1$ if $i = r$.

Let $p\in R$ be such that $(p)=\mathfrak{m}$. Then we have the $RG$-isomorphism $M_i\to \mathfrak{m}^jM_i=L',\; m\mapsto p^j m$. Hence we may suppose that $j=0$, so that $L'=M_i$. By assumption, $M_i/\mathfrak{m}M_i$ is a uniserial
$kG$-module. Therefore, $M_i$ has a unique maximal $RG$-sublattice; by construction, the latter equals $M_{i+1}$ if $i<r$
and $\mathfrak{m}M=\mathfrak{m}M_1$ if $i=r$.
This completes the proof of (a).

\smallskip

(b)\, 
Since $L\cong \mathfrak{m}^jL$, for every $RG$-lattice $L$ and every $j\in\NN_0$, the hypotheses of (b)
imply that  $\{ \mathfrak m^j M_i : j \in \NN_0, 1 \leq i \leq r \}$ is the set of $RG$-sublattices of $M$ of rank $n$, and the assertion follows.
\end{proof}

\begin{prop}\label{prop max forms}
Suppose that $R$ is local with maximal ideal $\mathfrak{m}$ and residue field $k:=R/\mathfrak{m}$. 
Let $V$ be a $KG$-module, and let $M$ be an $R$-form of $V$.
Suppose that $M/\mathfrak{m}M= D_1\oplus D_2$, for absolutely simple $kG$-modules
$D_1$ and $D_2$ with $D_1\not\cong D_2$.
Then the $RG$-lattice $M$ has precisely two maximal sublattices $N_1$ and $N_2$. If, moreover, $N_1/\mathfrak{m}N_1$ and 
$N_2/\mathfrak{m}N_2$ are both uniserial $kG$-modules, then $M,N_1$ and $N_2$ are representatives of the isomorphism classes
of $R$-forms of $V$.
\end{prop}

\begin{proof}
Let $N$ be a maximal $RG$-sublattice of $M$. By Lemma~\ref{lemma max sub}, $\mathfrak{m}M\subseteq N$, since $R$ is local. 
Hence, by Corollary~\ref{cor p max sublattices}, the $RG$-lattice $M$ has precisely two maximal sublattices $N_1$ and $N_2$ with
$M/N_i \cong D_i$, for $i\in\{1,2\}$.

Now suppose that $N_1/\mathfrak{m}N_1$ and $N_2/\mathfrak{m}N_2$ are uniserial $kG$-modules. Then they must have Loewy series
$$N_1/\mathfrak{m}N_1=\begin{bmatrix} D_1\\D_2 \end{bmatrix}\quad \text{ and }\quad N_2/\mathfrak{m}N_2=\begin{bmatrix} D_2\\D_1 \end{bmatrix}\,.$$
By Corollary~\ref{cor p max sublattices}, for $i\in\{1,2\}$, the $RG$-lattice $N_i$ has a unique maximal sublattice, which must then be $\mathfrak{m}M$. As in the proof of Proposition~\ref{prop max forms 2}(a)
we deduce that
the $RG$-sublattices of $M$ of rank $\rk_R(M)$ are precisely those in $\{\mathfrak{m}^jM,\mathfrak{m}^jN_1,\mathfrak{m}^jN_2: j\in\NN_0\}$, which by Remark~\ref{rem forms all in one} contain representatives of the isomorphism classes of $R$-forms of $V$. 
The $RG$-lattices $N_1,N_2$ and $M$ are pairwise non-isomorphic, since their reductions modulo $\mathfrak{m}$
are pairwise non-isomorphic as $kG$-modules.
Since $\mathfrak{m}^jL\cong L$, for every $RG$-lattice $L$
and every $j\in\NN$, the assertion of the proposition now follows.
\end{proof}

\section{Exterior Powers}\label{sec exterior}

In this short section we collect a number of important properties of exterior powers of $RG$-lattices, most of which
can be found in \cite[Section 9.8]{Rotman2002}.
They will be crucial for our study of Specht modules labelled by hook partitions in subsequent sections.

\begin{noth}\label{noth exterior}{\bf Exterior powers of lattices and modules.}\,
As before, let $R$ be a principal ideal domain, let $G$ be  a finite group, and let $M$ be an $RG$-lattice of rank $n\in\NN$.

\smallskip

(a)\, The {\sl exterior algebra} $\bigwedge(M)$ of $M$ is a graded $R$-algebra, with multiplication denoted by
$$\bigwedge(M)\times\bigwedge(M)\to \bigwedge(M)\,,\; (x,y)\mapsto x\wedge y\,.$$
For $k\geq 0$, the $k$-th homogeneous component of $\bigwedge(M)$ is denoted by $\bigwedge^k(M)$, and is called
the {\sl $k$-th exterior power of $M$}.

Suppose that $\{b_1,\ldots,b_n\}$ is an $R$-basis of $M$, and let $k\in\NN$. If $k\leq n$, then $\bigwedge^k(M)$ is 
an $RG$-lattice with $R$-basis $\{b_{i_1}\wedge\cdots\wedge b_{i_k}:1\leq i_1<\cdots < i_k\leq n\}$; the $RG$-action on $\bigwedge^k(M)$ is given by
$$g\cdot (m_1\wedge\cdots\wedge m_k)=gm_1\wedge\cdots\wedge gm_k\,,$$
for $g\in G$ and $m_1,\ldots,m_k\in M$. If $k>n$ then $\bigwedge^k(M)=\{0\}$. Thus, for every $k\in\NN_0$, the exterior power
$\bigwedge^k(M)$ is an $RG$-lattice of $R$-rank $\binom{n}{k}$, where $\binom{n}{k}:=0$, for $k>n$.

We also recall that $\wedge$ is alternating, that is, for $k\in\{2,\ldots,n\}$, $m_1,\ldots,m_k\in M$ and $1\leq i<j\leq k$, one has
$m_1\wedge\cdots \wedge m_k=-(m_1\wedge \cdots \wedge m_{i-1}\wedge m_j\wedge m_{i+1}\wedge\cdots\wedge m_{j-1}\wedge m_i\wedge m_{j+1}\wedge\cdots \wedge m_k)$.

\smallskip

(b)\, Let $M$ and $N$ be $R$-lattices, and let $\phi:M\to N$ be an $R$-linear map. As well, let $k\in\NN_0$. Then there is a unique
$R$-linear map $\bigwedge^k(\phi):\bigwedge^k(M)\to \bigwedge^k(N)$ such that 
$$\bigwedge^k(\phi)(m_1\wedge\cdots\wedge m_k)=\phi(m_1)\wedge\cdots\wedge \phi(m_k)\,,$$
for all $m_1,\ldots,m_k\in M$. If $M$ and $N$ are also $RG$-lattices and $\phi$ is an $RG$-homomorphism, then so is
$\bigwedge^k(\phi)$.

\smallskip

(c)\, Exterior powers commute with taking the dual, that is, if $M$ is an $RG$-lattice an $k\in\NN_0$ , then there is 
an $RG$-isomorphism $\bigwedge^k(M^*)\cong (\bigwedge^k(M))^*$. 
For $k>n$, both modules are $\{0\}$.
For $k\leq n$, let $\{b_1,\ldots,b_n\}$
be an $R$-basis of $M$, and consider the $R$-linear map $\bigwedge^k(M^*)\to (\bigwedge^k(M))^*$ mapping
the basis element $b_{i_1}^*\wedge\cdots\wedge b_{i_k}^*$ of $\bigwedge^k(M^*)$ to the
basis element $(b_{i_1}\wedge\cdots\wedge b_{i_k})^*$ of $(\bigwedge^k(M))^*$, for $1\leq i_1<\cdots<i_k\leq n$.
This map is obviously bijective, and it is routine to check that it is also an $RG$-homomorphism.

\smallskip

(d)\, Consider principal ideal domains $S$ and $R$ and a unitary ring homomorphism $\rho:R\to S$. Let 
further $M$ be an $R$-lattice with $R$-basis $\{b_1,\ldots,b_n\}$. In analogy to \ref{noth change rings}(a), we obtain an 
isomorphism of $S$-lattices $S\otimes_R \bigwedge^k(M)\to \bigwedge^k(S\otimes_R M)$ mapping
the basis element $1\otimes (b_{i_1}\wedge\cdots\wedge b_{i_k})$ to the basis element 
$(1\otimes b_{i_1})\wedge\cdots\wedge (1\otimes b_{i_k})$, for $1\leq i_1<\cdots <i_k\leq n$. If $M$ is
an $RG$-lattice, then this yields an isomorphism of $SG$-lattices. As an immediate consequence, we mention the following:
\end{noth} 

\begin{prop}\label{prop exterior forms}
Let $R$ be a principal ideal domain and $K$ its field of fractions. Let further $V$ be a $KG$-module, and let $M$ be
an $R$-form of $V$. For every $k\in\NN$, the exterior power $\bigwedge^k(M)$ is then an $R$-form of the $KG$-module
$\bigwedge^k(V)$.
\end{prop}

\begin{prop}\label{prop exterior}
Let $M$ and $N$ be $R$-lattices, and let $k\in\NN$. Moreover, let $\phi:M\to N$ be an $R$-linear map.

\smallskip

{\rm (a)}\, If $\phi$ is injective (respectively, surjective), then also the $R$-linear map
$\bigwedge^k(\phi):\bigwedge^k(M)\to \bigwedge^k(N)$ is injective (respectively, surjective).

\smallskip

{\rm (b)}\, Suppose that $N = M$, and let $\rk_R(M) = n\in\NN$. Then one has $\det(\bigwedge^k(\phi))=\det(\phi)^{\binom{n-1}{k-1}}$.

\smallskip

{\rm (c)}\, Suppose that $\rk_R(M)=\rk_R(N)=n\in\NN$. Suppose further that $\{v_1,\ldots,v_n\}$ is an $R$-basis of  $M$ and $\{w_1,\ldots,w_n\}$ is
an $R$-basis of $N$ such that the matrix of $\phi$ with respect to these bases is a diagonal matrix with diagonal entries
$d_1,\ldots,d_n\in R$. Then, with respect to the $R$-bases $\{v_{i_1}\wedge\cdots\wedge v_{i_k}: 1\leq i_1<\ldots <i_k\leq n\}$
and $\{w_{i_1}\wedge\cdots\wedge w_{i_k}: 1\leq i_1<\ldots <i_k\leq n\}$ of $\bigwedge^k(M)$ and $\bigwedge^k(N)$, respectively, 
the $R$-linear map $\bigwedge^k(\phi)$ is represented by a diagonal matrix with diagonal entries
$d_{i_1}\dotsm d_{i_k}$, for $1\leq i_1<\ldots,i_k\leq n$.
\end{prop}

\begin{proof}
Assertion (a) can be found in \cite[III,\S7, no. 8--9]{Bourbaki1970}, assertion (b) can be found in \cite[\S 38]{Aitken1939}. Assertion (c) is 
clear.
\end{proof}

In consequence of Proposition~\ref{prop exterior}(a), whenever we have $RG$-lattices $M$ and $N$ with $N\subseteq M$, we
may view $\bigwedge^k(N)$ as an $RG$-sublattice of $\bigwedge^k(M)$, for every $k\in\NN_0$. 
The next result will be important in Section~\ref{sec p odd}.

\begin{lemma}\label{lemma exterior index}
%Suppose that $R$ is residually finite. 
Suppose that all proper factor rings of $R$ are finite. Let $M$ and $N$ be $RG$-lattices of rank $n\in\NN$ such that $N\subseteq M$. 
For $k\in\NN$, one then has
$$\left[\bigwedge^k(M):\bigwedge^k(N)\right]=[M:N]^{\binom{n-1}{k-1}}\,.$$
\end{lemma}

\begin{proof}
  Let $\phi$ be an $R$-endomorphism of $M$ with $\phi(M) = N$. Suppose that $\{v_1,\dotsc,v_n\}$ is an $R$-basis of $M$.
  If for $i\in\{1,\ldots,n\}$, we set $w_i:= \phi(v_i)$. Then $\{w_1,\dotsc,w_n\}$ is an $R$-basis of $N$, since $\phi(M) = N$.
  As $\bigwedge^k(\phi)(v_{i_1} \wedge \dotsb \wedge v_{i_n}) = w_{i_1} \wedge
  \dotsb \wedge w_{i_n}$, for $1 \leq
  i_1 < \dotsb < i_n \leq n$, and since $\{w_{i_1} \wedge \dotsb \wedge w_{i_n}:1 \leq i_1 < \dotsb < i_n \leq n\}$ is an $R$-basis of $\bigwedge^k(N)$, we have that $\bigwedge^k(\phi)$ is an $R$-endomorphism of $\bigwedge^k(M)$ with
  $\bigwedge^k(\phi)(\bigwedge^k(M)) = \bigwedge^k(N)$.  The assertion of the lemma is now a
  consequence of Proposition~\ref{prop det} and Proposition~\ref{prop
    exterior}(b).
\end{proof}

The next theorem will be crucial in the proof of our main result in Section~\ref{sec p odd}, but should also be of independent interest. We do not expect Theorem~\ref{thm exterior forms iso} to be new. However, we did not find an appropriate reference in the literature, and thus provide a proof here.

\begin{thm}\label{thm exterior forms iso}
Let $R$ be a principal ideal domain an $K$ its field of fractions. Moreover, let
$V$ be an absolutely simple $KG$-module with $R$-forms $M$ and $N$. Suppose that $k\in\{1,\ldots,\dim_K(V)-1\}$ is such that
also $\bigwedge^k(V)$ is an absolutely simple $KG$-module. Then one has $M\cong N$ if and only if $\bigwedge^k(M)\cong \bigwedge^k(N)$, as $RG$-lattices.
\end{thm}

\begin{proof}
By assumption, the $K$-vector spaces $\End_{KG}(V)$ and $\End_{KG}(\bigwedge^k(V))$ are of dimension one.
Since $K$ is a torsion-free $R$-module, we further have $K$-vector space isomorphisms
$$K \otimes_R \Hom_{RG}(M, N) \cong \Hom_{KG}(K\otimes_R M, K\otimes_R N) \cong \End_{KG}(V)$$ and 
$$K \otimes_R \Hom_{RG}\left(\bigwedge^k(M), \bigwedge^k(N)\right) \cong \End_{KG}\left(\bigwedge^k(V)\right)\,,$$
by \cite[Theorem 1.11.12]{Nagao1989} and \ref{noth exterior}(c). Thus $\Hom_{RG}(M,N)$ and $\Hom_{RG}(\bigwedge^k(M),\bigwedge^k(N))$
are $R$-lattices of rank one each. Suppose that $\phi:\bigwedge^k(M)\to \bigwedge^k(N)$ is an $RG$-isomorphism.
By Lemma~\ref{lem:hom1}, $\phi$ spans $\Hom_{RG}(\bigwedge^k(M),\bigwedge^k(N))$ as an $R$-module. 

Now let $\{\psi\}$ be an $R$-basis of $\Hom_{RG}(M,N)$. We shall show that $\psi$ is an $RG$-isomorphism.
In order to do so, it suffices to verify that $\psi$ is an $R$-isomorphism. % or, equivalently, that $\det(\psi)\in R^\times$. 

Since $\phi$ spans $\Hom_{RG}(\bigwedge^k(M),\bigwedge^k(N))$, there is some $r\in R$ such that $\bigwedge^k(\psi)=r\phi$.
By appealing to the Smith normal form, we may choose $R$-bases of $M$ and $N$ in such a way that the matrix 
$A\in\Mat_{n\times n}(R)$ representing $\psi$ with respect to these bases is diagonal, say $A=\mathrm{diag}(d_1,\ldots,d_n)$.
Since $\psi$ also induces a non-zero $KG$-homomorphism $KM\to KN$ and $KM\cong V\cong KN$ is simple, $\psi$ actually
induces a $KG$-isomorphism $KM\to KN$. Thus we must have $d_i\neq 0$, for
all $i\in\{1,\ldots,n\}$. 

By Proposition~\ref{prop exterior}(c), the $R$-linear map $\bigwedge^k(\psi)$ can be represented by
a diagonal matrix with diagonal entries $d_{i_1}\cdots d_{i_k}$, for $1\leq i_1<\cdots <i_k\leq n$. Hence also $\phi$
can be represented by a diagonal matrix, say $D:=\mathrm{diag}(u_1,\ldots,u_s)\in\Mat_{s\times s}(R)$, where $s:=\binom{n}{k}$. 
Since $\phi$ is an $RG$-isomorphism, we have $u_1\cdots u_s=\det(D)\in R^\times$, thus $u_1,\ldots,u_s\in R^\times$.
Since $r\phi=\bigwedge^k(\psi)$, we may suppose that $d_1\cdots d_k=r u_1$. Moreover, for every
$j\in\{k+1,\ldots,n\}$ and every $i\in\{1,\ldots,k\}$, the element $d_1\cdots d_{i-1}d_j d_{i+1}\cdots d_k$ is a diagonal
entry of the fixed matrix representing $\bigwedge^k(\psi)$, so that $d_1\cdots d_{i-1}d_j d_{i+1}\cdots d_k=ru_{s_{i,j}}$,
for some $s_{i,j}\in\{1,\ldots,s\}$. This in turns implies
$$\frac{d_j}{d_i} = \frac{d_1 \dotsm d_{i-1}d_j d_{i+1} \dotsm d_k}{d_1\dotsm d_k}= \frac{r u_{s_{i,j}}}{r u_1} \in R^\times\,,$$
for $j\in\{k+1,\ldots,n\}$ and $i\in\{1,\ldots,k\}$. 
This shows that any two entries of $A$ only differ by a unit in $R$. Thus $A/d_1\in\Mat_{n\times n}(R)$.
% and therefore
%the $KG$-homomorphism $\psi/d_1\in \Hom_{KG}(KM,KN)$ is contained in $\Hom_R(M,N)$. 
Therefore, the $KG$-homomorphism $\psi/d_1\in \Hom_{KG}(KM,KN)$  is the extension of an element in $\Hom_{RG}(M,N)$.
By Lemma~\ref{lem:hom1}, this
implies $d_1\in R^\times$, and we conclude $\det(A)\in R^\times$, so that $\psi$ is indeed an $R$-isomorphism, thus an $RG$-isomorphism,
as desired.

\smallskip

Conversely, if $M\cong N$, then clearly $\bigwedge^k(M)\cong \bigwedge^k(N)$, by Proposition~\ref{prop exterior}(a).
\end{proof}

%%%%%%%%%%%%%%%%%%%%%%%%%%%%%%%%%%%%%%%%%%%%%%%%%%%%%%%%%%%%%

\section{Specht lattices}\label{sec specht}

Throughout this section, let $n\in\NN$. The symmetric group of degree $n$ will be denoted by $\mathfrak{S}_n$.
For background on the representation theory of symmetric groups and the known results stated below, we refer
the reader to \cite{James1978,Sagan2001}.

By $\mathcal{P}(n)$ we denote the set of partitions of $n$. If $p$ is a prime number and $\lambda\in\mathcal{P}(n)$ is
such that each non-zero part of $\lambda$ occurs at most $p-1$ times, then $\lambda$ is called {\sl $p$-regular}.

\begin{noth}\label{noth tableaux}{\bf Young tableaux and tabloids.}\,
(a)\,  Given $\lambda=(\lambda_1,\ldots,\lambda_l)\in\mathcal{P}(n)$, let $t$ be a 
{\sl $\lambda$-tableau}; recall that $t$ is obtained by taking the Young diagram
$$[\lambda]:=\{(i,j)\in\NN\times\NN: 1\leq i\leq l,\, 1\leq j\leq \lambda_i\}$$
and replacing each node bijectively by a number in $\{1,\ldots,n\}$.

The {\sl conjugate partition} $\lambda'$ is the partition of $n$ whose Young diagram is obtained by transposing $[\lambda]$. For every
$\lambda$-tableau $t$, we thus obtain a $\lambda'$-tableau $t'$ by transposing $t$.

A $\lambda$-tableau $t$ is called {\sl standard} if its entries strictly increase along every row from left to right, and down 
every column from top to bottom.

\smallskip

(b)\, Let $\lambda=(\lambda_1,\ldots,\lambda_l)\in\mathcal{P}(n)$. The symmetric group $\mathfrak{S}_n$ acts naturally
on the set of $\lambda$-tableaux. For every $\lambda$-tableau $t$, one denotes by $R_t$ and $C_t$ its {\sl row stabilizer} and {\sl column stabilizer}, respectively. That is, $R_t$ contains the permutations in $\mathfrak{S}_n$ fixing the rows of $t$ setwise, and
$C_t$ contains the permutations fixing the columns of $t$ setwise. For every $\pi\in \mathfrak{S}_n$, one obviously has
$R_{\pi t}=\pi R_t \pi^{-1}$ as well as $C_{\pi t}=\pi C_t \pi^{-1}$. Furthermore, $C_t\cap R_t=\{1\}$.

For every $\lambda$-tableau $t$, one defines
$$\rho_t:=\sum_{\sigma\in R_t}\sigma \in \ZZ \mathfrak{S}_n\quad \text{ and }\quad \kappa_t:=\sum_{\sigma\in C_t}\sgn(\sigma) \sigma \in \ZZ \mathfrak{S}_n\,.$$

(c)\, The set $\{\sigma\cdot t:\sigma\in R_t\}$ is called the {\sl $\lambda$-tabloid corresponding to $t$}, and will be denoted by 
$\{t\}$. 
If $\{t\}$ contains a standard $\lambda$-tableau, then $\{t\}$ is called a {\sl standard $\lambda$-tabloid}.
The set of $\lambda$-tabloids will be denoted by $\mathcal{T}(\lambda)$.
%The $\mathfrak{S}_n$-action on the set of $\lambda$-tableaux also induces an $\mathfrak{S}_n$-action on the set
%of $\lambda$-tabloids.
\end{noth}

\begin{noth}\label{noth Young Specht}{\bf Young permutation modules and Specht modules.}\,
Let $\lambda\in\mathcal{P}(n)$.

(a)\, The $\mathfrak{S}_n$-action on the set of $\lambda$-tableaux induces a transitive $\mathfrak{S}_n$-action on the
set $\mathcal{T}(\lambda)$ of $\lambda$-tabloids. The resulting permutation $\ZZ \mathfrak{S}_n$-lattice will be denoted
by $M^\lambda_\ZZ$, and is called a {\sl Young permutation $\ZZ\mathfrak{S}_n$-lattice}. We shall call $\mathcal{T}(\lambda)$ the
{\sl standard basis} of $M^\lambda_\ZZ$.

For every $\lambda$-tableau $t$, one defines the corresponding {\sl $\lambda$-polytabloid}
$$e_t:=\kappa_t\cdot\{t\}=\sum_{\sigma\in C_t} \sgn(\sigma) \{\sigma t\}\in M^\lambda_\ZZ\,.$$
If $t$ is a standard $\lambda$-tableau, then $e_t$ is called a {\sl standard $\lambda$-polytabloid}.
The definition of $e_t$ depends on the tableau $t$. However, if $\pi\in\mathfrak{S}_n$, then
one has $\pi\cdot e_t=e_{\pi t}$.
This shows that if $s$ is also a $\lambda$-tableau, then the
cyclic $\ZZ\mathfrak{S}_n$-sublattices of $M^\lambda_\ZZ$ generated by $e_t$ and $e_s$, respectively, coincide; one
calls $_{\ZZ\mathfrak{S}_n}\langle e_t\rangle\subseteq M^\lambda_\ZZ$ the {\sl Specht  $\ZZ\mathfrak{S}_n$-lattice} corresponding to $\lambda$ and denotes it by $S^\lambda_\ZZ$. In consequence of \cite[Section 8]{James1978}, the standard $\lambda$-polytabloids
form a $\ZZ$-basis of $S^\lambda_\ZZ$.

\smallskip

(b)\,  Let $R$ be any principal ideal domain, which is naturally a $(\ZZ,\ZZ)$-bimodule. The $R\mathfrak{S}_n$-lattices
  $R\otimes_\ZZ M^\lambda_\ZZ$ and  $R\otimes_\ZZ S^\lambda_\ZZ$ will be denoted by $M^\lambda_R$ and $S^\lambda_R$, respectively. 
  Let $\iota:S^\lambda_\ZZ\to M^\lambda_\ZZ$ be the inclusion map.
  Then the $R\mathfrak{S}_n$-homomorphism $\id_R\otimes \iota: S^\lambda_R\to M^\lambda_R$ is always injective, independently of $R$. 
  To see this, write $e_t$, for  a standard $\lambda$-tableau $t$, as a $\ZZ$-linear combination of tabloids. Then $\{t\}$ occurs with coefficient 1, and every other
 tabloid occurring is strictly smaller than $\{t\}$  in the total order on tabloids defined in \cite[Definition 3.10]{James1978}; see \cite[Lemma 8.3]{James1978}. Moreover, each tabloid occurring in $e_t$ has coefficient 1 or $-1$. For $\lambda\in\mathcal{P}(n)$, one calls $S^\lambda_R$ the {\sl Specht $R\mathfrak{S}_n$-lattice labelled by $\lambda$.}
 
\smallskip

(c)\, By $\sgn$ we denote the $\ZZ \mathfrak S_n$-lattice of rank one on which $\sigma \in \mathfrak S_n$ acts by multiplication with $\sgn(\sigma)$.
Then, for every $\lambda\in\mathcal{P}(n)$, there is an isomorphism of $\ZZ \mathfrak S_n$-lattices $S^{\lambda'}_\ZZ \otimes \sgn \cong (S^\lambda_\ZZ)^\ast$, which by extension of scalars induces an isomorphism
$S^{\lambda'}_R \otimes \sgn_R \cong (S^\lambda_R)^\ast$ for any principal ideal domain $R$; for a proof, see 
Proposition~\ref{prop dual sign}.
\end{noth}

\begin{rem}\label{rem specht basis}
In consequence of \ref{noth tableaux}(b), we may view the Specht $R\mathfrak{S}_n$-lattice $S^\lambda_R$ as an $R\mathfrak{S}_n$-sublattice of $M^\lambda_R$ and the set of standard $\lambda$-polytabloids as an $R$-basis
of $S^\lambda_R$, for every $\lambda\in\mathcal{P}(n)$. As well, $S^\lambda_R={}_{R\mathfrak{S}_n}\langle e_t\rangle$,
for every $\lambda$-polytabloid $t$.
\end{rem}

In later sections it will be important to investigate the structure of Specht lattices over various coefficient rings and fields.
To this end, we mention the following well-known properties of Specht modules over fields.

\begin{thm}\label{thm Specht properties}
Let $F$ be a field.

\smallskip

{\rm (a)}\, If $\mathrm{char}(F)=0$ then the Specht $F\mathfrak{S}_n$-modules $S^\lambda_F$, for $\lambda\in\mathcal{P}(n)$,
yield a set of representatives of the isomorphism classes of absolutely simple $F\mathfrak{S}_n$-modules.
Moreover, every Specht $F\mathfrak{S}_n$-module is self-dual.

\smallskip

{\rm (b)}\, If $\mathrm{char}(F)=p>0$ and if $\lambda\in\mathcal{P}(n)$ is $p$-regular, then
the Specht module $S^\lambda_F$ has a unique simple quotient module $D^\lambda_F$.
If $\lambda$ varies over the set of $p$-regular partitions of $n$, then $D^\lambda_F$ varies over a set of representatives of
the isomorphism classes of absolutely simple $F\mathfrak{S}_n$-modules. Moreover, every
simple $F\mathfrak{S}_n$-module is self-dual.

\smallskip

{\rm (c)}\, If $\mathrm{char}(F)\geq 3$, then the Specht $F\mathfrak{S}_n$-modules $S^\lambda_F$, for
$\lambda\in\mathcal{P}(n)$, are pairwise non-isomorphic and absolutely indecomposable. Moreover, for
$\lambda\in\mathfrak{S}_n$, one has $\End_{F\mathfrak{S}_n}(S^\lambda_F)\cong F$.

\smallskip

{\rm (d)}\, If $\mathrm{char}(F)=2$ and if $\lambda\in\mathcal{P}(n)$ is $2$-regular, then one has
$\End_{F\mathfrak{S}_n}(S^\lambda_F)\cong F$; in particular, $S^\lambda_F$ is then absolutely indecomposable.
\end{thm}

\begin{proof}
Assertions (a) and (b) can be found in \cite[Theorem 4.12, Theorem 11.5]{James1978}. Assertions (c) and (d) follow from
\cite[Corollary 13.17]{James1978}.
\end{proof}

It should be emphasized that Specht modules over fields of characteristic 2 are, in general, not indecomposable; the first examples
of decomposable Specht modules can be found in work of G. Murphy \cite{Murphy1980}. 
This lack of knowledge also causes problems when trying to determine the isomorphism classes of $\ZZ$-forms
of Specht $\QQ \mathfrak{S}_n$-modules. We shall come back to this at the end of Section~\ref{sec p 2}.
%We shall come back to this below.

\medskip

Next we shall focus on Specht lattices labelled by hook partitions, that is, partitions of the form $(n-k,1^k)$, for
$k\in\{0,\ldots,n-1\}$. As before, for every $\lambda\in\mathcal{P}(n)$ and every
principal ideal domain $R$, we consider the Young permutation $R\mathfrak{S}_n$-lattice $M^\lambda_R$ with its standard
$R$-basis $\mathcal{T}(\lambda)$.

\begin{noth}{\bf Hook partitions and exterior powers.}\,\label{noth hooks} Let $n\geq 2$.

(a)\, Consider the partition $(n-1,1)$ of $n$. Each $(n-1,1)$-tabloid $\{t\}$ is uniquely determined by the entry in the second
row of any $(n-1,1)$-tableau in $\{t\}$. In this way, one identifies the set $\mathcal{T}((n-1,1))$ with the set $\{\b{1},\ldots,\b{n}\}$, and
the $\mathfrak{S}_n$-action on $\mathcal{T}((n-1,1))$ corresponds just to the natural $\mathfrak{S}_n$-action on $\{\b{1},\ldots,\b{n}\}$.

Similarly, every $(n-1,1)$-polytabloid is uniquely determined by the first column of the underlying tableau. If $i,j\in\{1,\ldots,n\}$ are such
that $i\neq j$ and if $t$ is any $(n-1,1)$-tableau with first column entries $i$ and $j$, then we shall denote the corresponding 
polytabloid by $e(i,j)$. Note that $e(i,j)=\b{j}-\b{i}\in M^{(n-1,1)}_R$. Note further that $e(1,2),\ldots, e(1,n)$ are precisely the standard $(n-1,1)$-polytabloids.

One calls $M^{(n-1,1)}_R$ the {\sl natural permutation $R\mathfrak{S}_n$-lattice}, and $S^{(n-1,1)}_R$ the {\sl natural Specht $R\mathfrak{S}_n$-lattice}.

\smallskip

(b)\, Set $M:=M^{(n-1,1)}_R$ and $S:=S^{(n-1,1)}_R$. Let $k\in\{1,\ldots,n\}$. We have the $R\mathfrak{S}_n$-lattice
$\bigwedge^k(M)$ with $R$-basis $\{\b{i_1}\wedge\cdots\wedge \b{i_k}: 1\leq i_1<\cdots <i_k\leq n\}$. By Proposition~\ref{prop exterior}, we can also
regard $\bigwedge^k(S)$ as an $R\mathfrak{S}_n$-sublattice of $\bigwedge^k(M)$, and $\bigwedge^k(S)$ has $R$-basis
$\{e(1,i_1)\wedge\cdots\wedge e(1,i_k): 2\leq i_1<\cdots <i_k\leq n\}$. For $2\leq i_1<\cdots<i_k\leq n$, we set
$$b(i_1,\ldots,i_k):=e(1,i_1)\wedge\cdots\wedge e(1,i_k)=(\b{i_1}-\b{1})\wedge\cdots\wedge (\b{i_k}-\b{1})\,.$$

\smallskip

Now let $k\in\{0,\ldots,n-1\}$, and consider the {\sl hook partition} $(n-k,1^k)$ of $n$. As in the case $k=1$, every $(n-k,1^k)$-polytabloid
is determined by the first-column entries of the underlying tableau. If $i_1,\ldots,i_{k+1}\in\{1,\ldots,n\}$ are pairwise distinct, then we denote
the $(n-k,1^k)$-polytabloid corresponding to the $(n-k,1^k)$-tableau with first-column entries $i_1,\ldots,i_{k+1}$ by
$e(i_1,\ldots,i_{k+1})$. With this convention, the standard $(n-k,1^k)$-polytabloids are precisely those of the form
$e(1,i_1,\ldots,i_k)$, for $2\leq i_1<\cdots <i_k\leq n$. Moreover, the following $R$-linear map defines an $R\mathfrak{S}_n$-isomorphism:
\begin{equation}\label{eqn hook iso}
\bigwedge^k(S^{(n-1,1)}_R)\to S^{(n-k,1^k)}_R\,, \; e(1,i_1)\wedge\cdots\wedge e(1,i_k)\mapsto e(1,i_1,\ldots,i_k)\,,
\end{equation}
where $2\leq i_1<\cdots <i_k\leq n$; for a proof see \cite[Proposition 2.3]{Muller2007}, which is stated under the assumption
that $R$ is a field, but works for every principal ideal domain.
In particular, $S^{(n-k,1^k)}_R$ has $R$-rank $\binom{n-1}{k}$, for $k\in\{0,\ldots,n-1\}$.

It has proved to be very useful to
identify the Specht lattice $S^{(n-k,1^k)}_R$ with the exterior power $\bigwedge^k(S^{(n-1,1)}_R)$ via the isomorphism (\ref{eqn hook iso}), and we shall exploit this repeatedly in the following.
\end{noth}

\begin{rem}\label{rem hook heads}
It is clear from the definition, that $S^{(n)}_R$ is isomorphic to the trivial $R\mathfrak{S}_n$-lattice, and
$S^{(1^n)}_R$ is isomorphic to the sign $R\mathfrak{S}_n$-lattice.

Below we summarize some known results concerning the structure of
hook Specht modules over fields of characteristic $p\geq 3$. 
\end{rem}

\begin{thm}[\protect{\cite[Theorem 23.7, Theorem 24.1]{James1978}}]\label{thm hooks}
Let $F$ be a field of characteristic $p\geq 3$. 

\smallskip

{\rm (a)}\, Suppose that $p\nmid n$. Then the Specht modules $S^{(n-k,1^k)}_F$, for $k\in\{0,\ldots,n-1\}$, are
pairwise non-isomorphic absolutely simple $F\mathfrak{S}_n$-modules.

\smallskip

{\rm (b)}\, Suppose that $p\mid n$. For $k\in\{1,\ldots,n-2\}$, the Specht $F\mathfrak{S}_n$-module $S^{(n-k,1^k)}_F$
has a unique simple submodule $D(k)_F$, and the quotient module $\overline{D}(k)_F:=S^{(n-k,1^k)}_F/D(k)_F$ is also simple and
not isomorphic to $D(k)_F$. 
Moreover, one has $D(1)_F\cong F$, $\overline{D}(k-1)_F\cong D(k)_F$ for $k\in\{2,\ldots,n-2\}$, and $\overline{D}(n-2)_F\cong S^{(1^n)}_F=:D(n-1)_F$.
\end{thm}

As an immediate consequence, we also mention the following, which will be needed in Section~\ref{sec p odd}.

\begin{cor}\label{cor dims simples}
Let $F$ be a field of characteristic $p\geq 3$. For $k\in\{1,\ldots,n-2\}$, the simple $F\mathfrak{S}_n$-module
$D(k)$ in Theorem~\ref{thm hooks}(b) has dimension $\binom{n-2}{k-1}$.
\end{cor}

%%%%%%%%%%%%%%%%%%%%%%%%%%%%%%%%%%%%%%%%%%%%%%%%%%%%%%%%%%%

\section{The case $p\geq 3$}\label{sec p odd}

In this section, let $p$ be an odd prime. We shall determine a set of representatives of the isomorphism classes
of $\ZZ_p$-forms of the Specht $\QQ_p\mathfrak{S}_n$-modules $S^{(n-k,1^k)}_{\QQ_p}$, for
$k\in\{1,\ldots,n-2\}$. For ease of notation, for $k\in\{0,\ldots,n-1\}$, we set
$S(k)_R:=S^{(n-k,1^k)}_R$, for every principal ideal domain $R$ under consideration. Moreover, we
identify $S(k)_R$ with the exterior power $\bigwedge^k(S(1)_R)$ via the $R\mathfrak{S}_n$-isomorphism (\ref{eqn hook iso})
in \ref{noth hooks}.

The aim of this section is to prove Theorem~\ref{thm forms bijection p} below. Together with the results
of Craig \cite{Craig1976} and Plesken \cite[Theorem 5.1]{Plesken1977} concerning the isomorphism classes
of $\ZZ$-forms of $S(1)_{\QQ}$ and $S(n-2)_{\QQ}$, this will enable us to prove Theorem~\ref{thm intro}(a) in Section~\ref{sec proofs}.

\begin{thm}\label{thm forms bijection p}
Let $n \in \NN$ with $n \geq 3$, and let $k \in \{1,\dotsc,n-2\}$. Moreover let $p \geq 3$ be a prime number.
If $L_1,\ldots,L_s$ are representatives of the $\ZZ_p\mathfrak{S}_n$-isomorphism
classes of $\ZZ_p$-forms of $S(1)_{\QQ_p}$, then $\bigwedge^k(L_1),\ldots,\bigwedge^k(L_s)$ are
representatives of the $\ZZ_p\mathfrak{S}_n$-isomorphism
classes of $\ZZ_p$-forms of $S(k)_{\QQ_p}$; in particular, the number of isomorphism classes of $\ZZ_p$-forms
of $S(1)_{\QQ_p}$ and the number of isomorphism classes of $\ZZ_p$-forms
of $S(k)_{\QQ_p}$ coincide.
\end{thm}

As the case $p \nmid n$ will be dealt with by Proposition~\ref{prop notdivides} and Theorem~\ref{thm Specht properties}(a), we shall from now on suppose that $p \mid n$, for the remainder of this section.

\begin{rem}\label{rem possible Loewy}
Let $k\in\{1,\ldots,n-2\}$.
Recall from Theorem~\ref{thm hooks} that the $\FF_p\mathfrak{S}_n$-module $S(k)_{\FF_p}\cong S(k)_{\ZZ_p}/pS(k)_{\ZZ_p}$ is uniserial with Loewy series
$$S(k)_{\FF_p}\sim\begin{bmatrix} D(k+1)\\D(k) \end{bmatrix}\,,$$
where $D(k)\not\cong D(k+1)$. Thus, for every $\ZZ_p$-form $L$ of the $\QQ_p\mathfrak{S}_n$-module $S(k)_{\QQ_p}$, we know
that the $\FF_p\mathfrak{S}_n$-module $L/pL$ has composition factors $D(k)$ and $D(k+1)$ as well; in particular, there are the following three possibilities for
the Loewy series of $L/pL$:
$$L/pL\sim\begin{bmatrix} D(k+1)\\D(k) \end{bmatrix}\quad \text{ or }\quad  L/pL\sim\begin{bmatrix} D(k)\\D(k+1) \end{bmatrix}\quad\text{ or }\quad L/pL\cong D(k)\oplus D(k+1)\,;$$
in particular, by Proposition~\ref{prop p max sublattices}, $L$ has either one or two $\ZZ_p\mathfrak{S}_n$-sublattices $N$ such that $pL\subsetneq N\subsetneq L$, and these are
precisely the maximal  $\ZZ_p\mathfrak{S}_n$-sublattices of $L$. 
\end{rem}

As an application of Lemma~\ref{lemma possible Loewy} we now obtain:

\begin{prop}\label{prop possible Loewy}
Let $k\in\{1,\ldots,n-2\}$, and let $p\mid n$.
Let $N$ be a $\ZZ_p$-form of $S(1)_{\QQ_p}$. Then $\bigwedge^k(N)$ is a $\ZZ_p$-form of $S(k)_{\QQ_p}$. Moreover
\smallskip

{\rm (a)}\, the following are equivalent

\quad {\rm (i)}\, $N/pN\sim\begin{bmatrix} D(2)\\D(1) \end{bmatrix}$;

\quad {\rm (ii)}\, $N\cong S(1)_{\ZZ_p}$;

\quad {\rm (iii)}\, $ \bigwedge^k(N)\cong S(k)_{\ZZ_p}$;

\quad {\rm (iv)}\, $\bigwedge^k(N)/p\bigwedge^k(N)\sim\begin{bmatrix} D(k+1)\\D(k) \end{bmatrix}$;

\smallskip

{\rm (b)}\, the following are equivalent

\quad {\rm (i)}\, $N/pN\sim\begin{bmatrix} D(1)\\D(2) \end{bmatrix}$;

\quad {\rm (ii)}\,  $N\cong S(1)_{\ZZ_p}^*$;

\quad {\rm (iii)}\, $ \bigwedge^k(N)\cong S(k)_{\ZZ_p}^*$;

\quad {\rm (iv)}\,  $\bigwedge^k(N)/p\bigwedge^k(N)\sim\begin{bmatrix} D(k)\\D(k+1) \end{bmatrix}$;

\smallskip

{\rm (c)}\, one has
$$N/pN\cong D(1)\oplus D(2)\Leftrightarrow  \bigwedge^k(N)/p\bigwedge^k(N)\cong D(k)\oplus D(k+1)\,.$$
\end{prop}

\begin{proof}
Let $N$ be a $\ZZ_p$-form of $S(1)_{\QQ_p}$. By Proposition~\ref{prop exterior forms}, we  
know that $L:=\bigwedge^k(N)$ is a $\ZZ_p$-form of $\bigwedge^k(S(1)_{\QQ_p})\cong S(k)_{\QQ_p}$.

To prove (a), suppose that $N/pN$ has Loewy series $\begin{bmatrix} D(2)\\D(1) \end{bmatrix}$. Then $N\cong S(1)_{\ZZ_p}$, by
Lemma~\ref{lemma possible Loewy}. This implies $\bigwedge^k(N)\cong S(k)_{\ZZ_p}$, by \ref{noth hooks}(b), and then
$\bigwedge^k(N)/p\bigwedge^k(N)$ has Loewy series $\begin{bmatrix} D(k+1)\\ D(k)\end{bmatrix}$, by Theorem~\ref{thm hooks}.
Conversely, if $\bigwedge^k(N)/p\bigwedge^k(N)$ has Loewy series $\begin{bmatrix} D(k+1)\\ D(k)\end{bmatrix}$,
then $\bigwedge^k(N)\cong S(k)_{\ZZ_p}$, by Lemma~\ref{lemma possible Loewy} again. 
So $\bigwedge^k(N)\cong\bigwedge^k(S(1)_{\ZZ_p})$, which forces $N\cong S(1)_{\ZZ_p}$, by Theorem~\ref{thm exterior forms iso}.

Analogously one obtains (b).

Lastly suppose that $N/pN$ is semisimple, that is, $N/pN\cong D(1)\oplus D(2)$.
Then $S(1)_{\ZZ_p}\not\cong N\not\cong S(1)_{\ZZ_p}^*$, by Lemma~\ref{lemma possible Loewy}. Hence
$S(k)_{\ZZ_p}\cong \bigwedge^k(S(1)_{\ZZ_p})\not\cong \bigwedge^k(N)\not\cong \bigwedge^k(S(1)_{\ZZ_p}^*)\cong \bigwedge^k(S(1)_{\ZZ_p})^*\cong S(k)_{\ZZ_p}^*$, by Theorem~\ref{thm exterior forms iso} and \ref{noth hooks}(b).
Therefore, Lemma~\ref{lemma possible Loewy} and Remark~\ref{rem possible Loewy} show that $\bigwedge^k(N)/p\bigwedge^k(N)\cong D(k)\oplus D(k+1)$.
Conversely, if $\bigwedge^k(N)/p\bigwedge^k(N)\cong D(k)\oplus D(k+1)$ then we 
get $S(k)_{\ZZ_p}\not\cong \bigwedge^k(N)\not\cong S(k)_{\ZZ_p}^*$, thus
$S(1)_{\ZZ_p}\not\cong N\not\cong S(1)_{\ZZ_p}^*$, by Theorem~\ref{thm exterior forms iso}, and then
$N/pN\cong D(1)\oplus D(2)$, by Lemma~\ref{lemma possible Loewy}.
\end{proof}

\begin{nota}\label{nota forms divide}
Suppose that $G$ is any finite group and $M$ is a $\ZZ_p G$-lattice. 
Suppose further that $N$ is a $\ZZ_pG$-sublattice of $M$ such that $N\subseteq p^iM$, for some $i\geq 1$. 
Then $\{p^{-i}x: x \in N\}\subseteq \QQ_pM$
is also a $\ZZ_pG$-sublattice of $M$ isomorphic to $N$. We shall use the notation $N/p^i:=\{p^{-i}x: x \in N\}$.
Note that $N/p^i\cong p^i(N/p^i)= N$ as $\ZZ_pG$-lattices.
\end{nota}

In the proof of the next proposition we shall have to assume $k<n-3$. However,
this is not really a restriction, since the case $k=n-2$ has been dealt with anyway by Plesken and Craig, as we shall 
see in the proof of Theorem~\ref{thm forms bijection p}.

\begin{prop}\label{prop hook forms}
Let $p\mid n$, and let $k\in\{1,\ldots,n-3\}$.
Let $M$ be a $\ZZ_p$-form of $S(1)_{\QQ_p}$.

\smallskip

{\rm (a)}\,   If $M$ has only one maximal $\ZZ_p\mathfrak{S}_n$-sublattice $N$ and if $N$ has index $p$ in $M$, then 
$\bigwedge^k(N)$ is the only maximal $\ZZ_p\mathfrak{S}_n$-sublattice of $\bigwedge^k(M)$.

\smallskip

{\rm (b)}\,   If $M$ has only one maximal $\ZZ_p\mathfrak{S}_n$-sublattice $N$ and if $N$ has index $p^{n-2}$ in $M$, then $\bigwedge^k(N)/p^{k-1}$ is the only maximal $\ZZ_p\mathfrak{S}_n$-sublattice of $\bigwedge^k(M)$.

\smallskip

{\rm (c)}\,  If $M$ has two maximal $\ZZ_p\mathfrak{S}_n$-sublattices $N_1$ and $N_2$ of index $p$ and $p^{n-2}$, respectively, then $\bigwedge^k(N_1)$ and $\bigwedge^k(N_2)/p^{k-1}$ are the maximal $\ZZ_p\mathfrak{S}_n$-sublattices of $\bigwedge^k(M)$.
Moreover $\bigwedge^k(N_1) \neq \bigwedge^k(N_2)/p^{k-1}$.

\smallskip
\noindent
In particular, every maximal $\ZZ_p\mathfrak{S}_n$-sublattice of $\bigwedge^k(M)$ is isomorphic to $\bigwedge^k(N)$ for some maximal $\ZZ_p\mathfrak{S}_n$-sublattice $N$ of $M$.
\end{prop}

\begin{proof}
(a)\, Suppose that $M$ has a unique maximal sublattice $N$. Then, by Proposition~\ref{prop p max sublattices}, the $\FF_p\mathfrak{S}_n$-module $M/pM$ has a simple head. If $N$ has index $p$ in $M$ then $\dim_{\FF_p}(\Hd(M/pM))=1$, by Proposition~\ref{prop p max sublattices}.
Thus Proposition~\ref{prop possible Loewy} implies $M\cong S(1)_{\ZZ_p}^*$ and $\bigwedge^k(M)\cong S(k)_{\ZZ_p}^*$. So by
Theorem~\ref{thm hooks} and
Proposition~\ref{prop p max sublattices} again, $\bigwedge^k(M)$ also has a unique maximal sublattice, which must have index
$p^{\dim_{\FF_p}(D(k))}=p^{\binom{n-2}{k-1}}$; in particular, there is a unique sublattice
of $\bigwedge^k(M)$ of index $p^{\binom{n-2}{k-1}}$. Moreover, by Lemma~\ref{lemma exterior index},
$\bigwedge^k(N)$ is a $\ZZ_p\mathfrak{S}_n$-sublattice of $\bigwedge^k(M)$ of index $p^{\binom{n-2}{k-1}}$ as well. This proves (a).

\smallskip

(b)\, In analogy to (a) we deduce that in this case we must have $M\cong S(1)_{\ZZ_p}$, $\bigwedge^k(M)\cong S(k)_{\ZZ_p}$,
and $\bigwedge^k(M)$ has a unique maximal sublattice $\tilde{N}$. Moreover $\tilde{N}$ has index $p^{\dim_{\FF_p}(D(k+1))}=p^{\binom{n-2}{k}}$.

We now want to show that $p^k \bigwedge^k(M) \subseteq \bigwedge^k(N) \subseteq p^{k-1} \bigwedge^k(M)$.
As $M/N$ is an $\FF_p$-vector space of dimension $n - 2$, by Lemma~\ref{lemma aligned bases}, we can find $\ZZ_p$-bases 
$\{v_1,\dotsc,v_{n-1}\}$ and $\{w_1,\dotsc,w_{n-1}\}$ of $M$ and $N$, respectively, such that
 $w_1 = v_1$ and $w_i = p v_i$, for $2 \leq i \leq n -1$.
Thus, for $1 \leq i_1 < \dotsb < i_k \leq n - 1$, we have
\[ p^k( v_{i_1} \wedge \dotsb \wedge v_{i_k} ) = \begin{cases} w_{i_1} \wedge \dotsb \wedge w_{i_k}, \quad &\text{ if $i_1\neq 1$} \\ p (w_{i_1} \wedge \dotsb \wedge w_{i_k}), \quad & \text{ otherwise,} \end{cases} \]
and
\[ w_{i_1} \wedge \dotsb \wedge w_{i_k} = \begin{cases} p^{k-1} (v_{i_1} \wedge \dotsb \wedge v_{i_k}), \quad &\text{ if $i_1=1$} \\
p^k (v_{i_1} \wedge \dotsb \wedge v_{i_k}), \quad &\text{ otherwise\,;} \end{cases} \]
in particular, $p^k( v_{i_1} \wedge \dotsb \wedge v_{i_k} )\in\bigwedge^k(N)$ and $w_{i_1} \wedge \dotsb \wedge w_{i_k}\in p^{k-1}\bigwedge^k(M)$.
Hence we obtain $p^k \bigwedge^k(M) \subseteq \bigwedge^k(N) \subseteq p^{k-1} \bigwedge^k(M)$, that is,
\[ p \bigwedge^k(M) \subseteq \frac{\bigwedge^k(N)}{p^{k-1}} \subseteq \bigwedge^k(M). \]
By Lemma~\ref{lemma exterior index}, the index of $\bigwedge^k(N)/p^{k-1}$ in $\bigwedge^k(M)$ is
\[ \frac{[ \bigwedge^k(M) : \bigwedge^k(N) ]}{(p^{k-1})^{{n-1}\choose k}} = p^{(n-2){{n-2}\choose {k-1}} - (k - 1) {{n-1}\choose k}}.\]
A quick calculation reveals that
$${(n-2){{n-2}\choose {k-1}} - (k - 1) {{n-1}\choose k}} = {{n-2} \choose k}\,,$$
and we conclude that $\tilde N$ and $\bigwedge^k(N)/p^{k-1}$ have the same index in $\bigwedge^k(M)$.
Since $\tilde N$ is unique with this index, we have $\tilde N = \bigwedge^k(N)/p^{k-1}$.

\smallskip

(c)\, Using Proposition~\ref{prop p max sublattices} and Lemma~\ref{lemma possible Loewy} we see that $\bigwedge^k(M)$ has 
exactly two maximal $\ZZ_p\mathfrak{S}_n$-sublattices $\tilde N_1$ and $\tilde N_2$. 
If $\tilde N_1$ and $\tilde N_2$ have different index in $\bigwedge^k(M)$, then we may argue as in the proof of (a) and (b) above to show that $\{ \bigwedge^k(N_1), \bigwedge^k(N_2)/p^{k-1} \} = \{\tilde N_1, \tilde N_2\}$.
If $\tilde N_1$ and $\tilde N_2$ have the same index in $\bigwedge^k(M)$, that is, if ${n - 2 \choose
k-1} = {n-2\choose k}$, then we only get $\{ \bigwedge^k(N_1),
\bigwedge^k(N_2)/p^{k-1} \} \subseteq \{ \tilde N_1,\tilde N_2\}$.  To get
equality, it suffices to show that $\bigwedge^k(N_1) \neq \bigwedge^k(N_2)/p^{k-1}$.

We first apply Lemma~\ref{lemma aligned bases}(a) to get $\ZZ_p$-bases $\{v_1,\dotsc,v_{n-1}\}$ and $\{w_1,\dotsc,w_{n-1}\}$ of $M$ and $N_2$, respectively, as in the proof of~(b).
Since $N_1$ has index $p$ in $M$, we can then apply Lemma~\ref{lemma aligned bases}(b) to find a $\ZZ_p$-basis $\{u_1,\dotsc,u_{n-1}\}$ of $N_1$ such that $u_j = p v_j$ for some $1 \leq j \leq n - 1$, and $u_i = v_i + \lambda_i v_j$ for $i\neq j$ and 
suitable $\lambda_1,\ldots,\lambda_{n-1} \in \ZZ_p$.
Now, since $k<n-2$, we can choose $1 \leq i_1 < \dotsb < i_k \leq n - 1$ with $1, j \not\in \{ i_1,\dotsc,i_k \}$ 
and consider the element
\[ \alpha = u_{i_1} \wedge \dotsb \wedge u_{i_k}
= (v_{i_1} + \lambda_{i_1} v_j) \wedge \dotsb \wedge (v_{i_k} + \lambda_{i_k} v_j)
= (v_{i_1} \wedge \dotsb \wedge v_{i_k}) + w \in\bigwedge^k(N_1)\,, \]
where $w$ does not involve a basis element of the form $v_{i_1} \wedge \dotsb \wedge v_{i_k}$.
To show that $\alpha \notin \bigwedge^k(N_2)/p^{k-1}$ let us assume the contrary.
Then, for $1\leq j_1<\cdots <j_k\leq n-1$, there exists $\lambda_{j_1,\dotsc,j_k} \in \ZZ_p$ such that we can write
\begin{align*}
\alpha &= (v_{i_1} \wedge \dotsb \wedge v_{i_k}) + w
= \sum_{1 \leq j_1 < \dotsb < j_k \leq n} \frac{\lambda_{j_1,\dotsc,j_k}}{p^{k-1}} \cdot w_{j_1}\wedge \dotsb \wedge w_{j_k} \\
&= \sum_{1 < j_1 < \dotsb < j_k \leq n} \lambda_{j_1,\dotsc,j_k} \cdot p \cdot v_{j_1}\wedge \dotsb \wedge v_{j_k} +
\sum_{1 = j_1 < \dotsb < j_k \leq n} \lambda_{j_1,\dotsc,j_k} \cdot v_{j_1}\wedge \dotsb \wedge v_{j_k}\,.
\end{align*}
Comparing the coefficients at $v_{i_1} \wedge \dotsb \wedge v_{i_k}$ we see that 
$1=\lambda_{j_1,\ldots,j_k}\cdot p$, which is not possible.
This completes the proof of (c).
\end{proof}

We are now  in the position to prove Theorem~\ref{thm forms bijection p}.

\begin{proof}[\textsl{\textbf{Proof of Theorem~\ref{thm forms bijection p}}}]
Let first $p \nmid n$. Then, for $k\in\{2,\ldots,n-2\}$, by Proposition~\ref{prop notdivides} and Theorem~\ref{thm hooks}(a), both $S(1)_{\QQ_p}$ and $S(k)_{\QQ_p}$ have only one $\ZZ_p$-form up to isomorphism. Thus the assertion follows in this case.

Now let $p \mid n$ and $k < n - 2$. Let further $L_1,\ldots,L_s$ be representatives of the isomorphism classes of $\ZZ_p$-forms
of $S(1)_{\QQ_p}$. 
The exterior powers $\bigwedge^k(L_1),\ldots,\bigwedge^k(L_s)$
are $\ZZ_p$-forms of $\bigwedge^k(S(1)_{\QQ_p})\cong S(k)_{\QQ_p}$, by Proposition~\ref{prop exterior forms}. 
By Theorem~\ref{thm exterior forms iso} and Theorem~\ref{thm Specht properties}(a),  $\bigwedge^k(L_1),\ldots,\bigwedge^k(L_s)$ are pairwise non-isomorphic.

Conversely, let $L$ be any $\ZZ_p$-form of $S(k)_{\QQ_p}$, and let $M$ be any $\ZZ_p$-form of $S(1)_{\QQ_p}$.
Then, by Remark~\ref{rem forms all in one}, $L$ is isomorphic to a $\ZZ_p\mathfrak{S}_n$-sublattice of the $\ZZ_p$-form 
$\bigwedge^k(M)$ of $S(k)_{\QQ_p}$, and we may thus suppose that 
$L\subseteq \bigwedge^k(M)$. Then there exists an $l\in\NN_0$ and $\ZZ_p\mathfrak{S}_n$-sublattices
$N_1,\ldots,N_l$ of $\bigwedge^k(M)$ such that
$$L=N_l\subsetneq N_{l-1}\subsetneq \cdots \subsetneq N_0=\bigwedge^k(M)$$
and such that $N_{i+1}$ is
maximal in $N_i$, for $i\in\{0,\ldots,l-1\}$. We need to show that $L\cong \bigwedge^k(L_j)$, for some $j\in\{1,\ldots,s\}$. To do
so we argue by induction in $l$. If $l=0$ then the assertion is trivial. So let $l>0$. By induction,
for every $i\in\{0,\ldots,l-1\}$, there is some $j_i\in\{1,\ldots,s\}$ such that $N_i\cong \bigwedge^k(L_{j_i})$. 
We may suppose that $N_{l-1} =\bigwedge^k(L_{j_{l-1}})$. Thus, by Proposition~\ref{prop hook forms},
there is some $j_l\in\{1,\ldots,s\}$ such that $L=N_l\cong \bigwedge^k(L_{j_l})$
or $L=N_l\cong \bigwedge^k(L_{j_l})/p^{k-1}$. Since $L_{j_l}\cong \bigwedge^k(L_{j_l})/p^{k-1}$, the assertion of the
theorem follows, for $1\leq k\leq n-3$. 
 
 It remains to deal with the case $k=n-2$. But, by Theorem~\ref{thm Specht properties}(a) and Proposition~\ref{prop dual sign} (see also \cite[Theorem 6.7]{James1978}),
 we have $S(n-2)_{\QQ_p}\cong S(1)_{\QQ_p}\otimes \sgn_{\QQ_p}$.
From this we deduce that $L_1\otimes\sgn_{\ZZ_p},\ldots, L_s\otimes \sgn_{\ZZ_p}$
must be representatives of the isomorphism classes of $\ZZ_p$-forms of $S(n-2)_{\QQ_p}$.
On the other hand, by Theorem~\ref{thm exterior forms iso}  and Proposition~\ref{prop exterior forms} again, 
$\bigwedge^k(L_1),\ldots,\bigwedge^k(L_s)$ are pairwise non-isomorphic $\ZZ_p$-forms of $S(n-2)_{\QQ_p}$, hence
have to be representatives of the isomorphism classes of such forms as well.
This completes the proof of the theorem.
\end{proof}

%%%%%%%%%%%%%%%%%%%%%%%%%%%%%%%%%%%%%%%%%%%%%%%%%%%%%%%%%%

\section{The case $p=2$}\label{sec p 2}

In this section we now turn to the case where $p=2$ and $n\geq 4$. We shall determine a set of representatives of the isomorphism classes
of $\ZZ_2$-forms of the Specht $\QQ_2\mathfrak{S}_n$-module $S^{(n-2,1^2)}_{\QQ_2}$, in the case where $n\not\equiv 0\pmod{4}$,
thereby proving Theorem~\ref{thm intro}(b).
We retain the notation introduced in \ref{noth hooks}.
For ease of notation, we shall use the following convention throughout this section: we set $S_R:=S^{(n-2,1^2)}_R$, for 
every principal ideal domain $R$ under consideration. We shall also regard $S_R$ as an $R\mathfrak{S}_n$-sublattice
of $\bigwedge^2(M^{(n-1,1)})$ as in \ref{noth hooks}(b), and work with our
standard $R$-basis $\{b(i_1,i_2): 2\leq i_1<i_2\leq n\}$ of $S_R$, where
$$b(i_1,i_2)=(\b{i_1}-\b{1})\wedge (\b{i_2}-\b{1})=(\b{i_1}\wedge \b{i_2})+(\b{1}\wedge \b{i_1})-(\b{1}\wedge\b{i_2})\,,$$
for $2\leq i_1<i_2\leq n$.

\begin{rem}\label{rem Young modules}
In the proof of the next proposition, we shall use a result of M\"uller and Orlob on 
Young modules in \cite{Muller2011}. Suppose that $\lambda\in\mathcal{P}(n)$, and let $F$ be any field. 
In any given indecomposable direct sum decomposition of the $F\mathfrak{S}_n$-module $M^\lambda_F$, there
is a unique indecomposable direct  summand containing $S^\lambda_F$ as a submodule. This $F\mathfrak{S}_n$-module
is unique up to isomorphism and is called the {\sl Young $F\mathfrak{S}_n$-module} labelled by $\lambda$;  see \cite{Erdmann2001, Grabmeier1985}. It is usually
denoted by $Y^\lambda_F$.
\end{rem}

\begin{prop}\label{prop S in char 2}
Let $F$ be a field of characteristic $2$, and let $n\geq 5$.

\smallskip

{\rm (a)}\, If $n\equiv 1\pmod{4}$, then $S_F$ is uniserial with Loewy series
$$\begin{bmatrix}
F\\
D^{(n-2,2)}_F\\
F\\
\end{bmatrix}\,.$$

\smallskip

{\rm (b)}\, If $n\equiv 3\pmod{4}$, then $S_F\cong F\oplus D^{(n-2,2)}_F$.

\smallskip

{\rm (c)}\, If $n\equiv 2\pmod{4}$, then $S_F$ is uniserial with Loewy series
$$\begin{bmatrix}
F\\
D^{(n-2,2)}_F\\
F\\
D^{(n-1,1)}_F\\
\end{bmatrix}\,.$$

\smallskip

{\rm (d)}\, If $n\equiv 0\pmod{4}$, then $S_F$ is indecomposable with Loewy series
$$\begin{bmatrix}
F\oplus D^{(n-2,2)}_F\\
D^{(n-1,1)}_F
\end{bmatrix}\,.$$
\end{prop}

\begin{proof}
The Specht module $S_F$ is isomorphic to a submodule of the Young module $Y^{(n-2,1^2)}_F$. 
In \cite[Theorem 1.1]{Muller2011} the complete submodule lattice of $Y^{(n-2,1^2)}_F$ has been determined.
From this one obtains that $Y^{(n-2,1^2)}_F$ has  a unique submodule of dimension $\dim_F(S_F)=\binom{n-1}{2}$,
and this module has the claimed Loewy series.
\end{proof}

We now start to construct representatives of the isomorphism classes of $\ZZ_2$-forms of
$S_{\ZZ_2}$ along the lines of Proposition~\ref{prop p max sublattices}, analyzing
the $2$-modular reduction of $S_{\ZZ_2}$. We shall deal with the case where $n$ is odd first,
and mention the following useful observation beforehand.

\begin{lemma}\label{lemma perm basis}
Suppose that $F$ is a field of characteristic $2$, and let $k\in\{1,\ldots,n\}$. Then $\bigwedge^k(M^{(n-1,1)}_F)$ is 
a transitive permutation $F\mathfrak{S}_n$-module and is isomorphic to the induced module
$\ind_{\mathfrak{S}_k\times\mathfrak{S}_{n-k}}^{\mathfrak{S}_n}(F)$. 
Furthermore,
$$\mathcal{T}_k:=\{\b{i_1}\wedge\cdots\wedge \b{i_k}: 1\leq i_1<\cdots <i_k\leq n\}$$
is a transitive permutation basis of $\bigwedge^k(M^{(n-1,1)}_F)$.
\end{lemma}

\begin{proof}
Since $\mathrm{char}(F)=2$, it is clear that the $F$-basis $\mathcal{T}_k$ of $\bigwedge^k(M^{(n-1,1)}_F)$
is a transitive $\mathfrak{S}_n$-set. Moreover, the stabilizer of $\b{1}\wedge\cdots\wedge \b{k}$ in $\mathfrak{S}_n$
is obviously $\mathfrak{S}_k\times\mathfrak{S}_{n-k}$, whence the claim.
\end{proof}

\begin{lemma}\label{lemma U}
Let $n\geq 4$, let $F$ be a field of characteristic $2$, and consider the following $F$-subspaces of $S_{F}$:
\begin{align*}
U_1&:=\left\{\sum_{2\leq i_1<i_2\leq n} \beta(i_1,i_2) b(i_1,i_2): \beta(i_1,i_2)\in F,\, \sum_{2\leq i_1<i_2\leq n} \beta(i_1,i_2)=0\right\}\,,\\
U_2&:=\left\{\beta\sum_{2\leq i_1<i_2\leq n}b(i_1,i_2):\beta\in F\right\}\,.
\end{align*}

\smallskip

{\rm (a)}\, The $F$-vector space $U_1$ is an $F\mathfrak{S}_n$-submodule of $S_{F}$ of codimension $1$.

\smallskip

{\rm (b)}\, If $n$ is odd, then $U_2$ is the unique trivial $F\mathfrak{S}_n$-submodule of $S_{F}$.

\smallskip

{\rm (c)}\, If $n\equiv 1\pmod{4}$, then $\{0\}\subseteq U_2\subseteq U_1\subseteq S_{F}$ is the unique composition series
of $S_F$; in particular, $S_F/U_1\cong F\cong U_2$ and $U_1/U_2\cong D^{(n-2,2)}_F$.

\smallskip

{\rm (d)}\, If $n\equiv 3\pmod{4}$, then $S_F=U_1\oplus U_2$; in particular, $U_1\cong D^{(n-2,2)}_F$.
\end{lemma}

\begin{proof}
Set $M:=\bigwedge^2(M^{(n-1,1)}_F)$. By Lemma~\ref{lemma perm basis},  $M$ is a permutation $F\mathfrak{S}_n$-module with transitive permutation basis
$\mathcal{T}_2$.  Hence 
$$V_1:=\left\{\sum_{1\leq i_1<i_2\leq n}\gamma(i_1,i_2)\, \b{i_1}\wedge\b{i_2}: \sum_{1\leq i_1<i_2\leq n}\gamma(i_1,i_2)=0\right\}$$
is the unique $F\mathfrak{S}_n$-submodule of $M$ with trivial factor module. Moreover,
$$V_2:=\left\{\gamma\sum_{1\leq i_1<i_2\leq n} \b{i_1}\wedge\b{i_2}: \gamma\in F\right\}$$
is the unique trivial submodule of $M$.

\smallskip

(a)\, It is clear that $U_1$ has codimension 1 in $S$, since it is the kernel
of the $F$-epimorphism
$$S\to F\,,\; \sum_{2\leq i_1<i_2\leq n}\beta(i_1,i_2)b(i_1,i_2)\mapsto\sum_{2\leq i_1<i_2\leq n}\beta(i_1,i_2)\,.$$
We show that $U_1=V_1\cap S$, so that $U_1$ is an $F\mathfrak{S}_n$-submodule of $S$. 
%Since $S\not\subseteq V_1$ and since
%$V_1$ has codimension 1 in $M$, $U_1$ has codimension 1 in $S$.

So let $x:=\sum_{2\leq i_1<i_2\leq n}\beta(i_1,i_2)b(i_1,i_2)\in U_1$. We write $x$ as an $F$-linear combination of the
permutation basis
$\mathcal{T}_2$ of $M$. For $i\in\{2,\ldots,n\}$, we set
$$\beta(1,i):=\sum_{j=i+1}^n\beta(i,j)-\sum_{j=2}^{i-1}\beta(j,i)\in F\,.$$
Observe that, for $2\leq i_1<i_2\leq n$, we see $\beta(i_1,i_2)$ as a summand of $\beta(1,i_1)$, and $-\beta(i_1,i_2)$ as
a summand of $\beta(1,i_2)$.
Hence we obtain $x=\sum_{1\leq j_1<j_2\leq n}\beta(j_1,j_2)\, \b{j_1}\wedge\b{j_2}$, and
$$\sum_{1\leq j_1<j_2\leq n}\beta(j_1,j_2)=\sum_{2\leq i_1<i_2\leq n}\beta(i_1,i_2)+\sum_{i=2}^n\beta(1,i)=0+0=0\,,$$
so that $x\in V_1$.

This proves $U_1\subseteq V_1\cap S$. Since, for instance, $b(n-1,n)\notin V_1$, we have  $S\not\subseteq V_1$. Thus, since  $V_1$ has codimension 1 in $M$, also
$V_1\cap S$ has codimension 1 in $S$. Hence $\dim_F(U_1)=\dim_F(V_1\cap S)$, and $U_1=V_1\cap S$.

\smallskip

(b)\, Let $n$ be odd. It suffices to show that we have $U_2=V_2$, which holds, since
\begin{align*}
\sum_{2\leq i_1<i_2\leq n}b(i_1,i_2)&=\sum_{2\leq i_1<i_2\leq n}(\b{i_1}\wedge \b{i_2}+\b{1}\wedge \b{i_1}+\b{1}\wedge\b{i_2})\\
&=\sum_{2\leq i_1<i_2\leq n}\b{i_1}\wedge\b{i_2}+ (n-2)\sum_{i=2}^n\b{1}\wedge\b{i}=\sum_{1\leq j_1<j_2\leq n}\b{j_1}\wedge\b{j_2}\,.
\end{align*}

\smallskip

(c)\, If $n\equiv 1\pmod{4}$, then $\binom{n-1}{2}$ is even, so that $U_2\subseteq U_1$. We thus get
a strictly ascending chain of $F\mathfrak{S}_n$-submodules $\{0\}\subsetneq U_2\subsetneq U_1\subsetneq S$, and the assertion follows from
Proposition~\ref{prop S in char 2}(a).

\smallskip

(d)\, Lastly, if $n\equiv 3\pmod{4}$, then $\binom{n-1}{2}$ is odd, so that $U_2\not\subseteq U_1$. Hence, $U_1+U_2=U_1\oplus U_2=S$,
comparing dimensions.
By (b), we have $U_2\cong F$, thus $U_1\cong D^{(n-2,2)}_F$, by Proposition~\ref{prop S in char 2}(b).
\end{proof}

\begin{lemma}\label{lemma S1}
Let $n\geq 4$, and let $S:=S_{\ZZ_2}$. Let further
$$S_1:=\left\{\sum_{2\leq i_1<i_2\leq n}\alpha(i_1,i_2) b(i_1,i_2):  \alpha(i_1,i_2)\in\ZZ_2, \sum_{2\leq i_1<i_2\leq n}\alpha(i_1,i_2)\equiv 0\pmod{2} \right\}\subseteq S\,.$$
Then

\smallskip

{\rm (a)}\, $S_1$ is a maximal $\ZZ_2\mathfrak{S}_n$-sublattice of $S$; in particular, $2S\subseteq S_1$. Moreover, 
$S_1/2S\cong U_1$ as $\FF_2\mathfrak{S}_n$-module;

\smallskip

{\rm (b)}\, a $\ZZ_2$-basis
of $S_1$ is given by 
\begin{equation}\label{eqn basis S1} 
\{b(i_1,i_2)+b(n-1,n), 2b(n-1,n): 2\leq i_1<i_2\leq n\,, (i_1,i_2)\neq (n-1,n)\}\,.
\end{equation}
\end{lemma}

\begin{proof}
We have the canonical projection
$$\pi:S\to S/2S\cong S_{\FF_2}\,,$$
which is a $\ZZ_2\mathfrak{S}_n$-epimorphism.
By Lemma~\ref{lemma U}, we have the $\FF_2\mathfrak{S}_n$-submodule $U_1$ of $S_{\FF_2}$, and $S_{\FF_2}/U_1\cong \FF_2$.
Obviously, $S_1$ is the preimage of the $U_1$ under $\pi$. Consequently, by Lemma~\ref{lemma max sub}, $S_1$ is a maximal $\ZZ_2\mathfrak{S}_n$-sublattice
of $S$. 

It is clear that the elements in (\ref{eqn basis S1}) are linearly independent over $\ZZ_2$ and contained in $S_1$.
If $x=\sum_{2\leq i_1<i_2\leq n}\beta(i_1,i_2)b(i_1,i_2)\in S_1$, then we have $\sum_{2\leq i_1<i_2\leq n}\beta(i_1,i_2)=2\alpha$,
for some $\alpha\in\ZZ_2$.
Setting $y:=\sum_{(j_1,j_2)\neq (n-1,n)}\beta(j_1,j_2)(b(j_1,j_2)+b(n-1,n))$ we deduce that $x-y=2\alpha b(n-1,n)$, hence
$x$ is a $\ZZ_2$-linear combination of the elements in (\ref{eqn basis S1}).
\end{proof}

\begin{lemma}\label{lemma S2}
Let $n\geq 5$ be odd. Let further
$$S_2:=\left\{ \sum_{2\leq i_1<i_2\leq n} \alpha(i_1,i_2)b(i_1,i_2): \alpha(i_1,i_2)\equiv \alpha (n-1,n)\pmod{2}\,,\text{ for all } i_1,i_2 \right\}\subseteq S\,.$$
Then

\smallskip

{\rm (a)}\, $S_2$ is a $\ZZ_2\mathfrak{S}_n$-sublattice of $S$ such that $2S\subseteq S_2$. Moreover, $S_2/2S\cong U_2$ as
$\FF_2\mathfrak{S}_n$-module;

\smallskip

{\rm (b)}\, one has
$$S_2=\left\{\alpha\sum_{2\leq i_1<i_2\leq n}b(i_1,i_2):\alpha\in \ZZ_2\right\}+2S\,;$$

\smallskip

{\rm (c)}\, a $\ZZ_2$-basis of $S_2$ is given by
\begin{equation}\label{eqn basis S2}
\left\{  \sum_{2\leq j_1<j_2\leq n}b(j_1,j_2),\, 2b(i_1,i_2): 2\leq i_1<i_2\leq n\,, (i_1,i_2)\neq (n-1,n) \right\}\,.
\end{equation}
\end{lemma}

\begin{proof}
Let $n$ be odd. It is easily verified that $S_2=\{\alpha\sum_{2\leq i_1<i_2\leq n}b(i_1,i_2):\alpha\in \ZZ_2\}+2S$. 
By Lemma~\ref{lemma U}(b), we have the unique trivial submodule $U_2$ of $S_{\FF_2}$. Moreover, 
$S_2$ is the preimage of $U_2$ under the canonical $\ZZ_2\mathfrak{S}_n$-epimorphism
$$\pi:S\to S/2S\cong S_{\FF_2}\,.$$
So $S_2$ is a $\ZZ_2\mathfrak{S}_n$-sublattice of $S$.

It is clear that the elements in (\ref{eqn basis S2}) are linearly independent over $\ZZ_2$ and contained in $S_2$. Since
$\{2b(i_1,i_2): 2\leq i_1<i_2\leq n\}$ is a $\ZZ_2$-basis of $2S$, it suffices to show that $2b(n-1,n)$ is a 
$\ZZ_2$-linear combination of the elements in (\ref{eqn basis S2}). But this is clear, since
$2b(n-1,n)=2\sum_{2\leq i_1<i_2\leq n}b(i_1,i_2)-\sum_{(j_1,j_2)\neq (n-1,n)} 2b(j_1,j_2)$. This completes the proof of the lemma.
\end{proof}

\begin{lemma}\label{lemma S submodules}
Let $n\geq 5$, and let $S_1$ be the $\ZZ_2\mathfrak{S}_n$-sublattice of $S$ defined in Lemma~\ref{lemma S1}.
If $n$ is odd, then let further $S_2$ be the  $\ZZ_2\mathfrak{S}_n$-sublattice of $S$ defined in
Lemma~\ref{lemma S2}. Then

\smallskip

{\rm (a)}\, the $\FF_2\mathfrak{S}_n$-module $S_1/2S_1$ does not have a trivial factor module;

\smallskip

{\rm (b)}\, if $n$ is odd, then the $\FF_2\mathfrak{S}_n$-module $S_2/2S_2$ does not have  a trivial submodule.
\end{lemma}

\begin{proof}
Assume first that $S_1/2 S_1$ has a trivial factor module, so that there is an $\FF_2\mathfrak{S}_n$-epimorphism
$$\varphi:S_1/2S_1\to \FF_2\,.$$
Thus $\varphi(\sigma x+2S_1)=\sigma\varphi(x+2S_1)=\varphi(x+2S_1)$, for all $x\in S_1$ and all $\sigma\in\mathfrak{S}_n$. 
In the following, let 
$${}^{-}: S_1\to S_1/2S_1$$
be the canonical epimorphism, and recall the $\ZZ_2$-basis of $S_1$ from (\ref{eqn basis S1}). If 
$\varphi(\overline{b(i_1,i_2)+b(n-1,n)})=0$, for all $2\leq i_1<i_2\leq n$ with $(i_1,i_2)\neq (n-1,n)$, then also
$$0=(n-1,n)\varphi(\overline{b(n-3,n-2)+b(n-1,n)})=\varphi(\overline{b(n-2,n-3)-b(n-1,n)})\,,$$ 
and then $0=\varphi(\overline{2b(n-1,n)})$.
But this would imply $\varphi=0$, a contradiction. 
Thus, we have some $2\leq i_1<i_2\leq n$ with $(i_1,i_2)\neq (n-1,n)$ such
that $\varphi(\overline{b(i_1,i_2)+b(n-1,n)})=1$. We distinguish three cases:

\smallskip

{\it Case 1:} $\{i_1,i_2\}\cap \{n-1,n\}=\emptyset$. Then we have
$1=(i_1,i_2)\varphi(\overline{b(i_1,i_2)+b(n-1,n)})=\varphi(\overline{-b(i_1,i_2)+b(n-1,n)})$, hence
$0=\varphi(\overline{2b(n-1,n)})$. 
Moreover, 
\begin{align*}
1&= (i_2,n-1)\varphi(\overline{b(i_1,i_2)+b(n-1,n)})=  \varphi(\overline{b(i_1,n-1)+b(i_2,n)})\\
&=(i_2,n)\varphi(\overline{b(i_1,n-1)+b(i_2,n)})=\varphi(\overline{b(i_1,n-1)-b(i_2,n)})\\
&=\varphi(\overline{b(i_1,n-1)+b(n-1,n)})-\varphi(\overline{b(i_2,n)+b(n-1,n)})\,.
\end{align*}
But then
\begin{align*}
\varphi(\overline{b(i_1,n-1)+b(n-1,n)})&=(i_1,i_2)(n-1,n)\varphi(\overline{b(i_1,n-1)+b(n-1,n)})\\
&=\varphi(\overline{b(i_2,n)-b(n-1,n)})\\
&=\varphi(\overline{b(i_2,n)+b(n-1,n)})-\varphi(\overline{2b(n-1,n)})\\
&=\varphi(\overline{b(i_2,n)+b(n-1,n)})\,,
\end{align*}
which yields the contradiction $1=0$.

%Since $n\geq 7$, there exist $2\leq j_1<j_2\leq n$ such that 
%$\{i_1,i_2\}\cap \{n-1,n\}=\emptyset=\{j_1,j_2\}\cap \{n-1,n\}$ and $\{i_1,i_2\} \cap \{j_1,j_2\} = \emptyset$.
%We set $\sigma:=(i_1,j_1)(i_2,j_2)$ and $\pi:=(n-1,j_1)(n,j_2)$, and deduce the following:
%\begin{align*}
%\varphi(\overline{b(j_1,j_2)+b(n-1,n)})&=\sigma\varphi(\overline{b(i_1,i_2)+b(n-1,n)})\\
%&=1\\
%&=\pi\varphi(\overline{b(i_1,i_2)+b(n-1,n)})=\varphi(\overline{b(i_1,i_2)+b(j_1,j_2)})\,,\\
%\end{align*}
%thus also
%\begin{align*}
%1&=(j_1,j_2)\varphi(\overline{b(i_1,i_2)+b(j_1,j_2)})=\varphi(\overline{b(i_1,i_2)-b(j_1,j_2)})\\
%&=\varphi(\overline{b(i_1,i_2)+b(n-1,n)})-\varphi(\overline{b(j_1,j_2)+b(n-1,n)})=0\,,\\
%\end{align*}
%a contradiction. 

\smallskip

{\it Case 2:} $i_2=n$, $i_1<n-1$. 
%Let $j\in\{2,\ldots,n\}\smallsetminus\{i_1,n-1,n\}$. Then
%\begin{align*}
%\varphi(\overline{b(j,n)+b(n-1,n)})&=(i_1,j)\varphi(\overline{b(i_1,n)+b(n-1,n)})=1\\
%&=(j,n-1)\varphi(\overline{b(i_1,n)+b(n-1,n)})=\varphi(\overline{b(i_1,n)+b(j,n)})\,,
%&\\
%\varphi(\overline{b(i_1,n-1)-b(n-1,n)})&=(n-1,n)\varphi(\overline{b(i_1,n)+b(n-1,n)})=1\\
%\end{align*}
Then
\begin{align*}
\varphi(\overline{b(i_1,n-1)-b(n-1,n)})&=(n-1,n)\varphi(\overline{b(i_1,n)+b(n-1,n)})=1\\
&=(i_1,n)\varphi(\overline{b(i_1,n)+b(n-1,n)})=\varphi(\overline{-b(i_1,n)-b(i_1,n-1)})\\
&=\varphi(\overline{b(i_1,n)+b(i_1,n-1)})\,,
\end{align*}
and hence
\begin{align*}
0&=\varphi(\overline{b(i_1,n)+b(i_1,n-1)})-\varphi(\overline{b(i_1,n-1)-b(n-1,n)})\\
&=\varphi(\overline{b(i_1,n)+b(n-1,n)})=1\,,
\end{align*}
a contradiction.

\smallskip

{\it Case 3:} $i_2=n-1$. Let $j\in\{2,\ldots,n\}\smallsetminus\{i_1,n-1,n\}$. This time we get
\begin{align*}
\varphi(\overline{b(j,n-1)+b(n-1,n)})&=(i_1,j)\varphi(\overline{b(i_1,n-1)+b(n-1,n)})=1\\
&=(j,n)\varphi(\overline{b(i_1,n-1)+b(n-1,n)})=\varphi(\overline{b(i_1,n-1)-b(j,n-1)})\,,
\end{align*}
hence we derive the contradiction
\begin{align*}
0&=\varphi(\overline{b(i_1,n-1)-b(j,n-1)})+\varphi(\overline{b(j,n-1)+b(n-1,n)})\\
&=\varphi(\overline{b(i_1,n-1)+b(n-1,n)})=1\,.
\end{align*}

This finally proves that $S_1/2S_1$ cannot have a trivial factor module.

\medskip

Next let $n$ be odd, and assume that $S_2/2S_2$ has a trivial submodule. Then there is some $x\in S_2\smallsetminus 2S_2$ such that $x-\sigma\cdot x\in  2S_2$, for all $\sigma\in\mathfrak{S}_n$. We show that this forces $x\in 2S$. To this end, we write $x=\sum_{2\leq i_1<i_2\leq n}\alpha(i_1,i_2)b(i_1,i_2)$, for uniquely determined $\alpha(i_1,i_2)\in\ZZ_2$.
By Lemma~\ref{lemma S2}, we have $\alpha(i_1,i_2)\equiv \alpha(n-1,n)\pmod{2}$, for all $2\leq i_1<i_2\leq n$.
Thus, if $\alpha(n-1,n)\equiv 0\pmod{2}$, then $x\in 2S$. So assume that $\alpha(i_1,i_2)\in \ZZ_2^\times$, for
all $2\leq i_1<i_2\leq n$. Then we may further assume that $\alpha(n-1,n)=1$, and get
$$x-(n-1,n)x=2b(n-1,n)+\sum_{i=2}^{n-2}(\alpha(i,n-1)-\alpha(i,n))b(i,n-1) +\sum_{i=2}^{n-2}(\alpha(i,n)-\alpha(i,n-1))b(i,n)\,;$$
in particular,  the coefficient of $x-(n-1,n)x$ at the basis element $b(n-3,n-2)$ of $S$ is 0, the coefficient at $b(n-1,n)$ is 2.
But, since $x-(n-1,n)x\in 2S_2$, Lemma~\ref{lemma S2} would then imply $2\equiv 0\pmod{4}$, a contradiction.
Hence $x\in 2S$, so that $x+2S_2$ also spans a trivial submodule of the $\FF_2\mathfrak{S}_n$-module $2S/2S_2\cong S/S_2$.

We now distinguish two cases: if $n\equiv 1 \pmod{4}$, then, by Lemma~\ref{lemma U} and Proposition~\ref{prop S in char 2},
we have $2S\subseteq S_2\subseteq S_1\subseteq S$ and $S/S_2\cong (S/2S)/(S_2/2S)\cong S_{\FF_2}/U_2$ is uniserial
with simple socle $D^{(n-2,2)}_{\FF_2}$; in particular, $S/S_2$ does not have  a trivial submodule, a contradiction.

If $n\equiv 3\pmod{4}$, then, by Lemma~\ref{lemma U}, we have
$$S/S_2\cong S_{\FF_2}/U_2\cong U_1\cong D^{(n-2,2)}_{\FF_2}\not\cong \FF_2\,,$$
and we obtain again a contradiction.

This completes the proof of the lemma.
\end{proof}

\begin{lemma}\label{lemma S submodules 2}
Let $n\geq 5$ be such that $n\equiv 1\pmod{4}$, and let $S_1$ and $S_2$ be the $\ZZ_2\mathfrak{S}_n$-sublattices of $S$ defined in Lemma~\ref{lemma S1}
and Lemma~\ref{lemma S2}, respectively. 
Then $2S/2S_1$ is the unique trivial submodule of the $\FF_2\mathfrak{S}_n$-module $S_1/2S_1$, and
$(S_2/2S_2)/(2S/2S_2)$ is the unique trivial factor module of the $\FF_2\mathfrak{S}_n$-module $S_2/2S_2$.
\end{lemma}

\begin{proof}
Suppose that $n\equiv 1\pmod{4}$. In consequence of Lemma~\ref{lemma S1} and Lemma~\ref{lemma S2}, 
we have $2S\subseteq S_2\subseteq S_1\subseteq S$, since $\binom{n-1}{2}$ is even.

We first show that $2S/2S_1$ is the unique trivial submodule of $S_1/2S_1$. By Lemma~\ref{lemma U}, we have
$$2S/2S_1\cong S/S_1\cong (S/2S)/(S_1/2S)\cong S_{\FF_2}/U_1\cong \FF_2\,,$$
so that $2S/2S_1$ is indeed a trivial submodule of $S_1/2S_1$. Now let $x\in S_1$ be such that $x+2S_1$ spans a trivial
submodule of $S_1/2S_1$, that is, $x\notin 2S_1$ and $x-\sigma x\in 2S_1$, for all $\sigma\in\mathfrak{S}_n$.  We need  to show
that $x\in 2S$. Assume not. Since $2S_1\subseteq 2S$, $x+2S$ then also spans  a trivial submodule of $S/2S$. By Proposition~\ref{prop S in char 2} and Lemma~\ref{lemma U}, the $\FF_2\mathfrak{S}_n$-module $S/2S$ is uniserial, and $S_2/2S\cong U_2$ is
its unique trivial submodule. Hence $x+2S\in S_2/2S$ and $x\in S_2$. So, by Lemma~\ref{lemma S2}, we have
$$x=\sum_{2\leq i_1<i_2\leq n}\alpha(i_1,i_2)b(i_1,i_2)\,,$$
for $\alpha(i_1,i_2)\in\ZZ_2$ such that $\alpha(i_1,i_2)\equiv \alpha(n-1,n)\pmod{2}$, for
all $2\leq i_1<i_2\leq n$. Since we are assuming $x\notin 2S$, we must have $\alpha(i_1,i_2)\in \ZZ_2^\times$, for all
$2\leq i_1<i_2\leq n$. We may thus suppose that $\alpha(n-1,n)=1$. Then we get
$$x-(n-1,n)x=2b(n-1,n)+\sum_{i=2}^{n-2}(\alpha(i,n-1)-\alpha(i,n))b(i,n-1) +\sum_{i=2}^{n-2}(\alpha(i,n)-\alpha(i,n-1))b(i,n)\,.$$
But then the sum of the coefficients at the basis elements of $S$ is $2\not\equiv 0\pmod{4}$, so that $x-(n-1,n)x\notin 2S_1$, a contradiction.

Consequently, we must have $x\in 2S$, as claimed.

\medskip

Now we show that $(S_2/2S_2)/(2S/2S_2)$ is the unique trivial factor module of $S_2/2S_2$. To this end, let 
$$\varphi:S_2/2S_2\to \FF_2$$
be an $\FF_2\mathfrak{S}_n$-epimorphism. We show that $\ker(\varphi)$ contains $2S/2S_2$. Since, by Lemma~\ref{lemma S2}, we
have
$$(S_2/2S_2)/(2S/2S_2)\cong S_2/2S\cong U_2\cong \FF_2\,,$$
this forces $\ker(\varphi)=2S/2S_2$, and we are done.

Let  ${}^-:S_2\to S_2/2S_2$ be the canonical epimorphism. Since $S_2/2S_2$ has composition factors $\FF_2$ with multiplicity 
2 and $D^{(n-2,2)}_{\FF_2}$ with multiplicity 1 and since, by Lemma ~\ref{lemma S submodules}, $S_2/2S_2$ does not have a 
trivial submodule, $S_2/2S_2$ has a unique simple $\FF_2\mathfrak{S}_n$-submodule, namely
$$2S_1/2S_2\cong S_1/S_2\cong (S_1/2S)/(S_2/2S)\cong U_1/U_2\cong D^{(n-2,2)}_{\FF_2}\,.$$
Thus $2S_1/2S_2\subseteq \ker(\varphi)$. Recall that $b(i_1,i_2)+b(n-1,n)\in S_1$, hence
$2b(i_1,i_2)+2b(n-1,n)\in 2S_1$
for all $2\leq i_1<i_2\leq n$. This implies 
$$0=\varphi(\overline{2b(i_1,i_2)+2b(n-1,n)})=\varphi(\overline{2b(i_1,i_2)})+\varphi(\overline{2b(n-1,n)})\,,$$
that is, $\varphi(\overline{2b(i_1,i_2)})=\varphi(\overline{2b(n-1,n)})$, for all $2\leq i_1<i_2\leq n$. Assume that
$\varphi(\overline{2b(n-1,n)})=1$. Then we get
\begin{align*}
\varphi(\overline{\sum_{2\leq i_1<i_2\leq n}b(i_1,i_2)})+1&=\varphi(\overline{\sum_{2\leq i_1<i_2\leq n}b(i_1,i_2)}-\overline{2b(n-1,1)})\\
&=\varphi(\overline{\sum_{(i_1,i_2)\neq (n-1,n)}b(i_1,i_2)-b(n-1,n)})\\
&=(n-1,n)\varphi(\overline{\sum_{2\leq i_1<i_2\leq n}b(i_1,i_2)})=\varphi(\overline{\sum_{2\leq i_1<i_2\leq n}b(i_1,i_2)})\,.
\end{align*}
a contradiction. Therefore, $\varphi(\overline{2b(i_1,i_2)})=0$, for all $2\leq i_1<i_2\leq n$, proving $2S/2S_2\subseteq \ker(\varphi)$.
\end{proof}

\begin{prop}\label{prop S uniserial}
Let $n\geq 5$ be odd, and let $S_1$ and $S_2$ be the $\ZZ_2\mathfrak{S}_n$-sublattices of $S$ defined in Lemma~\ref{lemma S1}
and Lemma~\ref{lemma S2}, respectively. Then the $\FF_2\mathfrak{S}_n$-modules $S_1/2S_1$ and $S_2/2S_2$ are uniserial.
More precisely, 

\smallskip

{\rm (a)}\, if $n\equiv 1\pmod{4}$, then $S_1/2S_2$ and $S_2/2S_2$ have the following Loewy series:
$$S_1/2S_1\sim\begin{bmatrix}
D^{(n-2,2)}_{\FF_2}\\
\FF_2\\
\FF_2
\end{bmatrix}\,,\quad
S_2/2S_2\sim\begin{bmatrix}
\FF_2\\
\FF_2\\
D^{(n-2,2)}_{\FF_2}
\end{bmatrix}\,;$$

\smallskip

{\rm (b)}\, if $n\equiv 3\pmod{4}$, then  $S_1/2S_2$ and $S_2/2S_2$ have the following Loewy series:
$$S_1/2S_1\sim\begin{bmatrix}
D^{(n-2,2)}_{\FF_2}\\
\FF_2
\end{bmatrix}\,,\quad
S_2/2S_2\sim\begin{bmatrix}
\FF_2\\
D^{(n-2,2)}_{\FF_2}
\end{bmatrix}\,.$$
\end{prop}

\begin{proof}
(a)\, First suppose that $n\equiv 1\pmod{4}$. Since $S_1$ and $S_2$ are $\ZZ_2\mathfrak{S}_n$-lattices
of finite index in $S$, both of them are $\ZZ_2$-forms of $S_{\QQ_2}$. By Proposition~\ref{prop S in char 2}, 
the $\FF_2\mathfrak{S}_n$-modules $S_1/2S_1$ and $S_2/2S_2$ thus have three composition factors: $\FF_2$ with multiplicity 2, and
$D^{(n-2,2)}_{\FF_2}$ with multiplicity 1. In consequence
of Lemma~\ref{lemma S submodules} and Lemma~\ref{lemma S submodules 2}, we also know that
$S_1/2S_1$ has a unique trivial submodule and no trivial factor module, while $S_2/2S_2$ has a unique trivial factor module and
no trivial submodule. Therefore, $S_1/2S_1$ and $S_2/2S_2$ both have to be uniserial with the claimed Loewy series.

\smallskip

(b)\, If $n\equiv 3\pmod{4}$, then Proposition~\ref{prop S in char 2} shows that $S_1/2S_1$ and
$S_2/2S_2$ have composition factors $\FF_2$ and $D^{(n-2,2)}_{\FF_2}$, both occurring with multiplicity 1. Now Lemma~\ref{lemma S submodules} immediately implies the assertion.
\end{proof}

We are now in the position to establish the first main result of this section:

\begin{thm}\label{thm n odd}
Let $n\geq 5$ be odd, and let $S_1$ and $S_2$ be the $\ZZ_2\mathfrak{S}_n$-sublattices of $S:=S_{\ZZ_2}$
defined in Lemma~\ref{lemma S1} and Lemma~\ref{lemma S2}, respectively. Then $S$, $S_1$ and $S_2$ are representatives
of the isomorphism classes of $\ZZ_2$-forms of $S_{\QQ_2}$.
\end{thm}

\begin{proof}
Suppose that $n\equiv 1\pmod{4}$ first. Then $S/2S$ is uniserial with composition series $\{0\}\subseteq U_2\subseteq U_1\subseteq S$, by
Lemma~\ref{lemma U}. Thus, by Lemma~\ref{lemma S1} and Lemma~\ref{lemma S2}, we deduce
that the only $\ZZ_2\mathfrak{S}_n$-sublattices of $S$ containing $2S\cong S$ are $2S$, $S_2$, $S_1$ and $S$. By Proposition~\ref{prop S uniserial} and Proposition~\ref{prop S in char 2}, $S/2S$, $S_1/2S_1$ and $S_2/2S_2$ are uniserial $\FF_2\mathfrak{S}_n$-modules.
Therefore, the assertion of the theorem follows from Proposition~\ref{prop max forms 2}.

\smallskip

If $n\equiv 3\pmod{4}$, then, by Proposition~\ref{prop S in char 2}, we know that $S/2S\cong S_{\FF_2}\cong \FF_2\oplus D^{(n-2,2)}_{\FF_2}$ as $\FF_2\mathfrak{S}_n$-module. Since $S$ is a $\ZZ_2$-form of $S_{\QQ_2}$, the assertion follows from Proposition~\ref{prop S uniserial} and Proposition~\ref{prop max forms}. 
\end{proof}

Next we turn to the case where $n\equiv 2\pmod{4}$, and start with the following immediate consequence of Proposition~\ref{prop S in char 2}(c).

\begin{lemma}\label{lemma T chain}
Let $n\geq 6$ be such that $n\equiv 2\pmod{4}$ and let $S:=S_{\ZZ_2}$.  There are uniquely determined $\ZZ_2\mathfrak{S}_n$-lattices
$T_1,T_2,T_3$ such that 
$$2S=:T_4\subseteq T_3\subseteq T_2\subseteq T_1\subseteq T_0:=S$$
and
such that $T_i$ is maximal in $T_{i-1}$, for $i\in\{1,\ldots,4\}$. Moreover, one has
$$S/T_1\cong \FF_2\,,\quad  T_1/T_2\cong D^{(n-2,2)}_{\FF_2}\,,\quad T_2/T_3\cong \FF_2\,,\quad T_3/2S\cong D^{(n-1,1)}_{\FF_2}$$
as $\FF_2\mathfrak{S}_n$-modules; in particular, $T_1=S_1$, where $S_1$ is the $\ZZ_2\mathfrak{S}_n$-sublattice of $S$
defined in Lemma~\ref{lemma S1}.
\end{lemma}

\begin{rem}\label{rem overZ}
Note that the $\ZZ_2$-forms defined in Lemmas~\ref{lemma S1}, \ref{lemma S2} and~\ref{lemma T chain} are all of the form $\hat{L}_{\ZZ_2}$, for some $\ZZ$-form $\hat{L}$ of $S_\QQ$ with $\hat{L}\subseteq S_{\ZZ}$ and $[S_{\ZZ}:\hat{L}]=[S_{\ZZ_2}:\hat{L}_{\ZZ_2}]$; in particular,
the indices $[S_{\ZZ}:\hat{L}]$ are $2$-powers.
While this is obvious for $S_1, S_2$, for $T_1,T_2,T_3$ this can be seen by lifting the unique composition series of $S_\ZZ/2S_\ZZ\cong S_{\ZZ_2}/2S_{\ZZ_2}$ along $\ZZ \to \FF_2$ to a chain of $\ZZ$-forms of $S_{\QQ}$ instead of lifting it along $\ZZ_2 \to \FF_2$ to a chain of $\ZZ_2$-forms
of $S_{\QQ_2}$. In this way, we, in particular, have $\FF_2\mathfrak{S}_n$-isomorphisms $\hat{S}_i/2\hat{S}_i\cong S_i/2S_i$
for $n\equiv 1\pmod{2}$ and $i\in\{1,2\}$, and $\hat{T}_j/2\hat{T}_j\cong T_j/2T_j$ for $n\equiv 2\pmod{4}$ and $j\in\{1,2,3\}$.

\end{rem}

Lemma~\ref{lemma T chain} thus provides us with $\ZZ_2$-forms $S,T_1,T_2,T_3$ of the Specht $\QQ_2\mathfrak{S}_n$-module
$S_{\QQ_2}$. Our aim now is to show that these form a set of representatives of the isomorphism classes of $\ZZ_2$-forms
of $S_{\QQ_2}$. 

\begin{rem}\label{rem dual S}
Recall that we are still identifying the Specht lattice $S=S_{\ZZ_2}^{(n-2,1^2)}$ with the exterior power 
$\bigwedge^2(S^{(n-1,1)}_{\ZZ_2})$, as in \ref{noth hooks}(b). As before, via this identification the standard $(n-2,1^2)$-polytabloid
$e(1,i_1,i_2)$, for $2\leq i_1<i_2\leq n$, corresponds to the basis element $b(i_1,i_2)$. 

In the following we shall also identify the dual Specht lattice $S^*$ with $(\bigwedge^2(S^{(n-1,1)}_{\ZZ_2}))^*$
via the isomorphism in \ref{noth exterior}(c).
So if $\{e(1,i_1,i_2)^*: 2\leq i_1<i_2\leq n\}$ denotes the $\ZZ_2$-basis of $S^*$ dual to the standard polytabloid basis of $S$ and
if $\{b(i_1,i_2)^*: 2\leq i_1<i_2\leq n\}$ denotes the $\ZZ_2$-basis of $(\bigwedge^2(S^{(n-1,1)}_{\ZZ_2}))^*$ dual to our standard 
basis of $\bigwedge^2(S^{(n-1,1)}_{\ZZ_2})$, then we have $e(1,i_1,i_2)^*=b(i_1,i_2)^*$, for all $2\leq i_1<i_2\leq n$.
\end{rem}

\begin{prop}\label{prop T dual}
Let $n\geq 6$ be such that $n\equiv 2\pmod{4}$, and let $S:=S_{\ZZ_2}$. With the notation of
Lemma~\ref{lemma T chain}, the $\ZZ_2\mathfrak{S}_n$-lattice $T_3$ is isomorphic to the dual $\ZZ_2\mathfrak{S}_n$-lattice
$S^*$.
\end{prop}

\begin{proof}
Let $\lambda:=(n-2,1^2)$ and consider the $\ZZ_2\mathfrak{S}_n$-monomorphism $\varphi:=\varphi^\lambda_{\ZZ_2}:S^*\to S$ in
Corollary~\ref{cor phi ZZ_p}. In consequence of Example~\ref{expl hook dual}, for $2\leq i<j\leq n$, we get:
\begin{align*}
\varphi(b(i,j)^*)&=\varphi(e(1,i,j)^*)=\sum_{\sigma\in\mathfrak{S}(\{1,\ldots,n\}\smallsetminus\{i,j\})}\sigma\cdot e(1,i,j)\\
&=\sum_{\sigma\in\mathfrak{S}(\{1,\ldots,n\}\smallsetminus\{i,j\})}\sigma\cdot b(i,j)\\
&=\sum_{\sigma\in\mathfrak{S}(\{1,\ldots,n\}\smallsetminus\{i,j\})}\sigma\cdot ((\b{i}\wedge\b{j})+(\b{1}\wedge\b{i})-(\b{1}\wedge\b{j})))\\
&=(n-3)!\mathop{\sum_{k=1}^n}_{i\neq k\neq j}((\b{i}\wedge\b{j})+(\b{k}\wedge\b{i})-(\b{k}\wedge\b{j})))\,.
\end{align*}
For $2\leq i<j\leq n$, we set
$$f(i,j):=\mathop{\sum_{k=1}^n}_{i\neq k\neq j}((\b{i}\wedge\b{j})+(\b{k}\wedge\b{i})-(\b{k}\wedge\b{j})))\in S\,.$$
Then the elements $f(i,j)\in S$, with $2\leq i<j\leq n$, are linearly independent over $\ZZ_2$, and span a $\ZZ_2\mathfrak{S}_n$-sublattice $T$ of $S$ isomorphic to $\varphi(S^*)\cong S^*$. It remains to verify that $2S\subseteq T$.

By  \ref{noth Young Specht}(b), $2S$ is a cyclic $\ZZ_2\mathfrak{S}_n$-module generated by any of the basis elements $2b(i,j)$, where $2\leq i<j\leq n$. Thus, it suffices to show that $2b(2,3)\in T$. In the following, we shall show that
\begin{equation}\label{eqn f(2,3)}
n\, b(2,3)=3f(2,3)+\sum_{k=4}^n(f(2,k)-f(3,k))\in T\,.
\end{equation}
Since $n\equiv 2\pmod{4}$, we have $n=2u$, for some $u\in \ZZ_2^\times$, so that (\ref{eqn f(2,3)}) then implies $2b(2,3)\in T$ as well.
By definition,
\begin{align*}
f(2,3)&=b(2,3)+\sum_{k=4}^n((\b{2}\wedge\b{3})+(\b{k}\wedge\b{2})-(\b{k}\wedge\b{3}))\\
&=b(2,3)+\sum_{k=4}^n((\b{3}\wedge\b{k})+(\b{2}\wedge\b{3})-(\b{2}\wedge\b{k}))\\
&=b(2,3)+\sum_{k=4}^n(b(3,k)+b(2,3)-b(2,k))=(n-2)b(2,3)+\sum_{k=4}^n(-b(2,k)+b(3,k))\,.
\end{align*}
For $k\in\{4,\ldots,n\}$, we have $(3,k)\cdot b(2,3)=b(2,k)$, hence also $(3,k)\cdot b(2,3)^*=b(2,k)^*$, by Example~\ref{expl hook dual}, and 
$(n-3)!f(2,k)=\varphi(b(2,k)^*)=(n-3)! (3,k)\cdot f(2,3)$ as well as $f(2,k)=(2,k)\cdot f(2,3)$. Analogously, for $k\in\{4,\ldots,n\}$,
we have $(2,3)\cdot b(2,k)=b(3,k)$, thus $(2,3)\cdot b(2,k)^*=b(3,k)^*$ and $(2,3)\cdot f(2,k)=f(3,k)$. From this we now deduce the
following
\begin{align*}
f(2,k)&=(n-2)b(2,k)-b(2,3)-\sum_{i=4}^{k-1}b(2,i)-\sum_{i=k+1}^nb(2,i)-\sum_{i=3}^{k-1}b(i,k)+\sum_{i=k+1}^nb(k,i)\\
f(3,k)&=(n-2)b(3,k)+b(2,3)-\sum_{i=4}^{k-1}b(3,i)-\sum_{i=k+1}^nb(3,i)-\sum_{i=2}^{k-1}b(i,k)+\sum_{i=k+1}^nb(k,i)\,.
\end{align*}
Using these equations one now verifies (\ref{eqn f(2,3)}). 

Therefore, we have now shown that $2S\subseteq T\cong S^*$. It remains to prove that actually $T=T_3$. 
By Proposition~\ref{prop S in char 2}, the $\FF_2\mathfrak{S}_n$-module $S_{\FF_2}$ is not self-dual. Thus, by \ref{noth change rings}(d) the $\ZZ_2\mathfrak{S}_n$-lattice $S$ cannot be self-dual either. In particular, $S\neq T\neq 2S$.
So, from Lemma~\ref{lemma T chain}, we deduce that $T= T_i$, for some $i\in\{1,2,3\}$. Set $D:=D^{(n-2,2)}_{\FF_2}$ and 
$D':=D^{(n-1,1)}_{\FF_2}$, for the remainder of this proof.

By Proposition~\ref{prop S in char 2}, the $\FF_2\mathfrak{S}_n$-module $S/2S\cong S_{\FF_2}$ has simple socle
$D'$, so that $T/2T\cong (S/2S)^*$ has simple head $(D')^*\cong D'$. Since $T/2S\neq\{0\}$ and
$T/2S\cong (T/2T)/(2S/2T)$, also $T/2S$ has a simple factor module isomorphic to $D'$. 
Since $S/2S$ is uniserial, by Proposition~\ref{prop S in char 2}, so are $T_i/2S$, for $i\in\{1,2,3\}$. More precisely,
Lemma~\ref{lemma T chain} shows that $T_1/2S$ is uniserial with head isomorphic to $D\not\cong D'$, and $T_2/2S$ is uniserial
with head isomorphic to $\FF_2$. Therefore, we must have $T=T_3$, and the proof of the proposition is complete. 
\end{proof}

\begin{cor}\label{cor T uniserial}
Let $n\geq 6$ be  such that $n\equiv 2\pmod{4}$, and let $S:=S_{\ZZ_2}$. 
Let $T_1$, $T_2$ and $T_3$ be the $\ZZ_2\mathfrak{S}_n$-sublattices of $S$ given in Lemma~\ref{lemma T chain}. Then the
$\FF_2\mathfrak{S}_2$-modules $T_1/2T_1$, $T_2/2T_2$ and $T_3/2T_3$ are uniserial with the following Loewy series:
$$T_1/2T_1\sim\begin{bmatrix} D^{(n-2,2)}_{\FF_2}\\\FF_2\\D^{(n-1,1)}_{\FF_2}\\\FF_2 \end{bmatrix}\,,\quad 
T_2/2T_2\sim\begin{bmatrix} \FF_2\\ D^{(n-1,1)}_{\FF_2}\\\FF_2\\D^{(n-2,2)}_{\FF_2}\end{bmatrix}\,,\quad 
T_3/2T_3\sim\begin{bmatrix} D^{(n-1,1)}_{\FF_2}\\\FF_2\\D^{(n-2,2)}_{\FF_2}\\\FF_2 \end{bmatrix}\,.
$$
\end{cor}

\begin{proof}
Let $D:=D^{(n-2,2)}_{\FF_2}$ and $D':=D^{(n-1,1)}_{\FF_2}$. 
Recall that every simple $\FF_2\mathfrak{S}_n$-module is self-dual.
By Proposition~\ref{prop S in char 2}, the $\FF_2\mathfrak{S}_n$-module
$S/2S\cong S_{\FF_2}$ is uniserial with Loewy series
$$S/2S\sim\begin{bmatrix}  \FF_2\\ D\\\FF_2\\D'\end{bmatrix}\,.$$
By Proposition~\ref{prop T dual} and \ref{noth change rings}(d), we further know that $T_3/2T_3\cong (S/2S)^*$, so that also $T_3/2T_3$
is  uniserial with Loewy series as claimed. It remains to treat $T_1/2T_1$ and $T_2/2T_2$. To this end, we 
first observe the following:

\smallskip

(i)\, Since 
$(T_3/2T_3)/(2T_1/2T_3)\cong T_3/2T_1\subseteq T_1/2T_1$, the $\FF_2\mathfrak{S}_n$-module $T_1/2T_1$ has
a uniserial submodule with Loewy series
$$\begin{bmatrix}  D'\\\FF_2\end{bmatrix}\,;$$
the corresponding factor module is isomorphic to $T_1/T_3\subseteq S/T_3\cong (S/2S)/(T_3/2S)$, hence  uniserial as well with Loewy series
$$\begin{bmatrix}  D\\\FF_2\end{bmatrix}\,.$$

\smallskip

(ii)\, We have $\FF_2\cong 2S/2T_1\subseteq T_1/2T_1$, and the factor module $(T_1/2T_1)/(2S/2T_1)\cong T_1/2S$
is uniserial with Loewy series
$$\begin{bmatrix}  D\\\FF_2\\D'\end{bmatrix}\,.$$

\smallskip

Analogously,

\smallskip

(iii)\; $T_2/2T_2$ has a uniserial $\FF_2\mathfrak{S}_n$-submodule $V$ with Loewy series
$$\begin{bmatrix}  \FF_2\\D\end{bmatrix}\,,$$
and the corresponding factor module is uniserial with Loewy series
$$\begin{bmatrix}  \FF_2\\D'\end{bmatrix}\,;$$

\smallskip

(iv)\, $T_2/2T_2$ has the uniserial $\FF_2\mathfrak{S}_n$-submodule $U:=T_3/2T_2$ with Loewy series
$$\begin{bmatrix}  D'\\\FF_2\\D\end{bmatrix}\,,$$
and $(T_2/2T_2)/(T_3/2T_2)\cong T_2/T_3\cong \FF_2$.

\smallskip

In consequence of (iv), we either have $U=\Rad(T_2/2T_2)$ or $\Rad(U)=\Rad(T_2/2T_2)$. In the second case, we
would have $(T_2/2T_2)/\Rad(T_2/2T_2)\cong D'\oplus \FF_2$, forcing $T_2/2T_2$ to have Loewy series
$$\begin{bmatrix}  D'\oplus \FF_2\\ \FF_2\\D\end{bmatrix}\,.$$
By (iii) we also get
$$\Rad(T_2/2T_2)/(\Rad(T_2/2T_2)\cap V)\cong (\Rad(T_2/2T_2)+V)/V\subseteq \Rad((T_2/2T_2)/V)\cong D'\,.$$
On the other hand, the factor module $\Rad(T_2/2T_2)/(\Rad(T_2/2T_2)\cap V)$ 
either has to be $\{0\}$, or has a composition factor isomorphic to $D$ or $\FF_2$. Since $\FF_2\not\cong D'\not\cong D$, 
we must have $\Rad(T_2/2T_2)/(\Rad(T_2/2T_2)\cap V)=\{0\}$, thus $\Rad(T_2/2T_2)\subseteq V$. Comparing $\FF_2$-dimensions
this implies $\Rad(T_2/2T_2)=V$, so that $(T_2/2T_2)/V\cong D'\oplus \FF_2$, which is not uniserial, a contradiction to (iii).
Therefore, we conclude that $U=\Rad(T_2/2T_2)$, and the assertion concerning the Loewy series of $T_2/2T_2$ follows. 

Similarly, considering socle series instead of Loewy series one shows that also $T_1/2T_1$ is uniserial with the claimed Loewy series.  
\end{proof}

\begin{thm}\label{thm n 2 mod 4}
Let $n\geq 6$ be  such that $n\equiv 2\pmod{4}$, and let $S:=S_{\ZZ_2}$. 
Let $T_1$, $T_2$ and $T_3$ be the $\ZZ_2\mathfrak{S}_n$-sublattices of $S$ given in Lemma~\ref{lemma T chain}.
Then
$S,T_1,T_2$ and $T_3$ are representatives of the isomorphism classes of $\ZZ_2$-forms of the Specht
$\QQ_2\mathfrak{S}_n$-module 
$S_{\QQ_2}$.
\end{thm}

\begin{proof}
By Lemma~\ref{lemma T chain}, $S$, $T_1$, $T_2$ and $T_3$ are the only $\ZZ_2\mathfrak{S}_n$-sublattices of $S$
properly containing $2S$. By Proposition~\ref{prop S in char 2} and Corollary~\ref{cor T uniserial}, the 
$\FF_2\mathfrak{S}_n$-modules $S/2S$, $T_1/2T_1$, $T_2/2T_2$ and $T_3/2T_3$ are uniserial and pairwise
non-isomorphic; in particular, the $\ZZ_2\mathfrak{S}_n$-lattices $S,T_1,T_2$ and $T_3$ are pairwise non-isomorphic.
Now Proposition~\ref{prop max forms 2} yields the assertion of the theorem.
\end{proof}

\begin{rem}\label{rem summing up}
(a)\, We would like to mention that the definition of the $\ZZ_2\mathfrak{S}_n$-lattices
$S_1$ and $S_2$ in Lemma~\ref{lemma S1} and Lemma~\ref{lemma S2}, respectively, as well as our proofs 
of Lemma~\ref{lemma S submodules} and Lemma~\ref{lemma S submodules 2} are reminiscent of 
\cite[Satz (I.11)]{Plesken1974}. There the $\ZZ$-forms of a certain absolutely simple $\QQ[\mathfrak{S}_2\wr\mathfrak{S_n}]$-module have been determined.

\smallskip

(b)\, Since $S(2)_{\QQ_2}$ is self-dual, by Theorem~\ref{thm Specht properties}(a),
for every $\ZZ_2$-form $L$ of $S(2)_{\QQ_2}$, the dual lattice $L^*$ has to be a $\ZZ_2$-form of $S(2)_{\QQ_2}$
as well. 

Suppose first that $n$ is odd. Then the natural Specht module $S(1)_{\QQ_2}$ is
absolutely simple and self-dual, by Theorem~\ref{thm Specht properties}, and the Specht
$\FF_2\mathfrak{S}_n$-module $S(1)_{\FF_2}\cong S(1)_{\ZZ_2}/2 S(1)_{\ZZ_2}$ is also
absolutely simple, by \cite[Theorem 23.7]{James1978}. Thus, by Proposition~\ref{prop notdivides},
$S(1)_{\ZZ_2}$ is the only $\ZZ_2$-form of $S(1)_{\QQ_2}$ and 
$S(1)_{\ZZ_2}\cong (S(1)_{\ZZ_2})^*$.
So, by \ref{noth change rings}(d), also $S(2)_{\ZZ_2}\cong \bigwedge^2(S(1)_{\ZZ_2})$ is self-dual.
By Proposition~\ref{prop S uniserial}, the $\FF_2\mathfrak{S}_n$-modules $S_1/2S_1$ and $S_2/2S_2$ are not self-dual,
thus the $\ZZ_2\mathfrak{S}_n$-lattices $S_1$ and $S_2$ are  not self-dual either. Hence we must have $S_1\cong S_2^*$, by
Theorem~\ref{thm n odd}.

\smallskip

Now suppose that $n\equiv 2\pmod{4}$. By Proposition~\ref{prop T dual}, we know that $(S(2)_{\ZZ_2})^*\cong T_3$.
Corollary~\ref{cor T uniserial} shows that $T_1/2T_1$ and $T_2/2T_2$ are not self-dual, thus $T_1$ and $T_2$ are not self-dual.
By Theorem~\ref{thm n 2 mod 4}, this gives $T_1\cong T_2^*$.

\smallskip

(c)\, In consequence of Theorem~\ref{thm n odd} and Theorem~\ref{thm n 2 mod 4}, as far as the
determination of the isomorphism classes of $\ZZ_2$-forms of $S(2)_{\QQ_2}$ is concerned, the case
$n\equiv 0\pmod{4}$ remains open so far. The Loewy series of the $\FF_2\mathfrak{S}_n$-module
$S(2)_{\FF_2}$ in Proposition~\ref{prop S in char 2}(d) already indicates that the lattice of 
$\ZZ_2\mathfrak{S}_n$-sublattices of $S(2)_{\ZZ_2}$ of full rank will have a much more complicated structure in this case.

As well, for $k\geq 3$, we are currently neither able to determine representatives
of the isomorphism classes of $\ZZ_2$-forms of $S(k)_{\QQ_2}$ nor their number. 
However, based on computer calculations with MAGMA \cite{MAGMA}, we would like to state the following conjecture.
Using the algorithms developed in \cite{Hofmann2016b}, we have checked part (a)  of the conjecture for $n\leq 52$,
and part (b) for $n\leq 23$.
\end{rem}

\begin{conj}\label{conj}
Let $n \in \NN$,  and let $n \geq 5$. For $k\in\{1,\ldots,n-2\}$, denote by $h_2(k)$ the number of isomorphism classes of $\ZZ_2$-forms
of $S(k)_{\QQ_2}$.

\smallskip

(a)\, If $n \equiv 0 \pmod{4}$ and $k\in\{2,n-3\}$, then $h_2(k) = 3\nu_2(n) + 1$.

\smallskip

(b)\, For $k\in\{3,n-4\}$, the following holds:

\smallskip 

\quad {\rm (i)} If $n \geq 7$ is odd, then $h_2(k) = 3$.

\quad {\rm (ii)} If $n \equiv 2\pmod{4}$, then $h_2(k) = 8$.

\quad {\rm (iii)} If $n \equiv 0 \pmod{4}$, then $h_2(k) = 9\nu_2(n) + 1$.

\end{conj}

%%%%%%%%%%%%%%%%%%%%%%%%%%%%%%%%%%%%%%%%%%%%%%%%%%%%%%%%%

\section{Proofs of the main results}\label{sec proofs}

This section is now devoted to the proofs of Theorem~\ref{thm intro} and Corollary~\ref{cor intro}, and to establishing a
precise version of Theorem~\ref{thm intro expl}.
To do so, we shall apply Theorems~\ref{thm  forms bijection p}, \ref{thm n odd} and \ref{thm n 2 mod 4}, and
invoke the results of Plesken~\cite[Satz (I.9), Korollar (I.10)]{Plesken1974}, \cite[Theorem 5.1]{Plesken1977}
and Craig \cite{Craig1976}, who determined the isomorphism classes of
$\ZZ$-forms of the natural Specht $\QQ\mathfrak{S}$-module $S(1)_{\QQ}=S^{(n-1,1)}_{\QQ}$. In fact, 
Craig works with the Specht module $S(1)_\QQ\otimes\sgn_\QQ\cong S(n-2)_{\QQ}=S^{(2,1^{n-2})}_{\QQ}$ rather than $S(1)_{\QQ}$, and Plesken starts out with the dual lattice $S(1)_\ZZ^*$. 
 Translated into our terminology, the result reads as follows:
 
 \begin{thm}\label{thm craig}
 Let $n\in\NN$ be such that $n\geq 3$, and let $M:=S(1)_{\ZZ}$. For every divisor $d\in \NN$ of $n$, let
 $$M_d:={}_\ZZ\langle \sum_{i=2}^n b(i)\rangle +dM\subseteq M\,.$$
 Then  
 \smallskip
 
 {\rm (a)}\, $\{M_d: d\mid n\}$ is a set of representatives of the isomorphism classes of $\ZZ$-forms of $S(1)_{\QQ}$;
 
 \smallskip
 
 {\rm (b)}\, for every $d\mid n$, one has $[M:M_d]=d^{n-2}$;

 \smallskip
 
 {\rm (c)}\, if $p\leq n$ is a prime number, then $\{(M_{p^i})_{\ZZ_p}: 0\leq i\leq \nu_p(n)\}$ is a set of representatives of the
 isomorphism classes of $\ZZ_p$-forms of $S(1)_{\QQ_p}$.
 \end{thm}

\begin{proof}
(a)\, It is easily checked that each $M_d$, for $d\mid n$, is a $\ZZ G$-sublattice of $M$ with $\ZZ$-basis
$\{\sum_{i=2}^n b(i), db(2),\ldots, d b(n-1)\}$. Thus each $M_d$ is a $\ZZ$-form of $S(1)_{\QQ}$, by Remark~\ref{rem forms all in one}.
Moreover, if $d\mid n$ and $a\in\ZZ\smallsetminus\{1,-1\}$, then $M_d$ is not contained in $a M$, since
then $\sum_{i=2}^nb(i)\notin aM$. Thus, by \cite[Proposition 2.3]{Plesken1977}, the $\ZZ \mathfrak{S}_n$-lattices
in $\{M_d: d\mid n\}$ are pairwise non-isomorphic. But, by \cite[Korollar (I.10)]{Plesken1974}, the number of
isomorphism classes of $\ZZ$-forms of $S(1)_{\QQ}$ equals the number of positive divisors of of $n$, whence (a).

\smallskip

Part (b) follows immediately from Proposition~\ref{prop det} using the $\ZZ$-basis of $M_d$ just mentioned.

\smallskip

To prove (c), let $d\mid n$, and let $p$ be a prime number such that $d=p^i q$, for some $q,i\in\NN$ with $p\nmid q$.
Then $dM\subseteq p^iM$, hence also $M_d\subseteq M_{p^i}$ and $(M_d)_{\ZZ_p}\subseteq (M_{p^i})_{\ZZ_p}\subseteq M_{\ZZ_p}$. As well, 
$(M_{\ZZ_p}:(M_d)_{\ZZ_p})=d^{n-2}\cdot \ZZ_p=p^i\cdot\ZZ_p=(M_{\ZZ_p}:(M_{p^i})_{\ZZ_p})$. 
Thus Proposition~\ref{prop det} gives 
%$[M_{\ZZ_p}:(M_d)_{\ZZ_p}]=[M_{\ZZ_p}:M_{p^i}_{\ZZ_p}]$
$(M_d)_{\ZZ_p}=(M_{p^i})_{\ZZ_p}$.
On the other hand, in consequence of (a) and Proposition~\ref{prop  local}, we obtain
that the $\ZZ_p \mathfrak{S}_n$-lattices in $\{(M_{p^i})_{\ZZ_p}: 0\leq i\leq \nu_p(n)\}$
have to be pairwise non-isomorphic, since $(M_{p^i})_{\ZZ_q}=M_{\ZZ_q}$, for
all $i\in\{0,\ldots,\nu_p(n)\}$ and every prime number $q\neq p$.
Now Proposition~\ref{prop rep iso p} implies the assertion in (c).
\end{proof}

\begin{rem}\label{rem transpose bij}
By Theorem~\ref{thm Specht properties} and Proposition~\ref{prop dual sign}, we know that $S(k)_{\QQ}\cong S(n-k+1)_\QQ\otimes\sgn_\QQ$ as well as $S(k)_{\QQ_p}\cong S(n-k+1)_{\QQ_p}\otimes \sgn_{\QQ_p}$, for every
$k\in\{1,\ldots,n-2\}$ and every prime number $p$. These isomorphisms entail a
bijection between the set of isomorphism classes of $\ZZ$-forms of $S(k)_\QQ$ and the set of isomorphism classes of $\ZZ$-forms
of $S(n-k+1)_\QQ$, as well as a bijection between the set of isomorphism classes of $\ZZ_p$-forms of $S(k)_{\QQ_p}$ and the set of isomorphism classes of $\ZZ_p$-forms
of $S(n-k+1)_{\QQ_p}$. 

Hence, together with Theorem~\ref{thm craig} we now get:
\end{rem}

\begin{proof}[\textsl{\textbf{Proof of Theorem~\ref{thm intro}}}]
Part~(a) follows from Theorem~\ref{thm forms bijection p} and Theorem~\ref{thm craig}(c).
Part~(b) follows from Theorem~\ref{thm n odd} and Theorem~\ref{thm n 2 mod 4}.
\end{proof}

\begin{proof}[\textsl{\textbf{Proof of Corollary~\ref{cor intro}}}]
If $n\in\NN$ has prime factor decomposition $n=p_1^{a_1}\cdots p_r^{a_r}$ and if $d(n)$ denotes
the number of divisors of $n$ in $\NN$, then $d(n)=\prod_{i=1}^r(a_i+1)$.
Thus the corollary follows from Theorem~\ref{thm intro} together with Corollary~\ref{cor number local}.
\end{proof}

%%make notation in Remark overZ a bit more precise.
We conclude by proving  Theorem~\ref{thm expl}. Here we shall use the following notation:
for every prime number $p\geq 3$, we set
$$X_p:=\left\{\bigwedge^2(M_{p^i}): 0\leq i\leq \nu_p(n)\right\}\,.$$ 
Furthermore, we set
$$X_2:=\begin{cases}
\{S(2)_{\ZZ},\hat{S}_1,\hat{S}_2\}&\text{ if } n\equiv 1\pmod{2}\\
\{S(2)_{\ZZ},\hat{T}_1,\hat{T}_2,\hat{T}_3\}&\text{ if } n\equiv 2\pmod{4}\,;
\end{cases}$$
here $\hat{S}_1,\hat{S}_2,\hat{T}_1,\hat{T}_2,\hat{T}_3$ are the $\ZZ$-forms of $S(2)_{\QQ}$ in Remark~\ref{rem overZ}.
With this notation, we have

\begin{thm}\label{thm expl}
Let $n\geq 5$ be such that $n \not\equiv 0 \pmod{4}$, and let $L$ be a $\ZZ$-form of $S(2)_{\QQ}$.
Let $\{p_1,\ldots,p_g\}$ be the set of prime divisors of $n$.
Then for each $j\in\{1,\ldots,g\}$, there exists a unique $N_{p_j} \in X_{p_j}$ such that
$$L\cong \bigcap_{j=1}^g N_{p_j}\,.$$
\end{thm}

\begin{proof}
This follows from Theorems~\ref{thm forms bijection p}, \ref{thm n odd} and~\ref{thm n 2 mod 4}, together with Proposition~\ref{prop rep iso}.
\end{proof}

By Remark~\ref{rem transpose bij}, Theorem~\ref{thm expl} also provides representatives of the
isomorphism classes of $\ZZ$-forms of $S(n-3)_\QQ$. 

%%%%%%%%%%%%%%%%%%%%%%%%%%%%%%

\begin{appendix}

\section{Dual Specht lattices}\label{sec dual}

Let $n\in\NN$, and let $\lambda$ be a partition of $n$.
The aim of this appendix is to construct an explicit $\ZZ\mathfrak{S}_n$-monomomorphism 
$$\varphi^\lambda_\ZZ:(S^\lambda_\ZZ)^*\to S^\lambda_\ZZ\,.$$
Since $S^\lambda_{\QQ}$ is self-dual, it is clear
that such a monomorphism has to exist, by Remark~\ref{rem forms all in one}. The particular monomorphism $\varphi^\lambda_\ZZ$ 
in Theorem~\ref{thm dual}, which has been essential in the proof of Proposition~\ref{prop T dual}, can be found in unpublished work of Wildon \cite{Wildon2002}. It can be obtained by a close inspection of the arguments in the proof of \cite[Theorem 6.7]{James1978}.
In the following we aim to work out the details of the construction of $\varphi^\lambda_\ZZ$ following the lines of \cite{Wildon2002}.

We start by summarizing some properties of bilinear forms of permutation modules that we shall need
throughout this appendix.

\begin{noth}\label{noth bilinear}{\bf Permutation modules and bilinear forms.}
Let $G$ be any finite group, let $R$ be a principal ideal domain and $K$ its field of fractions.
Suppose that $L$ is a permutation $RG$-lattice with permutation basis $\{\omega_1,\ldots,\omega_k\}$.

\smallskip

(a)\, Again, denote by $\{\omega_1^*,\ldots,\omega_k^*\}$ the dual $R$-basis of $L^*$. Moreover, consider the canonical
$R$-bilinear form
\begin{equation}\label{eqn bilinear}
\langle\cdot\mid\cdot\rangle: L\times L\to R\,, \,(\omega_i,\omega_j)\mapsto \delta_{ij}\,,
\end{equation}
which is symmetric, non-degenerate and $G$-invariant. 
Whenever $N$ is an $RG$-sublattice of $L$, we have the $RG$-sublattice $N^\circ:=\{f\in L^*: f(x)=0,\text{ for all } x\in N\}$ of
$L^*$ as well as the $RG$-sublattice $N^\perp:=\{x\in L:\langle x\mid y\rangle=0,\text{ for all } y\in N\}$ of $L^*$. 
The $RG$-isomorphism
$$\psi:L\to L^*, \, x\mapsto (y\mapsto \langle x\mid y\rangle)$$
maps $N^\perp$ to $N^\circ$.

\smallskip

(b)\, Now suppose that $(\omega_1,\ldots,\omega_k)$ is a transitive permutation basis of $L$. Then $L$ is clearly a cyclic $RG$-lattice,
generated by any basis element $\omega_i$. Since $\psi(\omega_i)=\omega_i^*$, for $i\in\{1,\ldots,k\}$, also
$L^*$ is a cyclic $RG$-lattice, generated by any $\omega_i^*$. Note that if $i,j\in\{1,\ldots,k\}$ and $g\in G$ are such that
$g\omega_i=\omega_j$, then also $g \omega_i^*=\omega_j^*$.

\smallskip

(c)\, Lastly, note that the $R$-bilinear form (\ref{eqn bilinear}) naturally extends to a symmetric, non-degenerate $G$-invariant $K$-bilinear form
on $KL$, which we denote by $\langle\cdot\mid\cdot\rangle_K$. Thus, if $N$ is an $RG$-sublattice of $L$, then we have
the $RG$-sublattice $N^\perp$ of $L$, on the one hand, and the $KG$-submodule $(KN)^\perp$ of $KL$, on the other hand.
The next lemma explains how these modules are related.
\end{noth}

\begin{lemma}\label{lemma perp}
Let $G$ be a finite group, and let $R$ be a principal ideal domain with field of fractions $K$.
Let further $L$ be a permutation $RG$-lattice, and let $N$ be an $RG$-sublattice of $L$.
With the notation of \ref{noth bilinear}, one has
$(KN)^\perp=KN^\perp$; in particular,
$$\rk_R(L)=\rk_R(N)+\rk_R(N^\perp)\,.$$
\end{lemma}

\begin{proof}
Let $x\in N^\perp$, and let $y\in KN$. Then $\alpha y\in N$, for some $\alpha\in R$ with $\alpha\neq 0$. For $\beta\in K$ we get
$$\langle \beta x\mid y\rangle_K = \beta\alpha^{-1}\langle x\mid \alpha y\rangle_K=\beta\alpha^{-1}\langle x\mid \alpha y\rangle=0\,,$$
which shows that $\beta x\in (KN)^\perp$. 

Conversely, let $x\in (KN)^\perp$, and let $y\in N$. Moreover, let $\alpha\in R$ be such that $\alpha\neq 0$ and $\alpha x\in L$.
Then
$$\langle \alpha x\mid y\rangle=\langle \alpha x\mid y\rangle_K=\alpha\langle x\mid y\rangle_K=0\,,$$
thus $\alpha x\in N^\perp$ and $x=\alpha^{-1}\alpha x\in KN^\perp$.

This proves $(KN)^\perp=KN^\perp$, and we deduce
\begin{align*}
\rk_R(L)&=\dim_K(KL)=\dim_K(KN)+\dim_K((KN)^\perp)\\
&=\dim_K(KN)+\dim_K(KN^\perp)=\rk_R(N)+\rk_R(N^\perp)\,.
\end{align*}
\end{proof}

For the remainder of this appendix we consider $M:=M^\lambda_\ZZ$ and
$S:=S^\lambda_\ZZ$.
As in \ref{noth bilinear}, we denote by $\langle \cdot\mid \cdot\rangle$ the non-degenerate symmetric $\mathfrak{S}_n$-invariant $\ZZ$-bilinear form on $M$ with respect to the standard basis $\mathcal{T}(\lambda)$.
As an immediate consequence of the $\mathfrak{S}_n$-invariance of this bilinear form, one has the following well-known fact:

\begin{lemma}\label{lemma sign}
Let $t$ be any $\lambda$-tableau, and let $x,y\in M$. Let $\kappa_t$ and $\rho_t$
be as in \ref{noth tableaux}(b). Then
$$\langle \kappa_tx|y\rangle=\langle x|\kappa_ty\rangle \quad\text{and} \quad \langle \rho_tx|y\rangle=\langle x|\rho_ty\rangle\,.$$
\end{lemma}

\begin{proof}
See the proof of \cite[Lemma 2.4.1]{Sagan2001}.
\end{proof}

The next result has already been stated in \cite[Lemma 2]{Wildon2002}, and is an immediate generalization of the 
classical Submodule Theorem \cite{James1976}; we include a short proof for completeness.

\begin{thm}[Integral Version of the Submodule Theorem]\label{thm submodule}
Let $U$ be a $\ZZ\mathfrak{S}_n$-sublattice of $M$. If $U\not\subseteq S^\perp$, then there is some $m\in\NN$ such that
$mS\subseteq U$.
\end{thm}

\begin{proof}
Recall that $e_t\in S$ denotes the $\lambda$-polytabloid labelled by $t$.
Let $u\in U$. Then, by \cite[Corollary 2.4.3]{Sagan2001}, there is some $q\in\QQ$ such that $\kappa_t\cdot u=q e_t$. Thus
$r\kappa_tu=me_t$, for some $r,m\in \ZZ$ such that $m\geq 0$ and $r\neq 0$. If $m\neq 0$, then we have
$mS={}_{\ZZ\mathfrak{S}_n}\langle me_t\rangle\subseteq U$. 

So suppose now that $\kappa_tu= 0$, for all $u\in U$ and every $\lambda$-tableau $t$, so that
$$0=\langle \kappa_tu|\{t\}\rangle=\langle u|\kappa_t\{t\}\rangle =\langle u|e_t\rangle\,,$$  
by Lemma~\ref{lemma sign}. This shows that $U\subseteq S^\perp$ in this case.
\end{proof}

Throughout this section, we now fix a $\lambda$-tableau $t$. The corresponding $\lambda'$-tableau obtained by transposing $t$ will again be denoted
by $t'$ as before. Moreover, by $\sgn$ we denote the $\ZZ\mathfrak{S}_n$-lattice of rank one on which
$\sigma\in\mathfrak{S}_n$ acts by multiplication with $\sgn(\sigma)$. The respective $\QQ\mathfrak{S}_n$-module
will be denoted by $\sgn_\QQ$.

\begin{lemma}\label{lemma f}
There is a $\ZZ\mathfrak{S}_n$-epimorphism
$$f_t:M\to S^{\lambda'}_\ZZ\otimes \sgn$$
such that $f_t(\{t\})=e_{t'}\otimes 1$ and $\ker(f_t)=S^\perp$.
\end{lemma}

\begin{proof}
As in the proof of \cite[Theorem 6.7]{James1978}, we set
$$f_t(\{\pi t\}):=\sgn(\pi) (e_{\pi t'}\otimes 1)\,,$$
for $\pi\in\mathfrak{S}_n$, 
and then extend this map $\ZZ$-linearly. 
Note that $\sgn(\pi) (e_{\pi t'}\otimes 1)=\pi\cdot (e_{t'}\otimes 1)=\pi\rho_t(\{t'\}\otimes 1)$, for $\pi\in\mathfrak{S}_n$.
We first show that $f_t$ is well defined. To this end, let $\pi_1,\pi_2\in\mathfrak{S}_n$ be such that
$\{\pi_1 t\}=\{\pi_2 t\}$, that is, $\pi_1^{-1}\pi_2\in R_t=C_{t'}$. Hence
$$\pi_1^{-1}\pi_2(e_{t'}\otimes 1)=(\pi_1^{-1}\pi_2 e_{t'})\otimes \sgn(\pi_1^{-1}\pi_2)=(\sgn(\pi_1^{-1}\pi_2) e_{t'})\otimes\sgn(\pi_1^{-1}\pi_2)=e_{t'}\otimes 1\,,$$
and
\begin{align*}
f_t(\pi_2 t\})&=f_t(\{\pi_1\pi_1^{-1}\pi_2 t\})=(\pi_1\pi_1^{-1}\pi_2)\cdot(e_{t'}\otimes 1)\\
&=\pi_1(\pi_1^{-1}\pi_2 (e_{t'}\otimes 1))=\pi_1(e_{t'}\otimes 1)=f_t(\{\pi_1 t\})\,.
\end{align*}
This implies that $f_t$ is well defined. By definition, $f_t$ is also a $\ZZ\mathfrak{S}_n$-homomorphism, and we have
$f_t(\{t\})=e_{t'}\otimes 1$. 
Since $S^{\lambda'}_\ZZ\otimes \sgn={}_{\ZZ}\langle e_{\pi t'}\otimes 1: \pi\in\mathfrak{S}_n\rangle$, we deduce that
$f_t$ is, in fact, a $\ZZ\mathfrak{S}_n$-epimorphism. It remains to determine the kernel of $f_t$. Note that 
$f_t(e_t)= (\rho_{t'}\kappa_{t'}\{t'\})\otimes 1\in M^{\lambda'}_\ZZ\otimes \sgn$. Moreover, Lemma~\ref{lemma sign} gives
$$\langle \rho_{t'}\kappa_{t'} \{t'\}\mid\{t'\}\rangle=\langle \kappa_{t'}\{t'\}\mid\rho_{t'}\{t'\}\rangle =\langle e_{t'}\mid |R_{t'}| \{t'\}\rangle=|R_{t'}|\,.$$
So, writing $f_t(e_t)$ as a $\ZZ$-linear combination of the standard $\ZZ$-basis of $M^{\lambda'}_\ZZ\otimes \sgn$, 
the basis element $\{t'\}\otimes 1$ occurs with non-zero coefficient $|R_{t'}|$; in particular, $f_t(e_t)\neq 0$ as well as $f_t(m e_t)\neq 0$, for
all $m\in \NN$. Therefore, $mS\not\subseteq  \ker(f_t)$, for all $m\in \NN$, implying $\ker(f_t)\subseteq S^\perp$, by Theorem~\ref{thm submodule}.

To show that we actually have equality, note first that $\rk_\ZZ(\ker(f_t))= \rk_\ZZ(M)-\rk_\ZZ(S^{\lambda'}_\ZZ\otimes \sgn)=\rk_\ZZ(M)-\rk_\ZZ(S)$. On the other hand, Lemma~\ref{lemma perp} also gives
$\rk_\ZZ(S^\perp)=\rk_\ZZ(M)-\rk_\ZZ(S)$. Hence, $\rk_\ZZ(\ker(f_t))=\rk_\ZZ(S^\perp)$, so that $S^\perp/\ker(f_t)$ is a finite 
abelian group of order $r\in\NN$, say. If $r>1$ then there would be some $x\in S^\perp$ such that $f_t(x)\neq 0$, but
$0=f_t(rx)=rf_t(x)$, a contradiction. This shows that, indeed, $S^\perp=\ker(f_t)$.
\end{proof}

\begin{rem}\label{rem basis S}
Let $k:=\rk_\ZZ(S)$.
In the following it will be useful to replace the standard $\ZZ$-basis of $S=S^\lambda_\ZZ$ by a slightly modified one. In general, 
writing a standard $\lambda$-polytabloid as a $\ZZ$-linear combination of $\lambda$-tabloids, several standard $\lambda$-tabloids will occur with non-zero coefficients. However, the proof of \cite[Corollary 8.12]{James1978} shows that $S$ admits a $\ZZ$-basis $\{b_1,\ldots,b_k\}$
such that, for each $i\in\{1,\ldots,k\}$, there is a unique standard $\lambda$-tabloid $\{t_i\}$ such that $\langle b_i\mid \{t_i\}\rangle=1$
and $\langle b_i\mid \{s\}\rangle =0$ for every standard $\lambda$-tabloid $\{s\}\neq \{t_i\}$. Moreover, $\{t_i\}\neq \{t_j\}$, for
$i\neq j$.
Denote the dual basis of $S^*$ by $\{b_1^*,\ldots,b_k^*\}$.

For each $i\in\{1,\ldots,k\}$, consider the element $\{t_i\}^*\in M^*$ of the $\ZZ$-basis of $M^*$
that is dual to the tabloid basis of $M$. Then $\{t_i\}^*(b_j)=\delta_{ij}$, for $i,j\in\{1,\ldots,k\}$, so that $b_i^*$ is precisely the restriction
of
$\{t_i\}^*$ to $S$.
\end{rem}

\begin{lemma}\label{lemma g}
The map
$$g:M\to S^*,\, m\mapsto (x\mapsto \langle x\mid m\rangle)$$
is a $\ZZ\mathfrak{S}_n$-epimorphism with kernel $S^\perp$.
\end{lemma}

\begin{proof}
The map $g$ is clearly $\ZZ$-linear with kernel $S^\perp$, and due to the $\mathfrak{S}_n$-invariance of $\langle \cdot \mid \cdot \rangle$ the map $g$ is 
also a $\ZZ\mathfrak{S}_n$-homomorphism. Remark~\ref{rem basis S} shows, moreover, that $g$ is surjective.
\end{proof}

\begin{lemma}\label{lemma h}
There is a $\ZZ\mathfrak{S}_n$-epimorphism
$$h_t:M^{\lambda'}_\ZZ\otimes \sgn\to S$$
such that $h_t(\{t'\}\otimes 1)=e_t$. Moreover, the restriction of $h_t$ to $S^{\lambda'}_\ZZ\otimes \sgn$ is injective.
\end{lemma}

\begin{proof}
By Lemma~\ref{lemma f}, we have the $\ZZ\mathfrak{S}_n$-epimorphism
$f_{t'}:M^{\lambda'}_\ZZ\to S^\lambda_\ZZ\otimes \sgn$, inducing a $\ZZ\mathfrak{S}_n$-epimorphism
$f_{t'}\otimes \id:M^{\lambda'}_\ZZ\otimes \sgn \to S^\lambda_\ZZ\otimes \sgn\otimes \sgn$.
Finally, $S^\lambda_\ZZ\otimes \sgn\otimes \sgn\cong S^\lambda_\ZZ$, via the $\ZZ$-linear map
sending $e_s\otimes 1\otimes 1$ to $e_s$, for every standard $\lambda$-tableau $s$. In total, this yields a
$\ZZ\mathfrak{S}_n$-epimorphism $h_t:M^{\lambda'}_\ZZ\otimes \sgn\to S$ such that $h_t(\{t'\}\otimes 1)=e_t$, as desired.

\smallskip

It remains to show that $h_t$ restricts to an injective $\ZZ\mathfrak{S}_n$-homomorphism $S^{\lambda'}_\ZZ\otimes\sgn\to S^\lambda_\ZZ$. Denote this restriction by $\phi$; for ease of notation, we suppress the index $t$ for the time being.
Then we have
\begin{align*}
\phi(e_{t'}\otimes 1)&=h_t(e_{t'}\otimes 1)=h_t((\sum_{\sigma\in C_{t'}}\sgn(\sigma)\{\sigma t'\})\otimes 1)=h_t(\sum_{\sigma\in C_{t'}}\sigma\cdot (\{t'\}\otimes 1))\\
&=\sum_{\sigma\in C_{t'}}\sigma\cdot h_t(\{t'\}\otimes 1)=\sum_{\sigma\in C_{t'}}\sigma e_t\\
&=\sum_{\sigma\in C_{t'}}\sigma \sum_{\pi\in C_t}\sgn(\pi) \pi\{t\}=\sum_{\sigma\in R_t}\sum_{\pi\in C_t}\sgn(\pi)\{\sigma\pi t\}\in S_\ZZ^\lambda\subseteq M^\lambda_\ZZ\,.
\end{align*}
The coefficient at the basis element $\{t\}$ of $M^\lambda_\ZZ$ in $h_t(e_{t'}\otimes 1)$ equals
$$\mathop{\mathop{\mathop{\sum_{\sigma\in R_t}}}_{\pi\in C_t}}_{\sigma\pi\in R_t} \sgn(\pi)\,.$$
If $\sigma\in R_t$ and $\pi\in C_t$ are such that $\sigma\pi\in R_t$ then $\pi\in R_t\cap C_t=\{1\}$. Hence the coefficient
at $\{t\}$ in $h_t(e_{t'}\otimes 1)$ is $|R_t|\neq 0$; in particular, $\phi\neq 0$.

Now $\phi$ induces a $\QQ\mathfrak{S}_n$-homomorphism
$$\phi_\QQ:S^{\lambda'}_\QQ\otimes \sgn_\QQ\to S^\lambda_\QQ$$
such that $\phi_\QQ(e_{t'}\otimes 1)\neq 0$. Consequently, $\phi_\QQ$ is an isomorphism, since 
$S^{\lambda'}_\QQ\otimes \sgn_\QQ$ and $S^\lambda_\QQ$ are simple $\QQ\mathfrak{S}_n$-modules.
Since $S^{\lambda'}_\ZZ\otimes_\ZZ \sgn_\ZZ$ is a $\ZZ$-form of $S^{\lambda'}_\QQ\otimes \sgn_\QQ$ and 
$S^\lambda_\ZZ$ is a $\ZZ$-form of $S^\lambda_\QQ$, Proposition~\ref{prop extend} implies that $\phi$ has to be an injective $\ZZ\mathfrak{S}_n$-homomorphism, and the proof
of the lemma is complete.
\end{proof}

%%%%%%%%%%%%%%%%%%%%%%%%%%%%%%%%%%%%%%%%%%%%%%%%%%%%%%%%%%%%%%

As an immediate consequence of Lemma~\ref{lemma f} and Lemma~\ref{lemma g} one has the following, which is \cite[Theorem 8.15]{James1978} in the case where $R$ is a field and \cite[Proposition 4.8]{Fayers2003} in the case where $R=\ZZ$:

\begin{prop}\label{prop dual sign}
For every $\lambda\in\mathcal{P}(n)$ and every principal ideal domain $R$, one has an isomorphism
$$S^{\lambda'}_R\otimes\sgn_R\cong (S^\lambda_R)^*$$
of $RG$-lattices.
\end{prop}

\begin{proof}
For $\lambda\in\mathcal{P}(n)$, Lemma~\ref{lemma f} and Lemma~\ref{lemma g} yield a $\ZZ\mathfrak{S}_n$-isomorphism
$S^{\lambda'}_\ZZ\otimes\sgn\cong (S^\lambda_\ZZ)^*$. Hence the assertion of the proposition follows from \ref{noth change rings}(a),(d).
\end{proof}

%%%%%%%%%%%%%%%%%%%%%%%%%%%%%%%%%%%%%%%%%%%%%%%%%%%%%%%%%%%%%%

We are now in the position to prove Theorem~\ref{thm dual} below. Recall from Remark~\ref{rem basis S} that, as $s$ varies over the set of standard $\lambda$-tableaux, the $\ZZ$-linear maps
$$\{s\}^*_{|S^\lambda_\ZZ}:S^\lambda_\ZZ\to \ZZ\,, x\mapsto \langle x\mid \{s\}\rangle$$
vary over a $\ZZ$-basis of $(S^\lambda_\ZZ)^*$.

\begin{thm}[Wildon \cite{Wildon2002}]\label{thm dual}
Let $\lambda$ be a partition of $n$. Then there is a $\ZZ\mathfrak{S}_n$-monomorphism
$$\varphi:(S^\lambda_\ZZ)^*\to S^\lambda_\ZZ$$
such that
$$\varphi(\{s\}^*_{|S^\lambda_\ZZ})=\sum_{\sigma\in R_s} e_{\sigma s}\,,$$
for every $\lambda$-tableau $s$. 
\end{thm}

\begin{proof}
Let $t$ be a fixed $\lambda$-tableau, as before. By Lemma~\ref{lemma f} and Lemma~\ref{lemma g}, we have 
$\ZZ\mathfrak{S}_n$-isomorphisms
$$\bar{f}_t:M^\lambda_\ZZ/(S^\lambda_\ZZ)^\perp \to S^{\lambda'}_\ZZ\otimes \sgn\,, m+(S^\lambda_\ZZ)^\perp\mapsto f_t(m)$$
and 
$$\bar{g}:M^\lambda_\ZZ/(S^\lambda_\ZZ)^\perp \to (S^\lambda_\ZZ)^*\,, m+(S^\lambda_\ZZ)^\perp\mapsto g(m)\,.$$
Furthermore, by Lemma~\ref{lemma h}, we have the $\ZZ\mathfrak{S}_n$-monomorphism
$$(h_t)_{|S^{\lambda'}_\ZZ\otimes \sgn}\to S^\lambda_\ZZ\,.$$
We set $\varphi:=(h_t)_{|S^{\lambda'}_\ZZ\otimes \sgn}\circ \bar{f}_t\circ\bar{g}^{-1}$, and show that $\varphi$ has the desired property (independently of $t$).
First note that
\begin{align*}
\varphi(\{t\}^*_{|S^\lambda_\ZZ})&=h_t(\bar{f}_t(\{t\}+(S^\lambda_\ZZ)^\perp))=h_t(f_t(\{t\}))=h_t(e_{t'}\otimes 1)\\
&=\sum_{\sigma\in C_{t'}}\sigma e_t=\sum_{\sigma\in R_t} \sigma e_t=\sum_{\sigma\in R_t} e_{\sigma t}\,.
\end{align*}
Now let $u$ be any $\lambda$-tableau. Then there is some $\pi\in \mathfrak{S}_n$ such that $\pi t=u$, thus
also $\pi\{t\}=\{u\}$, and then $\pi\{t\}^*=\{u\}^*$, by \ref{noth bilinear}(b) Furthermore, $R_{u}=R_{\pi t}=\pi R_t \pi^{-1}$.
Consequently,
\begin{align*}
\varphi(\{u\}^*_{|S^\lambda_\ZZ})&=\varphi(\pi\{t\}^*_{|S^\lambda_\ZZ})=\pi\cdot \varphi(\{t\}^*_{|S^\lambda_\ZZ})=\pi\cdot \sum_{\sigma\in R_t} e_{\sigma t}\\
&=\sum_{\sigma\in R_t} e_{\pi\sigma t}=\sum_{\sigma\in R_t} e_{\pi\sigma \pi^{-1}\pi t} =\sum_{\tau\in R_u} e_{\tau u}\,.
\end{align*}
Now the assertion of the theorem follows.
\end{proof}

%%%%%%%%%%%%%%%%%%%%%%%%%%%%%%%%%%%%%%%%%%%%%%%%%%%%%%%%%%%%%%%

\begin{cor}\label{cor phi ZZ_p}
Let $\lambda\in\mathcal{P}(n)$, and let $R$ be a principal ideal domain that is a flat right $\ZZ$-module.
Then there is an $R\mathfrak{S}_n$-monomorphism
 $$\varphi^\lambda_{R}:(S^\lambda_R)^*\to S^\lambda_R\,.$$
%such that
%$$\varphi^\lambda_R(\{s\}^*_{|S^\lambda_R})=\sum_{\sigma\in R_s} e_{\sigma s}\,,$$
%for every $\lambda$-tableau $s$. 
\end{cor}

\begin{proof}
By \ref{noth change rings}(d), we have an $R\mathfrak{S}_n$-isomorphism
$(S_R^\lambda)^*=(R\otimes_\ZZ S^\lambda_\ZZ)^*\cong R\otimes_\ZZ (S^\lambda_\ZZ)^*$. 
Since $R$ is a flat right $\ZZ$-module, the $\ZZ\mathfrak{S}_n$-monomorphism $\varphi^\lambda_\ZZ$
from Theorem~\ref{thm dual} induces an $R\mathfrak{S}_n$-monomorphism 
$\id_R\otimes\varphi_\ZZ^\lambda: R\otimes_\ZZ (S^\lambda_\ZZ)^*\to R\otimes_\ZZ S^\lambda_\ZZ=S_R^\lambda$.
\end{proof}

\begin{expl}\label{expl hook dual}
Suppose that $\lambda$ is a hook partition of $n$, that is, $\lambda=(n-k,1^k)$, for some $0\leq k\leq n-1$.

\smallskip

(a)\, Let $t$ be a standard
$\lambda$-tableau. When writing the standard polytabloid $e_t$ as a $\ZZ$-linear combination of $\lambda$-tabloids, $\{t\}$ is the unique standard $\lambda$-tabloid occurring with non-zero coefficient; in fact, the coefficient at $\{t\}$ is 1. In particular, in this case, we have
$$\{t\}^*_{|S^\lambda_\ZZ}=e_t^*\in (S^\lambda_\ZZ)^*\,,$$
for every standard $\lambda$-tableau $t$.

\smallskip

Recall, moreover, that $S^\lambda_\ZZ$ is a cyclic $\ZZ\mathfrak{S}_n$-lattice, generated by any (standard) $\lambda$-polytabloid $e_t$.
If $s$ is also a standard $\lambda$-tableau, then we have $\sigma\cdot t=s$, for some $\sigma\in\mathfrak{S}_n$, hence
also $\sigma\{t\}=\{s\}$. This in turn implies $\sigma\{t\}^*=\{s\}^*$, see \ref{noth bilinear}(b), thus $e_s^*=\sigma\cdot e_t^*$. This shows that $(S^\lambda_\ZZ)^*$ is a cyclic $\ZZ\mathfrak{S}_n$-module, generated by $e_t^*$, for any standard $\lambda$-polytabloid $e_t$. So, by Theorem~\ref{thm dual}, the element $\sum_{\sigma\in R_t}e_{\sigma t}$ then generates a $\ZZ\mathfrak{S}_n$-sublattice of $S^\lambda_\ZZ$ that is isomorphic to $(S^\lambda_\ZZ)^*$.

\smallskip

(b)\, Let $R$ be a principal ideal domain that is flat as right $\ZZ$-module.
For each $\lambda$-tableau $t$, we identify the element $1\otimes e_t\in S_R^\lambda$ simply with $e_t$ as in \ref{noth Young Specht},
and then the element $e_t^*\in (S_R^\lambda)^*$ with $1\otimes e_t^*\in R\otimes (S_\ZZ^\lambda)^*$ via
the $R\mathfrak{S}_n$-isomorphism in \ref{noth change rings}(d). With this notation, the $R\mathfrak{S}_n$-monomorphism
$\varphi_R^\lambda$ in Corollary~\ref{cor phi ZZ_p} again maps $e_t^*$ to  $\sum_{\sigma\in R_t}e_{\sigma t}$.

\end{expl}
\end{appendix}

\bibliography{hookSpechtLattices}

\providecommand{\bysame}{\leavevmode\hbox to3em{\hrulefill}\thinspace}
\providecommand{\MR}{\relax\ifhmode\unskip\space\fi MR }
\providecommand{\MRhref}[2]{%
  \href{http://www.ams.org/mathscinet-getitem?mr=#1}{#2}
}
\providecommand{\href}[2]{#2}
\begin{thebibliography}{10}

\bibitem{Adkins1992}
William~A. Adkins and Steven~H. Weintraub, \emph{Algebra}, Graduate Texts in
  Mathematics, vol. 136, Springer-Verlag, New York, 1992, An approach via
  module theory.

\bibitem{Aitken1939}
A.C. Aitken, \emph{Determinants and {M}atrices}, Oliver and Boyd, 1939.

\bibitem{MAGMA}
Wieb Bosma, John Cannon, and Catherine Playoust, \emph{The {M}agma algebra
  system. {I}. {T}he user language}, J. Symbolic Comput. \textbf{24} (1997),
  no.~3-4, 235--265.

\bibitem{Bourbaki1970}
Nicolas Bourbaki, \emph{{\'E}l\'ements de math\'ematique. {A}lg\`ebre.
  {C}hapitres 1 \`a 3}, Hermann, 1970.

\bibitem{Cohen1993}
Henri Cohen, \emph{A course in computational algebraic number theory}, Graduate
  Texts in Mathematics, vol. 138, Springer-Verlag, Berlin, 1993.

\bibitem{Craig1976}
Maurice Craig, \emph{A characterization of certain extreme forms}, Illinois J.
  Math. \textbf{20} (1976), no.~4, 706--717.

\bibitem{Curtis1962}
Charles~W. Curtis and Irving Reiner, \emph{Representation theory of finite
  groups and associative algebras}, Interscience Publishers, 1962.

\bibitem{Curtis1981}
\bysame, \emph{Methods of representation theory. {V}ol. {I}}, John Wiley \&
  Sons Inc., New York, 1981.

\bibitem{Erdmann2001}
Karin Erdmann, \emph{Young modules for symmetric groups}, J. Aust. Math. Soc.
  \textbf{71} (2001), no.~2, 201--210.

\bibitem{Fayers2003}
Matthew Fayers, \emph{On the structure of {S}pecht modules}, J. London Math.
  Soc. (2) \textbf{67} (2003), no.~1, 85--102.

\bibitem{Feit1998}
Walter Feit, \emph{Integral representations of {W}eyl groups rationally
  equivalent to the reflection representation}, J. Group Theory \textbf{1}
  (1998), no.~3, 213--218.

\bibitem{Feit2003}
\bysame, \emph{Some integral representations of complex reflection groups}, J.
  Algebra \textbf{260} (2003), no.~1, 138--153.

\bibitem{Fulton1997}
William Fulton, \emph{Young tableaux}, London Mathematical Society Student
  Texts, vol.~35, Cambridge University Press, 1997.

\bibitem{Grabmeier1985}
Johannes Grabmeier, \emph{Unzerlegbare {M}oduln mit trivialer {Y}oungquelle und
  {D}arstellungstheorie der {S}churalgebra}, Bayreuth. Math. Schr. (1985),
  no.~20, 9--152.

\bibitem{Hofmann2016b}
Tommy Hofmann, \emph{Integrality of representations of finite groups}, Ph.D.
  thesis, Technische Universtit\"at Kaiserslautern, 2016.

\bibitem{Hofmann2016}
\bysame, \emph{Zeta functions of lattices of the symmetric group}, Comm.
  Algebra \textbf{44} (2016), no.~5, 2243--2255.

\bibitem{James1976}
Gordon~D. James, \emph{The irreducible representations of the symmetric
  groups}, Bull. London Math. Soc. \textbf{8} (1976), no.~3, 229--232.

\bibitem{James1978}
\bysame, \emph{The representation theory of the symmetric groups}, Lecture
  Notes in Mathematics, vol. 682, Springer, 1978.

\bibitem{Muller2011}
J\"urgen M\"uller and Johannes Orlob, \emph{On the structure of the tensor
  square of the natural module of the symmetric group}, Algebra Colloq.
  \textbf{18} (2011), no.~4, 589--610.

\bibitem{Muller2007}
J\"urgen M\"uller and Ren\'e Zimmermann, \emph{Green vertices and sources of
  simple modules of the symmetric group labelled by hook partitions}, Arch.
  Math. (Basel) \textbf{89} (2007), no.~2, 97--108.

\bibitem{Murphy1980}
Gwendolen Murphy, \emph{On decomposability of some {S}pecht modules for
  symmetric groups}, J. Algebra \textbf{66} (1980), no.~1, 156--168.

\bibitem{Nagao1989}
Hirosi Nagao and Yukio Tsushima, \emph{Representations of finite groups},
  Academic Press, New York, 1989.

\bibitem{Plesken1974}
Wilhelm~G. Plesken, \emph{{B}eitr\"{a}ge zur {B}estimmung der endlichen
  irreduziblen {U}ntergruppen von {GL(n,Z)} und ihrer ganzzahligen
  {D}arstellungen}, Ph.D. thesis, RWTH Aachen, 1974.

\bibitem{Plesken1977}
\bysame, \emph{On absolutely irreducible representations of orders}, Number
  theory and algebra (Hans Zassenhaus, ed.), Academic Press, 1977,
  pp.~241--262.

\bibitem{Plesken1983}
\bysame, \emph{Group rings of finite groups over {$p$}-adic integers}, Lecture
  Notes in Mathematics, vol. 1026, Springer-Verlag, Berlin, 1983.

\bibitem{Reiner1970}
Irving Reiner, \emph{A survey of integral representation theory}, Bull. Amer.
  Math. Soc. \textbf{76} (1970), 159--227.

\bibitem{Rotman2002}
Joseph~J. Rotman, \emph{Advanced modern algebra}, Prentice Hall, Inc., Upper
  Saddle River, NJ, 2002.

\bibitem{Sagan2001}
Bruce~E. Sagan, \emph{The symmetric group}, second ed., Graduate Texts in
  Mathematics, vol. 203, Springer, 2001.

\bibitem{Wildon2002}
Mark Wildon, \emph{Some examples on duality}, unpublished manuscript (2003),
  available at
  \texttt{http://www.ma.rhul.ac.uk/\textasciitilde{}uvah099/Maths/duality.pdf}.

\end{thebibliography}
\bibliographystyle{amsplain}

\bigskip

{\sc S.D.: Department of Mathematics and Geography, KU Eichst\"att-Ingolstadt,
Ostenstr. 26, 
85072 Eichst\"att,
Germany}\\
{\sf susanne.danz@ku.de}

\medskip

{\sc T.H.: Department of Mathematics,
University of Kaiserslautern,\\
P.O. Box 3049,
67653 Kaiserslautern,
Germany}\\
{\sf thofmann@mathematik.uni-kl.de}

\end{document}